\documentclass[12pt]{article}

\usepackage{amsmath,amsthm,amsfonts,amssymb, color,xcolor,subcaption,graphicx} 

\newtheorem{theorem}{Theorem}[section]
\newtheorem{corollary}[theorem]{Corollary}

\newtheorem{assumption}[theorem]{Assumption}
\newtheorem{hypothesis}[theorem]{Hypothesis}
\newtheorem{proposition}[theorem]{Proposition}
\newtheorem{lemma}[theorem]{Lemma}
\newtheorem{definition}[theorem]{Definition}
\newtheorem{remark}[theorem]{Remark}

\def\cA{\mathcal{A}}
\def\cB{\mathcal{B}}

\def\cE{\mathcal{E}}
\def\cF{\mathcal{F}}
\def\cG{\mathcal{G}}

\def\cI{\mathcal{I}}
\def\cJ{\mathcal{J}}
\def\cL{\mathcal{L}}
\def\cM{\mathcal{M}}
\def\cP{\mathcal{P}}
\def\cR{\mathcal{R}}

\def\cS{\mathcal{S}}

\def\cV{\mathcal{V}}

\def\bE{\mathbb{E}}

\def\bN{\mathbb{N}}
\def\bP{\mathbb{P}}
\def\bR{\mathbb{R}}

\def\wP{\widetilde{\cP}}
\def\Pb{\widetilde{\cP}_b}

\def\e{\varepsilon}

\begin{document}

\title{Series expansions for SPDEs with \\symmetric $\alpha$-stable L\'evy noise}

\author{Raluca M. Balan\footnote{Corresponding author. University of Ottawa, Department of Mathematics and Statistics, 150 Louis Pasteur Private, Ottawa, Ontario, K1G 0P8, Canada. E-mail address: rbalan@uottawa.ca.} \footnote{Research supported by a grant from the Natural Sciences and Engineering Research Council of Canada.}
\and
Juan J. Jim\'enez\footnote{University of Ottawa, Department of Mathematics and Statistics, 150 Louis Pasteur Private, Ottawa, Ontario, K1G 0P8, Canada. E-mail address: jjime088@uottawa.ca.}
}

\date{September 20, 2024}
\maketitle

\begin{abstract}
\noindent In this article, we examine a stochastic partial differential equation (SPDE) driven 
by a symmetric $\alpha$-stable (S$\alpha$S) L\'evy noise, that is multiplied by a linear function
 $\sigma(u)=u$ of the solution. The solution is interpreted in the mild sense. For this models, in the case of the Gaussian noise, the solution has an explicit Wiener chaos expansion, and is studied using tools from Malliavin calculus. These tools cannot be used for an infinite-variance L\'evy noise. In this article, we provide sufficient conditions for the existence of a solution, and we give an explicit series expansion of this solution. To achieve this, we use the multiple stable integrals, which were developed in \cite{ST90,ST91}, and originate from the LePage series representation of the noise. To give a meaning to the stochastic integral which appears in the definition of solution, we embed the space-time L\'evy noise into a L\'evy basis, and use the stochastic integration theory \cite{BJ83,bichteler02} with respect to this object, as in other studies of SPDEs with heavy-tailed noise: \cite{chong17-JTP,chong17-SPA,CDH19}. As applications, we consider the heat and wave equations with linear multiplicative noise, also called the parabolic/hyperbolic Anderson models. 
\end{abstract}

\noindent {\em MSC 2020:} Primary 60H15; Secondary 60G60, 60G52

\vspace{1mm}

\noindent {\em Keywords:} stochastic partial differential equations, random fields, $\alpha$-stable random measure, L\'evy basis

\pagebreak

\tableofcontents

\section{Introduction}

In this article, we study the stochastic partial differential equation (SPDE):
\begin{equation}
\label{eq}
\cL u(t,x)=u(t,x) \dot{Z}(t,x), t \in [0,T], x \in \bR^d
\end{equation}
with constant initial condition  $1$, where $\cL$ is a second-order pseudo-differential operator and 
$\dot{Z}$ is the formal derivative associated with a {\em symmetric $\alpha$-stable (S$\alpha$S) random measure} $Z$ on $[0,T] \times \bR^d$, with stability index $\alpha \in (0,2)$ and control measure $m$ given by:
\begin{equation}
\label{def-m}
m(B)=C_{\alpha}^{-1}{\rm Leb}(B) \quad \mbox{where} \quad C_{\alpha}=\left(\int_{0}^{\infty}\frac{\sin x}{x^{\alpha}} dx\right)^{-1}.
\end{equation}
This means that  $Z=\{Z(B)\}_{B \in \cB_b}$ is a collection of random variables defined on a complete probability space $(\Omega,\cF,\bP)$ and indexed by the class $\cB_b$ of bounded Borel sets in $[0,T]\times \bR^d$, which has the following properties:\\
{\em (i)} for any disjoint sets $B_1,\ldots,B_n$ in $\cB_b$, $Z(B_1),\ldots,Z(B_n)$ are independent;\\
{\em (ii)} for any disjoint sets $(B_i)_{i\geq 1}$ in $\cB_b$ with $\bigcup_{i\geq 1}B_i \in \cB_0$,
$$Z\Big(\bigcup_{i\geq 1}B_i\Big)=\sum_{i\geq 1}Z(B_i) \quad \mbox{a.s.};$$ 
{\em (iii)} for any $B \in \cB_b$, $Z(B)$ has an $\alpha$-stable $S_{\alpha}(m(B)^{1/\alpha},0,0)$ distribution, i.e.
\[
\bE[e^{iuZ(B)}]=e^{-m(B) |u|^{\alpha}} \quad \mbox{for any $u \in \bR$}.
\]
We refer the reader to Chapter 3 of \cite{ST94} for more properties of $\alpha$-stable random measures. 

\medskip

More generally, if instead of (iii), one assumes that $Z(B)$ has an infinitely divisible (ID) distribution, then  $Z$ is an {\em ID independently scattered random measure}, an object which was introduced and studied in \cite{RR89}, and shares many properties with L\'evy processes.

\medskip

One way to construct a S$\alpha$S random measure $Z$ is by setting: (see \cite{B14})
\begin{equation}
\label{Poisson-repr}
Z(B)=\int_{B \times \{0<|z|\leq 1\}}z \widehat{N}(dt,dx,dz)+\int_{B \times \{|z|> 1\}}z N(dt,dx,dz),
\end{equation}
where $N$ is a Poisson random measure $N$ on $[0,T] \times \bR^d \times \bR_0$ of intensity $\mu(dt,dx,dz)=dtdx \nu_{\alpha}(dz)$, $\bR_0=\bR \verb2\2\{0\}$, $\widehat{N}(F)=N(F)-\mu(F)$ is the compensated version of $N$, and
\begin{equation}
\label{def-nu-a}
\nu_{\alpha}(dz)=\frac{1}{2}\alpha|z|^{-\alpha-1}1_{\{|z|>0\}}dz.
\end{equation}

Another way is to use the {\em LePage series representation}: 
for any $B \in \cB_b$, 
\begin{equation}
\label{LePage1}
Z(B) =\sum_{i\geq 1}\e_i \Gamma_i^{-1/\alpha} \frac{1}{\psi(T_i,X_i)} 1_{B}(T_i,X_i) \quad \mbox{a.s.},
\end{equation}
where $(\e_i)_{i\geq 1}$ are i.i.d. Rademacher random variables, i.e.  $$\bP(\e_i=1)=\bP(\e_i=-1)=\frac{1}{2},$$ $\{\Gamma_i=\sum_{j=1}^{i}E_j, i\geq 1\}$ are the arrival times of a Poisson process on $\bR_{+}$ of intensity $1$ (with $(E_i)_{i\geq 1}$ i.i.d. exponential random variables of mean $1$), and
$\{(T_i,X_i)\}_{i\geq 1}$ are i.i.d. random variables on $[0,T] \times \bR^d$ with law $m_{\psi}(dt,dx)=\psi^{\alpha}(t,x)dtdx$, 
where $\psi:[0,T] \times \bR^d \to (0,\infty)$ is an arbitrary measurable function satisfying
\begin{equation}
\label{def-psi}
\int_0^T \int_{\bR^d}\psi^{\alpha}(t,x)dxdt=1.
\end{equation}
The sequences $(\e_i)_{i\geq 1}$, $(E_i)_{i\geq 1}$ and $\{(T_i,X_i)\}_{i\geq 1}$ are independent. Note that the series on the right-hand side of \eqref{LePage1} converges a.s. since $W_i=\frac{1}{\psi(T_i,X_i)} 1_{B}(T_i,X_i)$ are i.i.d. random variables with $\bE \left[ W_i^{\alpha} \right] <\infty$ (see Theorem 1.4.2 of \cite{ST94}).

\medskip

For simplicity, throughout this work we will use a weight function $\psi$ of the form:
\begin{equation}
\psi(t,x)=T^{-1/\alpha}\phi(x) \quad \mbox{for all $(t,x) \in [0,T] \times \bR^d$},
\end{equation}
where $\phi:\bR^d \to (0,\infty)$ is a measurable function such that $\int_{\bR^d} \phi^{\alpha}(x)dx=1$. 
Therefore, $(T_i)_{i \geq 1}$ are i.i.d. random variables with a uniform distribution on $[0,T]$, $(X_i)_{i\geq 1}$ are i.i.d. random vectors in $\bR^d$ with density $\phi^{\alpha}(x)$, and $(T_i)_{i \geq 1}$ and $(X_i)_{i \geq 1}$ are independent. The LePage series representation \eqref{LePage1} becomes: for any $B \in \cB_b$,
\[
Z(B)=T^{1/\alpha} \sum_{i\geq 1}\e_i \Gamma_i^{-1/\alpha} \frac{1}{\phi(X_i)} 1_{B}(T_i,X_i) \quad \mbox{a.s.}
\]

\medskip

A (mild) {\bf solution} of \eqref{eq} is a predictable random field $u=\{u(t,x);t\in [0,T],x\in \bR^d\}$ which satisfies the integral equation:
\begin{equation}
\label{int-eq}
u(t,x)=1+\int_0^t \int_{\bR^d}G_{t-s}(x-y)u(s,y)Z(ds,dy),
\end{equation}
where $G$ is the fundamental solution of the deterministic equation $\cL u=\delta_0$. 

\medskip

To give a meaning to the right-hand side of equation \eqref{int-eq}, we need a stochastic integral which can be defined for {\em random integrands}. This case is not discussed in \cite{ST94} or \cite{RR89}. It turns out that the theory of stochastic integration with respect to $\alpha$-stable random measures developed by the first author in \cite{B14} (as a multi-dimensional extension of the theory of \cite{GM83}) is too restrictive, and the best approach is to use the theory of stochastic integration with respect to $L^0$-random measures, which was introduced in \cite{BJ83} and was developed further in \cite{leb1,leb2,bichteler02,chong-kluppelberg15}. 
This theory was used in the literature for other studies of SPDEs with heavy-tailed noise, such as \cite{chong17-JTP,chong17-SPA,CDH19}. To implement this method, we need embed the S$\alpha$S random measure $Z$ into a more general process $\Lambda$ (called a {\em L\'evy basis}), indexed by subsets of $\Omega \times [0,T]\times \bR^d$. This will unavoidably increase the technical level of the paper, 
but the gain will be substantial. 
At the same time, we need to preserve the series representation \eqref{LePage1}, to define the multiple integrals with respect to $Z$. This delicate technical issue is addressed in Section \ref{subsection-construction} below, where we give an explicit construction of a L\'evy basis $\Lambda$ which allows us to achieve both goals. This construction uses as the source of randomness the three sequences $(\e_i)_{i\geq 1}$, $(\Gamma_i)_{i\geq 1}$ and $\{(T_i,X_i)\}_{i\geq 1}$ to define a Poisson random measure $N$, which in turn is used to define $\Lambda$, via its canonical decomposition. 
Specifically, $u(t,x)$ is represented as a series of {\em random multilinear forms} that depend only on $(\epsilon_i)_{i \geq 1}$, $(\Gamma_i)_{i \geq 1}$, $\{(T_i, X_i)\}_{i \geq 1}$, and $G$.

\medskip

As examples, we consider the case of the heat operator $\cL=\frac{\partial }{\partial t}-\frac{1}{2}\Delta$, for which
\[
G_t(x)=\frac{1}{2}\exp\left(-\frac{|x|^2}{2t}\right),
\]
and the case of the wave operator $\cL=\frac{\partial^2 }{\partial t^2}-\Delta$ in dimension $d\leq 2$, for which
\begin{align}
	G_t(x)=
	\begin{cases}
		\displaystyle \frac{1}{2}1_{\{|x|<t\}}                               & \text{if $d=1,$}\\[1em]
		\displaystyle \frac{1}{2\pi} \frac{1}{\sqrt{t^2-|x|^2}}1_{\{|x|<t\}} & \text{if $d=2.$}
	\end{cases}
\end{align} 
Here $|\cdot|$ is the Euclidean norm in $\bR^d$. These examples are referred in the literature as the {\em parabolic Anderson model} (PAM), respectively {\em the hyperbolic Anderson model} (HAM).
By convention, we let $G_t(x)=0$ if $t \leq 0$.

\medskip

The existence of a solution of (PAM) driven by L\'evy noise with 
positive jumps was proved in \cite{berger} using a different method than in the present paper. 
This method relies on the ``continuum directed polymer'' developed in \cite{berger-lacoin22}, and consists in solving the equation driven by the truncated noise with jumps that exceed a fixed value $a$, and then let $a\to 0$. A different truncation method was used in \cite{chong17-JTP,chong17-SPA,CDH19} for the {\em stochastic heat equation} (SHE) with general L\'evy noise multiplied by a Lipschitz function $\sigma(u)$. 
This method relies on first solving the equation driven by the noise truncated to the 
region $\{(x,z);|z|<Kh(x)\}$ for a suitable function  $h(x)$, up to a stopping time $\tau_K$, 
show that the solutions are consistent if $K<K'$, and finally paste together all these solutions. 
Unfortunately, this method does not yield uniqueness of the solution.
The same truncation method was used  in \cite{balan49} to show the existence of a 
solution of the {\em stochastic wave equation} (SWE) on $\mathbb{R}^d$ with  $d \le 2$. The 
uniqueness of this solution has been obtained in the recent article \cite{JJ1} using 
the past light-cone property of the fundamental solution $G$ of the wave equation.
 In the case of (SHE) with general Lipschitz function $\sigma$, uniqueness of the solution remains an open problem.

\medskip

In the present article, we introduce a different method, which is more robust, 
does not require any truncation, and can be applied to a large class of SPDEs, being inspired 
by the methodology used in the Gaussian case. We explain this method below.
Writing
\[
u(s,y)=1+\int_{0}^{s}\int_{\bR^d}G_{s-r}(y-z) u(r,z)Z(dr,dz)
\]
and inserting this into \eqref{int-eq}, we obtain:
\begin{align*}
u(t,x)&=1+\int_0^t \int_{\bR^d}G_{t-s}(x-y)Z(ds,dy)+\\
& \quad \quad \int_0^t \int_{\bR^d} G_{t-s}(x-y) \left(\int_0^s \int_{\bR^d}G_{s-r}(y-z)u(r,z)Z(dr,dz)\right) Z(ds,dy).
\end{align*}

Intuitively, it should be possible to iterate this procedure, and therefore obtain that the solution has the ``stable chaos expansion'':
\begin{equation}
\label{series}
u(t,x)=1+\sum_{n\geq 1}\int_{([0,T] \times \bR^d)^n} f_n(t_1,x_1,\ldots,t_n,x_n,t,x) Z(dt_1,dx_1) \ldots Z(dt_n,dx_n),
\end{equation}
where the kernel $f_n(\cdot,t,x)$ is given by:
\begin{equation}
\label{def-fn}
f_n(t_1,x_1,\ldots,t_n,x_n,t,x)=G_{t-t_n}(x-x_n)\ldots G_{t_2-t_1}(x_2-x_1) 1_{\{0<t_1<\ldots<t_n<t\}},
\end{equation}
and the integral is interpreted as a ``multiple stable integral'', which was 
introduced and studied in \cite{ST90,ST91}, using series representations 
(originating from the LePage series representation \eqref{LePage1}). The construction
 and basic properties of the multiple stable integrals are recalled in Section \ref{subsection-multiple} below.

Note that the first integral appearing in series \eqref{series} is $\int_0^t \int_{\bR^d}G_{t-s}(x-y)Z(ds,dy)$, and this integral is well-defined if and only if 
\begin{equation}
\label{C-eq}
\int_0^t \int_{\bR^d}G_{t-s}^{\alpha}(x-y) dyds<\infty.
\end{equation}
In the case of the heat equation, \eqref{C-eq} is equivalent to $\alpha<1+\frac{2}{d}$, 
whereas in the case of the wave equation in dimension $d\leq 2$, \eqref{C-eq} holds for all $\alpha \in (0,2)$. 
In light of this, we introduce the following hypothesis:

\begin{hypothesis}
\label{G_assumption}
The fundamental solution $(t,x)\mapsto G_t(x)$ of the operator $\mathcal{L}$ is a jointly measurable
 function on $[0,T] \times \bR^d$, which satisfies the following condition:
\[
\int_0^T \int_{\bR^d}G_t^{\alpha}(x) dxdt<\infty.
 \]
\end{hypothesis}

Note that Hypothesis \ref{G_assumption} does not hold for the wave operator in dimension $d\geq 3$,
 since the fundamental solution
is a distribution.

\medskip

We expect that if the series \eqref{series} is well-defined, the corresponding partial sum sequence will coincide with the sequence $(u_n)_{n\geq 0}$ of Picard's iterations, defined by: $u_0(t,x)=1$,
\begin{equation}
\label{picard}
u_{n+1}(t,x)=1+\int_0^t \int_{\bR^d}G(t-s,x-y)u_n(s,y)Z(ds,dy), \quad n \geq 0.
\end{equation}

This procedure is well-established in the case of equations with Gaussian noise or finite variance L\'evy noise, using tools from Malliavin calculus (see e.g. \cite{hu-nualart09,BZ24}). 
The goal of this article is to show that this method can be extended to the case of the S$\alpha$S noise, using the LePage series representations of the noise and of the associated multiple stable integrals.

\medskip

Guided again by the intuition gained from the Gaussian framework, 
we expect this solution to be unique. But uniqueness turns out to be a delicate
problem in our framework.
The fact that our noise may not have any moments forces us to work in the space $L^0$ of 
random variables equipped with the pseudo-norm $\|\cdot\|_0$ (see \eqref{L0-norm} below).
Since the topology induced by \(\|\cdot\|_{L^0}\)  is not locally convex, 
 we do not have access to the same techniques as in a Hilbert space, such as $L^2$, 
 nor do we have the Banach space structure of $L^p$ for $p>0$. 
 As a consequence, proving the uniqueness of the solution remains an open problem.

\medskip

We should also mention that stochastic evolution equations driven by a {\em cylindrical 
$\alpha$-stable L\'evy noise}
have been studied by several authors: see  
\cite{BRT24,kosmala-riedle21} and the references therein. Unlike the Gaussian case
 (see \cite{dalang-quer11}), a direct comparison of our results with the results obtained using
  this framework is not available at the present time.

\medskip

We are now ready to state the main result. But first, we need to introduce some notation and assumptions. 
For any $t \in [0,T]$, $x \in \bR^d$ and $p>0$, we define:
\begin{equation}
\label{def-K}
K_n^{(p)}(t,x):=\int_{T_n(t)} \int_{(\bR^d)^n} f_n^p(t_1,x_1,\ldots,t_n,x_n,t,x) \prod_{k=1}^{n} \phi^{\alpha-p}(x_k) d\pmb{x}  d\pmb{t},
\end{equation}
with $\pmb{x}=(x_1,\ldots,x_n)$, $\pmb{t}=(t_1,\ldots,t_n)$ and  $T_n(t)=\{\pmb{t}\in[0,t]^n;t_1<\ldots<t_n\}$.

\begin{assumption}
\label{ass-A2}
There exists \(p \in (\alpha,2]\) such that for all \((t,x) \in [0,T] \times \mathbb{R}^d\), 
\begin{equation}
\label{A2}
     \sum_{n \geq 1} \left(T^{(\frac{p}{\alpha}-1)n} K_n^{(p)}(t,x)\right)^{1/2} \Big( \sum_{j \geq 1}\Gamma_j^{-p/\alpha}\Big)^{n/2}<\infty \ \text{a.s.}
\end{equation}    
\end{assumption}

\begin{assumption}
\label{ass-A3}
There exists \(p \in (\alpha,2]\) such that for all \((t,x) \in [0,T] \times \mathbb{R}^d\), 
\begin{equation}
\label{A3}
\sum_{n \geq 1}\left(T^{(\frac{p}{\alpha}-1)n} \int_0^t \int_{\mathbb{R}^d} G_{t-s}^{\alpha}(x-y) K_n^{(p)}(s,y) \, dy \, ds \right)^{\frac{\alpha\wedge 1}{p}} \Big(\sum_{j \geq 1} \Gamma_j^{-p/\alpha} \Big)^{\frac{n(\alpha\wedge 1)}{p}}<\infty \quad \text{a.s.}
\end{equation}
\end{assumption}

The following theorem is the main result of the present article.

\begin{theorem}
\label{main-th}
Suppose that the fundamental solution $G$ satisfies Hypothesis \ref{G_assumption}. 

(a) If Assumption \ref{ass-A2} holds,
then the series on the right-hand side of \eqref{series} converges absolutely almost surely, for any $(t,x) \in [0,T] \times \bR^d$.

(b) If Assumptions \ref{ass-A2} and \ref{ass-A3} hold (with possibly different values $p$), then the process $\{u(t,x);t\in [0,T],x\in \bR^d\}$ given by \eqref{series} is 
a solution of equation \eqref{eq}. 
Moreover, for any $(t,x) \in [0,T] \times \bR^d$,
\begin{equation}
\label{Gu-int}
\int_0^t \int_{\mathbb{R}^d} G_{t-s}^{\alpha} (x-y) |u (s,y)|^{\alpha} \, ds \, dy < \infty \quad \text{a.s.}
\end{equation}
and $u(t,x)$ has representation:
\begin{equation}
\label{u-series}
u(t,x) = 1 + \sum_{n \ge 1} T^{n/\alpha}  n! \sum_{j_1<\ldots<j_n}
\prod_{k=1}^{n} \e_{j_k} \Gamma_{j_k}^{-1/\alpha} \phi^{-1}(X_{j_k})  \widetilde{f}_n(T_{j_1},X_{j_1},\ldots,T_{j_n},X_{j_n},t,x) \quad \text{a.s.},
\end{equation}
where $\widetilde{f}_n(\cdot,t,x)$ is the symmetrization of $f_n(\cdot,t,x)$.
\end{theorem}


To verify Assumptions \ref{ass-A2} and \ref{ass-A3}, we introduce the following condition on $\phi$:

\begin{hypothesis}
\label{hypo2}
There exist some $c_0>0$ and $\delta>0$ such that
\begin{equation}
\label{bound-phi}
\frac{1}{\phi(x)} \leq c_0 (1+|x|^{\delta}) \quad \mbox{for all $x \in \bR^d$}.
\end{equation}
\end{hypothesis}

\begin{remark}
{\rm
An example of a function $\phi>0$ which satisfies Hypothesis \ref{hypo2} 
 is 
\begin{equation}
\label{def-phi}
\phi(x)=c
\big(1_{\{|x|\leq 1\}}+|x|^{-\delta}1_{\{|x|>1\}}\big),
\end{equation}
 with  $\delta>d/\alpha$, and value
$c=c(\alpha,\delta)>0$ chosen such that $\int_{\bR^d}\phi^{\alpha}(x)dx=1$.
}
\end{remark}

\medskip

As an application of Theorem \ref{main-th}, we obtain the following result.

\begin{theorem}
\label{main-appl}
Suppose that $\phi$ satisfies Hypothesis \ref{hypo2}. If either \\
(i) $\cL=\frac{\partial}{\partial t}-\frac{1}{2}\Delta$ is the heat operator and 
$\alpha<1+\frac{2}{d}$, or\\
(ii) $\cL=\frac{\partial^2}{\partial t^2}-\Delta$ is the wave operator and $d\leq 2$, \\
then Assumptions \ref{ass-A2} and \ref{ass-A3} are satisfied, and consequently, the conclusion of Theorem \ref{main-th}.(b) holds.
\end{theorem}

We include below some simulations for the profile of the solution of $\cL u(t,x)= \dot{Z}(t,x) $ with initial condition 1, as a function of $(t,x)\in [0,1]^2$,
\begin{equation}
\label{linear}
u(t,x) = 1 + \int_0^t  \int_{\bR^d} G_{t-s}(x-y) Z(ds,dy)=
1+\sum_{i\geq 1}\varepsilon_i \Gamma_i^{-1/\alpha}\frac{1}{\phi(X_i)}G_{t-T_i}(x-X_i),
\end{equation}
for (SHE) and (SWE) on $[0,1]\times \bR$, when $\alpha=0.7$ and $\alpha=1.5$.
We approximated the series \eqref{linear} by the partial sum up to $n=1000$, and we used the function $\phi$ given by \eqref{def-phi} with $\delta=1.5$.

\begin{figure}[h!]
    \centering
\begin{subfigure}[b]{.48\textwidth}
        \includegraphics[width=.9\textwidth]{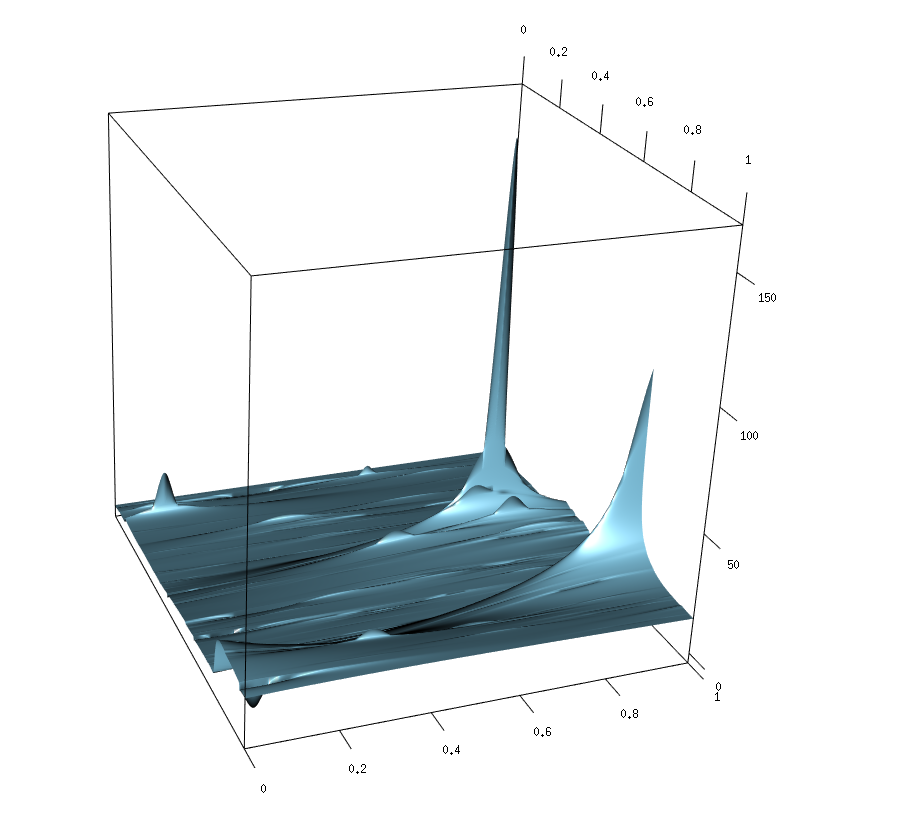}
        \caption{$\alpha=0.7$}
        \label{fig:heat0715}
\end{subfigure}
\begin{subfigure}[b]{.45\textwidth}
        \includegraphics[width=.9\textwidth]{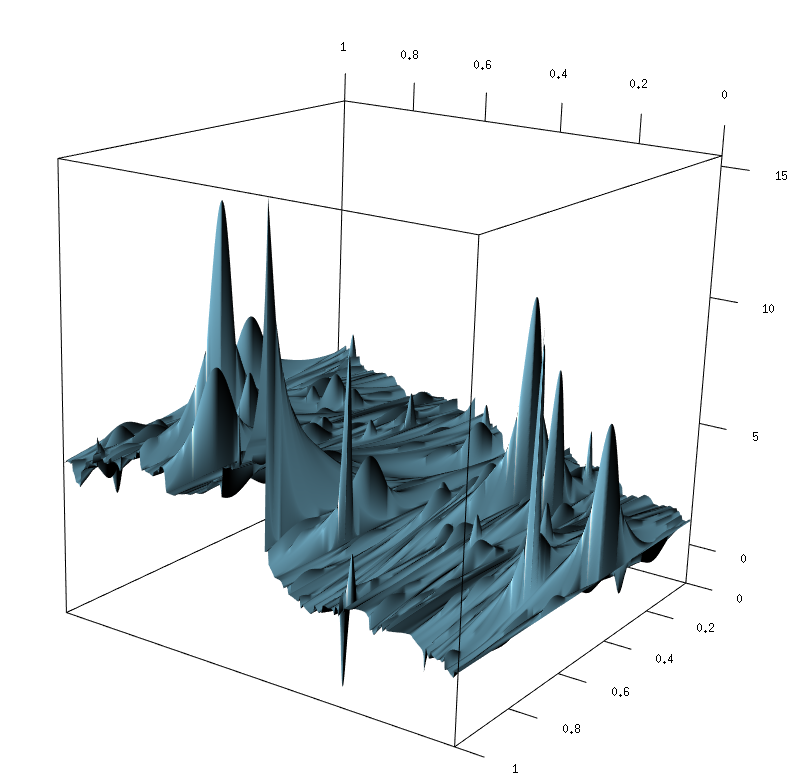}
        \caption{$\alpha=1.5$}
        \label{fig:heat1507}
\end{subfigure}
\caption{Simulation of the solution of the (SHE) with additive S$\alpha$S L\'evy noise}
    \label{fig:heat}
\end{figure}
\begin{figure}[h!]
    \centering
\begin{subfigure}[b]{0.47\textwidth}
        \includegraphics[width=0.9\textwidth]{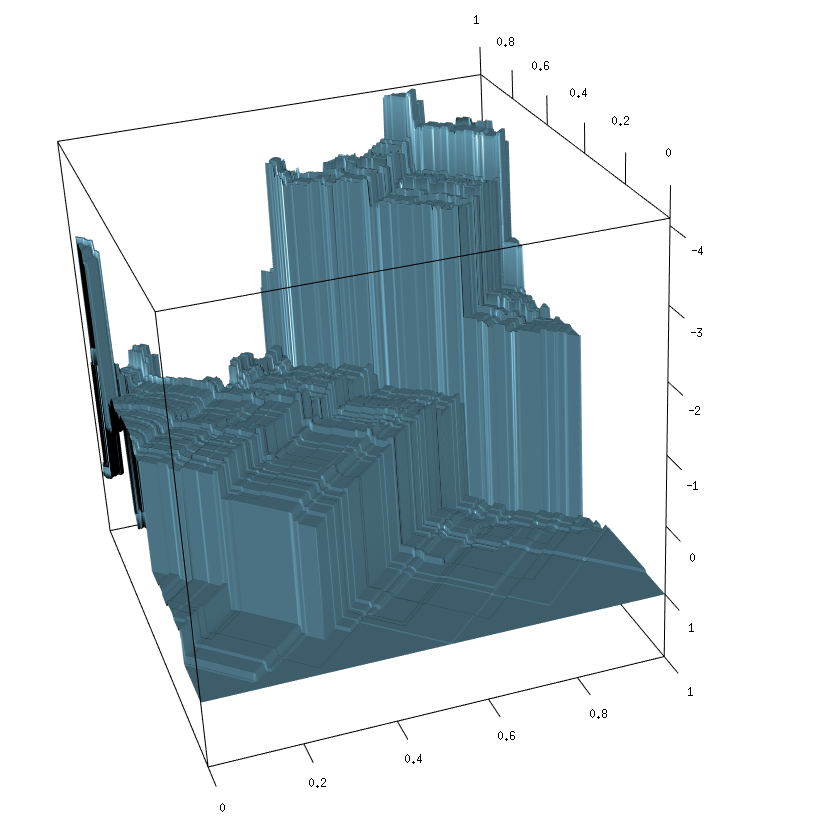}
        \caption{$\alpha=0.7$}
        \label{fig:wave0715}
\end{subfigure}
\begin{subfigure}[b]{0.48\textwidth}
        \includegraphics[width=.9\textwidth]{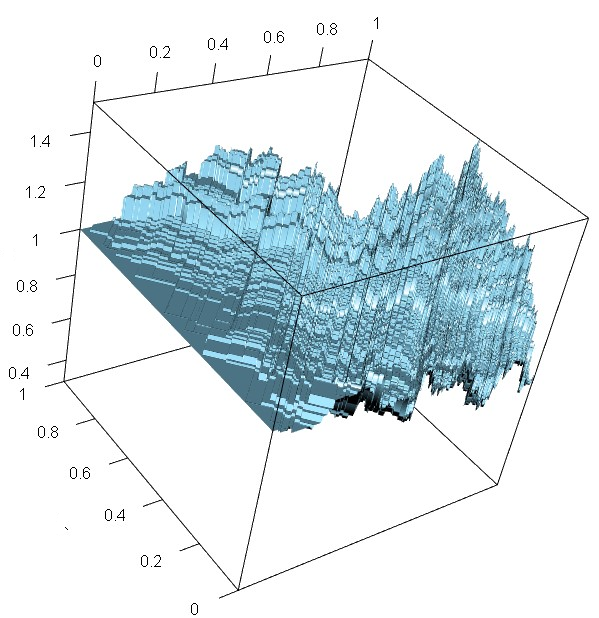}
        \caption{$\alpha=1.5$}
        \label{fig:wave1513}
\end{subfigure}
    \caption{Simulation of the solution of the (SWE) with additive S$\alpha$S L\'evy noise}
    \label{fig:wave}
\end{figure}

This article is organized as follows. In Section \ref{section-background}, we introduce 
the necessary background material, related to multiple stable integrals and L\'evy bases. 
In Section \ref{section-convergence} we give the proof of
Theorem \ref{main-th}.(a), i.e. we show that the series \eqref{series} converges absolutely 
almost surely. In Section \ref{section-recurrence}, we show that the partial sum sequence $(u_n)_{n\geq 0}$ 
associate with series \eqref{series} satisfies the recurrence relation \eqref{picard}.  
In Section \ref{section-solvability},  we show that 
the process $u(t,x)$ defined by \eqref{series} satisfies the integral equation
\eqref{int-eq}. In Section \ref{section-applications}, we give the proof of Theorem \ref{main-appl}, i.e.
we show that Assumptions \ref{ass-A2} and \ref{ass-A3} are satisfied in the case of the heat and 
wave equations.
Appendix \ref{app-integrN} contains some basic facts about the 
stochastic integration with respect to a compensated Poisson random measure. 
Appendix \ref{app-Fubini} includes an application of Fubini's theorem which is used in the sequel. 
Appendix \ref{app-heat} contains some estimates for several integrals associated with the heat
kernel.

\medskip

Finally, we mention few words about the notation. We work on 
a complete probability space  $(\Omega,\cF,\bP)$ equipped with a filtration $(\cF_t)_{t\in [0,T]}$ which satisfies the usual conditions. We denote by $L^0$ be the set of all random variables on $(\Omega,\cF,\bP)$, equipped with the {\em pseudo-norm}:
\begin{equation}
\label{L0-norm}
\|X\|_{L^0}=\bE[|X| \wedge 1].
\end{equation}
 Note that convergence in the $L^0$-norm is equivalent to the convergence in probability, since for any $\e\in (0,1)$, 
\[
\|X\|_{L^0} \leq \e+\bP(|X|>\e) \quad \mbox{and} \quad P(|X|>\e)=\bP(|X|\wedge 1>\e)\leq \frac{1}{\e}\|X\|_{L^0}.
\]

We let $\cB_b(\bR_0)$ be the class of bounded Borel sets in $\bR_0$ equipped with the 
distance $d(x,y)=|x^{-1}-y^{-1}|$. 
(A set $F \subset \bR_0$ is bounded if $F \subset \{|z|>a\}$ for some $a>0$.)
The notation $\phi^{-1}(x)$ is used for $1/\phi(x)$ (not for the inverse function of $\phi$).

\section{Background}
\label{section-background}

In this section, we include the necessary background material, which is substantial, since it covers two topics: multiple stable integrals (Section \ref{subsection-multiple}) and integration with respect to L\'evy bases (Section \ref{subsection-Levy-bases}). In Section \ref{subsection-construction}, we give the construction of the noise.

\subsection{Multiple stable integrals}
\label{subsection-multiple}

In this section, we review the construction of the multiple stable integral with respect to the S$\alpha$S random measure $Z$, following \cite{ST90,ST91}. The results contained in these references are stated for S$\alpha$S random measures on $\bR$, but they can be easily extended to more general spaces, such as $[0,T]\times \bR^d$.

\medskip

Before we begin, we recall that the stochastic integral $I_1(f)=\int fdZ$ of a deterministic function $f$ with respect to $Z$ is constructed using the classical procedure, starting with simple functions, followed by approximation. This is explained in Chapter 3 of \cite{ST94}. It turns out that the integral $I_1(f)$ is well-defined for any function $f \in L^{\alpha}([0,T] \times \bR^d)$, and the process $\{I_1(f); f \in L^{\alpha}([0,T] \times \bR^d)\}$ has $\alpha$-stable finite dimensional distributions. In particular, $I_1(f)$ has a $S_{\alpha}(\sigma_f,0,0)$-distribution, where
\[
\sigma_f=\left(\int_{[0,T] \times \bR^d} |f(t,x)|^{\alpha}m(dt,dx)\right)^{1/\alpha}.
\]
This property is in fact valid even in the non-symmetric case, when $\nu_{\alpha}(z,\infty)=pz^{-\alpha}$ and $\nu(-\infty,-z)=q(-z)^{-\alpha}$ for all $z>0$, for some $p,q\geq 0$. 
From this perspective, $\alpha$-stable random measures are similar to isonormal Gaussian processes from Malliavin calculus, in which case the integrals $I_1(f)$ live in the first Wiener chaos. This analogy is further enhanced by the fact that,
{\em in the symmetric case} (when $\nu_{\alpha}$ is given by \eqref{def-nu-a}), it is possible to define a multiple integral with respect to $Z$, and this is what we will explain below.

\begin{remark}
{\rm
If $\alpha \in (1,2)$, then $\int_{|z|>1}|z|\nu_{\alpha}(dz)<\infty$, and the Poisson representation \eqref{Poisson-repr} can be written as 
$Z(B)=\int_{B \times \bR_0}z\widehat{N}(dt,dx,dz)$.
In view of this, one may try to define the multiple integral with respect to $Z$ of a function $f:([0,T]\times \bR^d)^n \to \bR$ by:
\[
I_n(f)=\int_{([0,T]\times \bR^d \times \bR_0)^n} f(t_1,x_1,\ldots,t_n,x_n)z_1 \ldots z_n \widehat{N}(dt_1,dx_1,dz_1)\ldots \widehat{N}(dt_n,dx_n,dz_n).
\]
This procedure works for finite-variance L\'evy noise (see \cite{BZ24}), but fails for the infinite variance noise $Z$, since the multiple integral with respect to $\widehat{N}$  is defined only for functions which are square-integrable with respect to the measure $\mu^n$ 
(see Section 5.4 of \cite{applebaum09}).
}
\end{remark}

\medskip

A different idea, which has been exploited successfully in \cite{ST90,ST91}, is to start with a S$\alpha$S random measure $Z$ {\em which has the LePage representation \eqref{LePage1}}, and extend this representation to 
multi-index series. We explain this procedure below. To simplify the writing, we let $E=[0,T] \times \bR^d$, $\cB$ be the class of Borel sets of $E$, and $\ell$ be the Lebesgue measure on $\bR$. Let $\cB_0=\{A \in \cB; \ell(A)<\infty\}$.

\medskip

The first step is the construction of the product stable measure $Z^{(n)}$.
For this, we recall that a {\em symmetric rectangle} of $E^n$ is a set of the form
\[
B=\bigcup_{\pi \in \Sigma_n} B_{\pi(1)} \times \ldots \times B_{\pi(n)}
\]
for some  disjoint sets $B_1,\ldots,B_n \in \cB_0$, where $\Sigma_n$ is the set of all permutations of $\{1,\ldots,n\}$.
Defining
$$Z^{(n)}(B):=n! Z(B_1) \ldots Z(B_n),$$
and using the Le Page representation \eqref{LePage1} for each $B_i$, we obtain that:
\begin{equation}
\label{lepage-n}
Z^{(n)}(B)=n! \sum_{j_1<\ldots< j_n} \, \prod_{k=1}^{n}
\e_{j_k}
\Gamma_{j_k}^{-1/\alpha} \psi^{-1}(T_{j_k},X_{j_k})
1_{B}(T_{j_1},X_{j_1},\ldots,T_{j_n},X_{j_n}).
\end{equation}

The following theorem shows that this representation can be extended to more general sets.
First, we need to introduce some notation. Let $\cB_{n}^{(s)}$ be the $\sigma$-algebra generated by the symmetric rectangles and
$\cB_{n,0}^{(s)}=\{B \in \cB_n^{(s)};\ell^{(n)}(B)<\infty\}$, where $\ell^{(n)}$ is the Lebesgue measure on $E^n$.

\begin{theorem}
For any set $B \in \cB_{n,0}^{(s)}$, 
the series
\[
S_n(B)=n! \sum_{j_1<\ldots< j_n} \, \prod_{k=1}^{n}
\e_{j_k}
\Gamma_{j_k}^{-1/\alpha} \psi^{-1}(T_{j_k},X_{j_k})
1_{B}(T_{j_1},X_{j_1},\ldots,T_{j_n},X_{j_n})
\]
converges a.s. if and only if
\begin{equation}
\label{A-conv}
\sum_{j_1<\ldots<j_n} \prod_{k=1}^{n}
\Gamma_{j_k}^{-2/\alpha} \psi^{-2}(T_{j_k},X_{j_k})
1_{B}(T_{j_1},X_{j_1},\ldots,T_{j_n},X_{j_n}) <\infty \quad \mbox{a.s.}
\end{equation}
In this case, we let $Z^{(n)}(B):=S_n(B)$ a.s.
\end{theorem}




\medskip

We proceed now with the construction of the multiple stable integral. We recall some terminology.

\begin{definition}
{\rm Let $f:E^n \to \bR$ be an arbitrary function. We say that:\\
a) $f$ is {\em symmetric} if for any $x_1,\ldots,x_n\in E$ and for any $\pi \in \Sigma_n$,
\[
f(x_1,\ldots,x_n)=f(x_{\pi(1)},\ldots,x_{\pi(n)});
\]
b) $f$ {\em vanishes on the diagonals} if $f(x_1,\ldots,x_n)=0$ whenever $x_i=x_j$ for some $i\not=j$.
} 
\end{definition}

\begin{lemma}
A symmetric $\cB^n$-measurable function $f:E^n\to \bR$ which vanishes on the diagonals is $\cB_{n}^{(s)}$-measurable.
\end{lemma}

We say that $f:E^n \to \bR$  is a {\em simple function} if it is the form $f=\sum_{i=1}^{k}a_i 1_{B_i}$, for some $a_1,\ldots,a_k \in \bR$ and disjoint sets $B_1,\ldots,B_k \in \cB_{n,0}^{(s)}$. If $Z^{(n)}(B_i)$ is well-defined for any $i=1,\ldots,k$, we say that $f$ is {\em $n$-times integrable with respect to $Z$}, and we let
\[
I_n(f):=\sum_{i=1}^{k}a_i Z^{(n)}(B_i).
\]


\begin{definition}
\label{def-n-int-f}
{\rm Let $f:E^n \to \bR$ be a $\cB^{n}$-measurable symmetric function which vanishes on the diagonals of $E^n$.
We say that {\em $f$ is $n$-times integrable with respect to $Z$} if there exists a sequence $\{f^{(k)}\}_{k\geq 1}$ of simple functions which are $n$-times integrable with respect to $Z$ such that: \\
(i) $\{f^{(k)}\}_{k\geq 1}$ converges to $f$ in the measure $\ell^{(n)}$;\\
(ii) for any $B \in \cB_{n}^{(s)}$, the sequence $\{I_n(f^{(k)} 1_{B})\}_{k\geq 1}$ converges in probability to a limit denoted by $I_{n}(f 1_{B})$. \\
In this case, we let $I(f1_{B})$ be the limit in probability of $\{I(f^{(k)}1_{B})\}_{k\geq 1}$.
}
\end{definition}

Recall that the {\em symmetrization} of a function $f$ is defined by:
\[
\widetilde{f}(x_1,\ldots,x_n)=\frac{1}{n!}\sum_{\pi \in \Sigma_n} f(x_{\pi(1)},\ldots,x_{\pi(n)}).
\]
A $\cB^n$-measurable function $f$ which vanishes on the diagonals is  $n$-times integrable with respect to $Z$ if its symmetrization $\widetilde{f}$ is so; in this case, we let $I_n(f)=I_n(\widetilde{f})$. 

We use the notation
\[
I_n(f)=\int_{E^n}f(t_1,x_1,\ldots,t_n,z_n)Z(dt_1,dx_1)\ldots Z(dt_n,dx_n),
\]
and we say that $I_n$ is the {\em multiple integral of order $n$} of $f$ with respect to $Z$. 

\medskip



\medskip

\medskip


We now present a criterion for integrability.

\begin{theorem}
\label{integr-crit}
A symmetric $\cB^n$-measurable function $f:E^n\to \bR$ which vanishes on the diagonals is 
$n$-times integrable with respect to $Z$ if and only if
\begin{equation}
\label{conv-s2}
\sum_{j_1<\ldots<j_n} \,
\prod_{k=1}^{n} \Gamma_{j_k}^{-2/\alpha} \psi^{-2}(T_{j_k},X_{j_k})
f^2(T_{j_1},X_{j_1},\ldots,T_{j_n},X_{j_n})<\infty \quad \mbox{a.s.}
\end{equation}
In this case, the series
\[
S_n(f):=n!\sum_{j_1<\ldots<j_n} \,
\prod_{k=1}^{n} \e_{j_k}\Gamma_{j_k}^{-1/\alpha} \psi^{-1}(T_{j_k},X_{j_k})
f(T_{j_1},X_{j_1},\ldots,T_{j_n},X_{j_n})
\]
converges a.s. and $I_n(f)=S_n(f)$ a.s. 
\end{theorem}

A sufficient condition for \eqref{conv-s2} is
\[
\int_{E^n} |f(t_1,x_1,\ldots,t_n,x_n)|^{\alpha}\left[\ln_{+}\frac{|f(t_1,x_1,\ldots,t_n,x_n)|}{\psi(t_1,x_1) \ldots \psi(t_n,x_n)} \right]^{n-1} dt_1 dx_1 \ldots dt_n dx_n<\infty.
\]
Under this condition, the multiple integral has the following asymptotic tail behaviour:
\begin{equation}
\label{tail}
\lim_{\lambda \to \infty}\frac{\lambda^{\alpha}}{(\ln \lambda)^{n-1}} \bP(|I_n(f)|>\lambda)=n (n!)^{\alpha-2} \alpha^{n-1}\|f\|_{\alpha,n}^{\alpha}, \quad \mbox{for any $n\geq 3$},
\end{equation}
where $\|f\|_{\alpha,n}^{\alpha}=
\int_{E^n}|f(t_1,x_1,\ldots,t_n,x_n)|^{\alpha}dt_1dx_1\ldots dt_n dx_n$. Relation \eqref{tail} holds also for $n=2$ under the additional condition:
\[
\int_{E^2} |f(t_1,x_1,t_2,x_2)|^{\alpha}\ln_{+}\frac{|f(t_1,x_1,t_2,x_2)|}{\psi(t_1,x_1)\psi(t_2,x_2)}
\ln_{+}\left|\ln \frac{|f(t_1,x_1,t_2,x_2)|}{\psi(t_1,x_1)\psi(t_2,x_2)}  \right|dt_1 dx_1 dt_2 dx_2<\infty.
\]

This concludes our summary about multiple stable integrals.

\subsection{L\'evy bases and stochastic integration}
\label{subsection-Levy-bases}

In this section, we introduce the definition of a L\'evy basis, and present the construction and main properties of the stochastic integral with respect to this object, following closely the notation and terminology of \cite{chong17-JTP}.

Let $(\Omega,\cF,\bP)$ be a complete probability space, equipped with a filtration $(\cF_t)_{t \in [0,T]}$ which satisfies the usual conditions.
Let $\cP$ be the predictable $\sigma$-field on $\Omega \times [0,T]$ with respect to the filtration $(\cF_t)_{t\in [0,T]}$ and $\wP=\cP \times \cB(\bR^d)$, be the predictable $\sigma$-field on $\Omega \times [0,T] \times \bR^d$. The predictable $\sigma$-field on $\Omega \times [0,T] \times \bR^d \times \bR_0$ is  $\wP \times \cB(\bR_0)$.
Processes $\{X(t,x);t\in [0,T],x\in \bR^d\}$ and $\{H(t,x,z);t\in [0,T],x\in \bR^d,z\in \bR_0\}$ are called {\em predictable} if they are measurable with respect to the corresponding predictable $\sigma$-field.
We denote by $\Pb$ the class of all sets $A \in \widetilde{\cP}$ such that $A \subseteq \Omega \times [0,k] \times [-k,k]^d$ for some $k\in \bN$.

\begin{definition}
{\rm A {\em L\'evy basis} is a mapping $\Lambda:\Pb\to L^0$ which satisfies the following: \\
(1) $\Lambda(\emptyset)=0$ a.s.\\
(2) For any disjoint sets $(A_i)_{i\geq 1}$ in $\Pb$ with $\bigcup_{i\geq 1}A_i \in \Pb$, 
\[
\Lambda\Big(\bigcup_{i\geq 1}A_i\Big)=\sum_{i\geq 1}\Lambda(A_i) \quad \mbox{in} \quad L^0. 
\]
(3) For all $A \in \Pb$ with $A \subseteq \Omega \times [0,t] \times \bR^d$, $\Lambda(A)$ is $\cF_t$-measurable.\\
(4) For all $A \in \Pb$ and $F \in \cF_t$ for some $t \leq T$,
\begin{equation}
\label{predict}
\Lambda\big(A \cap (F \times (t,T] \times \bR^d)\big)=1_{F} \Lambda\big(A \cap (\Omega \times (t,T] \times \bR^d)\big) \quad \mbox{a.s.}
\end{equation}
(5) For any disjoint sets $(B_i)_{i\geq 1}$ in $\cB_b$, $\{\Lambda(\Omega \times B_i)\}_{i\geq 1}$ are independent. Moreover, if $B \in \cB_b$ is such that $B \subseteq (t,T] \times \bR^d$ for some $t\leq T$, then $\Lambda(\Omega \times B)$ is independent of $\cF_t$.\\
(6) For any $B \in \cB_b$, $\Lambda(\Omega \times B)$ has an infinitely divisible (ID) distribution.\\
(7) for any $t\in [0,T]$ and $k \in \bN$, $\Lambda(\Omega \times \{t\} \times [-k,k]^d)=0$ a.s.
}
\end{definition}

Intuitively, a L\'evy basis is a generalization of a L\'evy process. But for our purposes, it is important to realize that if $\Lambda$ is L\'evy basis, then the process $Z$ defined by
\begin{equation}
\label{def-Z}
Z(B)=\Lambda(\Omega \times B) \quad \mbox{for all $B \in \cB_b$ }
\end{equation}
is an ID independently scattered random measure.

 \medskip
Similarly to the classical L\'evy-It\^o decomposition of a L\'evy process, a L\'evy basis has a canonical decomposition (see Theorem 3.2 of \cite{chong-kluppelberg15}):
\begin{equation}
\label{canon-Lambda}
\Lambda(A)=B(A)+\Lambda^{C}(A)+\int_{A \times \{|z|\leq 1\}}z \widehat{N}(dt,dx,dz)+\int_{A \times \{|z|> 1\}}z N(dt,dx,dz),
\end{equation}
which contains three terms: \\
(i) a drift term, given by a deterministic signed measure $B$ on $[0,T] \times \bR^d$; \\
(ii) a continuous component, given by a Gaussian random measure $\Lambda^{C}$ on $[0,T]\times \bR^d$ with variance measure $C$; \\
(iii) a pure-jump component, characterized by an underlying Poisson random measure $N$ on $U=[0,T] \times \bR^d \times \bR_0$ of intensity measure $\mu$, whose compensated version is called $\widehat{N}$. 

Moreover, $B$, $C$ and $\mu(\cdot,dz)$ have respective densities $b(t,x)$, $c(t,x)$, and $\nu(t,x,dz)$ with respect to the Lebesque measure. When these densities do not depend on $(t,x)$, we say that $\Lambda$ is a {\em homogeneous L\'evy basis}. If $C=0$, we say that $\Lambda$ {\em pure-jump}. If $b(t,x)=0$ and the measure $\nu(t,x,\cdot)$ is symmetric for all $(t,x)$, then we say that $\Lambda$ is {\em symmetric}. 

\medskip

\begin{definition}
\label{def-SL}
{\rm A homogenous L\'evy basis $\Lambda$ with canonical decomposition \eqref{canon-Lambda} with $B=0$, $\Lambda^C=0$ and $\nu=\nu_{\alpha}$ (where $\nu_{\alpha}$ is given by \eqref{def-nu-a}) is called a {\em S$\alpha$S L\'evy basis}.
}
\end{definition}

\begin{remark}
\label{rem-SL}
{\rm  If $\Lambda$ is a S$\alpha$S L\'evy basis with canonical decomposition \eqref{canon-Lambda}, then the process $Z$ defined by \eqref{def-Z} is a S$\alpha$S random measure with Poisson representation \eqref{Poisson-repr}.}
\end{remark}

The Poisson random measure $N$ is the key object of this theory.
We recall its definition below, but we refer the reader to Chapter 3 of \cite{resnick87} for some of its properties.

\medskip

 Let $U$ be a locally compact space with a countable basis, equipped with its Borel $\sigma$-field (for us, $U=[0,T] \times \bR^d \times \bR_0$).
Let $M_p(U)$ be the set of point measures on $U$ equipped with the topology of vague convergence, and $\mathcal{M}_p(U)$ be the corresponding Borel $\sigma$-field. Let $\mu$ be a Radon measure on $U$, i.e. $\mu(F)<\infty$ for any compact set $F \subset U$.

\begin{definition}
{\rm 
 A {\em Poisson random measure} (PRM) on $U$ of intensity $\mu$ is a map $N:\Omega\to M_p(U)$ which is $\mathcal{M}_p(U)$- measurable and satisfies the following properties: \\ 
{\em (i)} $N(F)$ has a Poisson distribution with mean $\mu(F)$\footnote{We use the convention that $N(F)=\infty$ a.s. if $\mu(F)=\infty$}, for any Borel set $F$ in $U$; \\
{\em (ii)} $N(F_1),\ldots, N(F_n)$ are independent for any disjoint Borel sets $F_1,\ldots,F_n$ in $U$. 
}
\end{definition}

\medskip

If $N$ is a PRM on $U$ of intensity $\mu$, we define the {\em compensated process} $\widehat{N}$ by 
\[
\widehat{N}(F)=N(F)-\mu(F)
\] 
for any Borel set $F$ in $U$ with $\mu(F)<\infty$. Unlike $N(\omega,\cdot)$, $\widehat{N}(\omega,\cdot)$ is not a measure for any fixed $\omega \in \Omega$.
But one can develop a theory of stochastic integration with respect to $\widehat{N}$, similarly to It\^o's theory. We refer the reader to Appendix \ref{app-integrN} for more details.

\medskip

We recall now briefly the construction and main properties of the stochastic integral with respect to a L\'evy basis $\Lambda$, which shares many elements with the classical stochastic integral with respect to semi-martingales.

\medskip

A {\em simple integrand} is a linear combination of indicators of the form $1_{A}$ with $A \in \Pb$. On the other hand, an {\em elementary process} is a linear combination of processes of the form
\begin{equation}
\label{def-elem}
X(t,x)=Y 1_{(a,b]}(t) 1_{B}(x),
\end{equation}
where $Y$ is $\cF_a$-measurable, $0\leq a<b\leq T$ and $B \in \cB_b$.

We let $\cS$ be the set of all simple integrands and $\cE$ be the set of all elementary processes.
It is known that
 $\wP=\sigma(\cE)$. In particular,
a set of the form $A=F \times (a,b] \times B$ with $F\in \cF_a$, $0\leq a<b\leq T$ and $B \in \cB_b$  is in $\Pb$. Moreover, by property \eqref{predict} of $\Lambda$, we have:
\begin{equation}
\label{predict2}
\Lambda(F\times (a,b] \times B)=1_{F}Z((a,b] \times B).
\end{equation}

 The two sets $\cS$ and $\cE$ are not the same, but every process in one set can be approximated by a process in the other set, as shown by the next lemma.

\begin{lemma}
\label{approx-lem}
(i) For any process $S \in \cS$, there exists a sequence $(X_n)_{n\geq 1}$ in $\cE$ such that $X_n \to S$ uniformly on compact sets of $\bR_+ \times \bR^d$.\\
(ii) For any process $X\in \cE$, there exists a sequence $(S_n)_{n\geq 1}$ in $\cS$ such that for all $(t,x)$,
 $S_n(t,x)\to X(t,x)$ and $|S_n(t,x)|\leq  |X(t,x)|$ for all $n$.
\end{lemma}

\begin{proof}
(i) This follows from the following fact, stated on page 172 of \cite{bichteler02}: $\Pb$ is the sequential closure of $\cE$ with respect to the topology of ``confined uniform convergence'', i.e. for any $A \in \Pb$ and for any $\varepsilon>0$, there exists a process $X_{\varepsilon} \in \cE$, such that
for any compact set $K \subset [0,T] \times \bR^d$,
\[
\sup_{(t,x)\in K}|X_{\varepsilon}(t,x)-1_{A}(t,x)|<\varepsilon.
\]

(ii) Without loss of generality, we assume that $X$ is an elementary process of the form \eqref{def-elem}. By Theorem 13.5 of \cite{billingsley95}, there exists a sequence $(Y_n)_{n\geq 1}$ of simple random variables such that $Y_n\to Y$, $|Y_n|\leq |Y|$ for all $n$, and each $Y_n$ is a linear combination of indicator functions of the form $1_{F}$ with $F\in \cF_a$. It suffices to take $S_n(t,x)=Y_n 1_{(a,b]}(t) 1_{B}(x)$.
\end{proof}

For any $A \in \Pb$, let $\int 1_{A}d\Lambda=\Lambda(A)$. By linearity, we extend this
 definition to $\cS$. For any predictable process $H$, we define the {\em Daniell mean}: 
\[
\|H\|_{\Lambda}=\sup_{S \in \cS,|S|\leq |H|}\left\|\int S d\Lambda \right\|_{L^0}.
\]

Note that the Daniell mean satisfies the triangular inequality: 
\begin{equation}
\label{triang-Daniell}
\|H_1+H_2\|_{\Lambda} \leq \|H_1\|_{\Lambda} +\|H_2\|_{\Lambda}.
\end{equation}

\begin{definition}
\label{def-integrable} {\rm Let $\Lambda$ be a L\'evy basis and $Z$ be the corresponding ID independently scattered random measure (given by \eqref{def-Z}).\\
a) A predictable process $H$ is {\em integrable} with respect to $\Lambda$ if there exists a sequence $(S_n)_{n\geq 1}$ in $\cS$ such that $\|S_n-H\|_{\Lambda}\to 0$ as $n \to \infty$. \\
b) A predictable process $H$ is {\em integrable} with respect to $Z$ if it is integrable with respect to $\Lambda$.
}
\end{definition}

We denote by $L^{0}(\Lambda)$ the class of integrable processes with respect to $\Lambda$, which is the closure of $\cS$ with respect to $\|\cdot\|_{\Lambda}$.  
If $Z$ is the ID independently scattered random measure induced by $\Lambda$ (via relation \eqref{def-Z}), by abuse of terminology, we let 
\begin{equation}
\label{convention}
L^{0}(Z)=L^{0}(\Lambda) \quad \mbox{and} \quad \|\cdot\|_{Z}=\| \cdot\|_{\Lambda}
\end{equation}

\medskip

For any $S \in \cS$, we denote $I^{\Lambda}(S)=\int S d\Lambda$. Unlike It\^o's theory, the map $I^{\Lambda}: \cS \to L^0$ is {\em not} an isometry! But the fact that this map satisfies the following trivial inequality
\begin{equation}
\label{contraction}
\|I^{\Lambda}(S)\|_{L^0}\leq \|S\|_{\Lambda} \quad \mbox{for all $S \in \cS$}
\end{equation}
is sufficient for extending $I^{\Lambda}$ from $\cS$ to $L^{0}(\Lambda)$. More precisely, if $H \in L^{0}(\Lambda)$ and $(S_n)_{n\geq 1}$ is the approximating sequence of simple integrands given by Definition \ref{def-integrable}, then 
$\{I^{\Lambda}(S_n)\}_{n\geq 1}$ is a Cauchy sequence in $L^0$ since
\begin{align*}
\|I^{\Lambda}(S_n)-I^{\Lambda}(S_m)\|_{L^0} \leq \|S_n-S_m\|_{\Lambda} \leq \|S_n-H\|_{\Lambda}+
\|S_m-H\|_{\Lambda} \to 0,
\end{align*}
as $n,m \to \infty$. By definition, we set
$I^{\Lambda}(H)=\lim_{n \to \infty} I^{\Lambda}(S_n)$ in $L^0$, and we say that $I^{\Lambda}(H)$ is the {\em stochastic integral} of $H$ with respect to $\Lambda$ (or $Z$). We will use both notations:
\[
I^{\Lambda}(H)=\int H d\Lambda=\int H dZ=I^{Z}(H).
\]
By approximation, it follows that inequality \eqref{contraction} can be extended to $L^{0}(\Lambda)$:
\begin{equation}
\label{contraction2}
\|I^{\Lambda}(H)\|_{L^0}\leq \|H\|_{\Lambda} \quad \mbox{for all $H \in L^{0}(\Lambda)$}.
\end{equation}

The stochastic integral $I^{\Lambda}$ satisfies the dominated convergence theorem. We include this result below, since we will use it often in the present article. This result was stated as relation (2.6) of \cite{BJ83}. The form that we present here corresponds to Theorem A.1 of \cite{CDH19}.

\begin{theorem}[Dominated Convergence Theorem for $I^{\Lambda}$]
\label{DCT}
Let $\Lambda$ be a homogeneous L\'evy basis. 
Let $(H_n)_{n\geq 1}$ be predictable processes such that $(H_n)_{n\geq 1}$ converges pointwise to $H$, and  $|H_n|\leq |H_0|$ for all $n$, for some $H_0 \in L^0(\Lambda)$. Then $H_n,H \in L^0(\Lambda)$ and $\|H_n-H\|_{\Lambda} \to 0$. Consequently,
\[
I^{\Lambda}(H_n) \to I^{\Lambda}(H) \quad \mbox{in} \quad L^0.
\]
\end{theorem}

Theorem 4.1 of \cite{chong-kluppelberg15} gives necessary and sufficient conditions for integrability with respect to an arbitrary {\em $L^0$-random measure} (see Definition 2.1 of \cite{chong-kluppelberg15}); by Remark 4.4 of \cite{chong-kluppelberg15}, a L\'evy basis is an orthogonal $L^0$-random measure. We include the statement of this result below for a homogeneous L\'evy basis, which is a generalization Theorem 2.7 of \cite{RR89}, that corresponds to ID independently scattered random measures with deterministic integrands.

\begin{theorem}[Theorem 4.1 of  \cite{chong-kluppelberg15} for homogeneous L\'evy bases]
Let $\Lambda$ be a pure-jump homogeneous L\'evy basis,
with canonical decomposition \eqref{canon-Lambda} (with $\Lambda^C=0$), 
and $H$ be
a predictable process. Then $H \in L^0(\Lambda)$ if and only if
\begin{equation}
\label{integr-cond}
\int_{0}^{T} \int_{\bR^d} U(H(t,x))dxdt<\infty \ \mbox{a.s.} \quad \mbox{and} \quad\int_{0}^{T}\int_{\bR^d} V_0(H(t,x))dxdt<\infty \ \mbox{a.s.},
\end{equation}
where 
\[
U(y)=by+\int_{\bR} \big(\tau(yz)-y\tau(z)\big)\nu(dz) \quad \mbox{and} \quad V_0(y)=\int_{\bR}(|yz|^2 \wedge 1)\nu(dz),
\]
for any $y \in \bR$, where $\tau(z)=z1_{\{|z|\leq 1\}}$.
\end{theorem}

In the particular case of a S$\alpha$S L\'evy basis,
$b=0$ and $\nu=\nu_{\alpha}$ is given by \eqref{def-nu-a}, so
$U(y)=0$ and $V_0(y)=\frac{2}{2-\alpha}|y|^{\alpha}$, which means that
condition \eqref{integr-cond} is equivalent to $H \in L^{\alpha}([0,T] \times \bR^d)$ a.s. We state this result in the corollary below, noting that it gives a natural extension to random integrands of the integrability criterion 
 which was mentioned at the beginning of Section \ref{subsection-multiple}.

\begin{corollary}
\label{integr-th}
Let $\Lambda$ be a S$\alpha$S L\'evy basis, and $H$ be a predictable process. Then $H \in L^0(\Lambda)$ if and only if
\begin{equation}
\label{integr-cond-a}
\int_0^{T}\int_{\bR^d}|H(t,x)|^{\alpha}dxdt<\infty \quad \mbox{a.s.}
\end{equation}
\end{corollary}

The following result shows that for elementary processes, the stochastic integral $I^{\Lambda}$ coincides with the It\^o integral (defined in the Walsh' sense \cite{walsh86}). We include its proof since we could not find it in the literature.

\begin{theorem}
Let $X$ be an elementary process of the form \eqref{def-elem}.
If $\Lambda$ is a pure-jump homogeneous L\'evy basis, and $Z$ be given by \eqref{def-Z},
then $X \in L^0(\Lambda)$ and
\begin{equation}
\label{IXZ}
I^{\Lambda}(X)=Y Z\big((a,b] \times B\big) \quad \mbox{a.s.}
\end{equation}
\end{theorem}

\begin{proof} Note that $U(X(t,x))=U(Y)1_{(a,b]}(t) 1_{B}(x)$ and $V_0(X(t,x))=V_0(Y) 1_{(a,b]}(t) 1_{B}(x) $. Hence, condition \eqref{integr-cond} holds, and
$X \in L^0(\Lambda)$.

Let $(S_n)_{n\geq 1}$ be the sequence of simple integrands given by Lemma \ref{approx-lem}.(ii). More precisely, if $Y_n=\sum_{i=1}^{k_n}a_{i,n} 1_{F_{i,n}}$ with $a_{i,n} \in \bR$ and $F_{i,n}\in \cF_a$ is the sequence of simple random variable such that $Y_n \to Y$ and $|Y_n|\leq |Y|$ for all $n$, then 
\[
S_n(t,x)=\sum_{i=1}^{k_n}a_{i,n} 1_{F_{i,n}} 1_{(a,b]}(t) 1_B(x).
\]
Using \eqref{predict2}, we obtain that:
\[
I^{\Lambda}(S_n)=\sum_{i=1}^{k_n}a_{i,n} \Lambda(F_{i,n} \times (a,b] \times B)=\sum_{i=1}^{k_n}a_{i,n} 
1_{F_{i,n}} Z \big( (a,b] \times B\big)=Y_n Z \big( (a,b] \times B\big).
\]
Relation \eqref{IXZ} follows letting $n\to \infty$, since $I^{\Lambda}(S_n) \stackrel{P}{\to} I^{\Lambda}(X)$, by Theorem \ref{DCT}. 
\end{proof}

\subsection{Construction of the noise}
\label{subsection-construction}

In this section, we give the construction of the noise. 

\medskip

Let $(\e_i)_{i\geq 1}$, $(\Gamma_i)_{i\geq 1}$ and $\{(T_i,X_i)\}_{i\geq 1}$ be the sequences mentioned in the introduction, defined on a complete probability space $(\Omega,\cF,\bP)$. Note that $\big(\e_i \Gamma_i^{-1/\alpha}\big)_{i\geq 1}$ are the points of a PRM on $\bR_{0}$ of intensity $\nu_{\alpha}$.
Therefore, using a procedure called ``augmentation'' (see Proposition 3.8 of \cite{resnick87}), the process
\begin{equation}
\label{def-J-psi}
J_{\psi}=\sum_{i\geq 1}\delta_{(T_i,X_i,\e_i \Gamma_i^{-1/\alpha})}
\end{equation}
is a PRM on $[0,T] \times \bR^d \times \bR_0$ of intensity $m_{\psi} \times \nu_{\alpha}$. Consider the transformation $T_{\psi}(t,x,z)=(t,x,\frac{z}{\psi(t,x)})$. By Proposition 3.7 of \cite{resnick87}, the process
\[
N_{\psi}=J_{\psi}\circ T_{\psi}^{-1}=\sum_{i\geq 1}\delta_{T_{\psi}(T_i,X_i,\e_i\Gamma_i^{-1/\alpha})}=\sum_{i\geq 1}\delta_{(T_i,X_i, \e_i\Gamma_i^{-1/\alpha}\psi^{-1}(T_i,X_i) )}
\]
is also PRM on $[0,T] \times \bR^d \times \bR_0$ of intensity 
\[
(m_{\psi} \times \nu_{\alpha})\circ T_{\psi}^{-1}={\rm Leb} \times {\rm Leb} \times \nu_{\alpha}.
\]

Now, using the PRM $N_{\psi}$, we define the homogeneous L\'evy basis $\Lambda$ on $[0,T] \times \bR^d$:
\begin{equation}
\label{def-Lambda}
\Lambda(A)=\int_{0}^T \int_{\bR^d} \int_{\{|z|\leq 1\}} 1_A(t,x)z \widehat{N_{\psi}}(dt,dx,dz)+
\int_{0}^T \int_{\bR^d} \int_{\{|z|> 1\}} 1_A(t,x)z N_{\psi}(dt,dx,dz),
\end{equation}
where $\widehat{N_{\psi}}$ is the compensated version of $N_{\psi}$.

We let  $(\cF_t)_{t \in [0,T]}$ be the filtration associated with $N_{\psi}$, i.e. 
\begin{equation}
\label{def-filtration}
\cF_t=\sigma\left\{ N_{\psi}([0,s] \times B \times F);s \in [0,t], B \in \cB_b, F \in \cB_b(\bR_0) \right\}.
\end{equation}

For any $B \in \cB_b$, we define $Z(B)$ by \eqref{def-Z}. Then $\Lambda$ is a S$\alpha$S L\'evy basis  (see Definition \ref{def-SL}), and $Z$ is a S$\alpha$S random measure (see Remark \ref{rem-SL}). The next result shows that $Z$ has the series representation \eqref{LePage1}.

\begin{lemma}
If $\Lambda$ is the L\'evy basis given by \eqref{def-Lambda}, then the process
$Z$ given by \eqref{def-Z} has the LePage series representation \eqref{LePage1}.
\end{lemma}

\begin{proof}
 Let $\e\in (0,1)$ be arbitrary. By the symmetry of $\nu_{\alpha}$, $\int_{B \times \{\e<|z|\leq 1\}} z dt dx \nu_{\alpha}(dz)=0$ and hence,
\[
\int_{B \times \{\e<|z|\leq 1\}}z \widehat{N_{\psi}}(dt,dx,dz)=\int_{B \times \{\e<|z|\leq 1\}}z N_{\psi}(dt,dx,dz),
\]
and
\[
Z(B)=\int_{B \times \{0<|z|<\e\}}z \widehat{N_{\psi}}(dt,dx,dz)+\int_{B \times \{|z|> \e\}}z N_{\psi}(dt,dx,dz)=:Z_{\e}(B)+S_{\e}(B).
\]
Note that $Z_{\e}(B) \stackrel{L^2}{\to} 0$ as $\e \to 0$, since $\bE \left[ |Z_{\e}(B)|^2 \right] =\int_{B\times \{\e<|z|\leq 1\}}z^2 dtdx \nu_{\alpha} \to 0$ as $\e \to 0$, by the dominated convergence theorem. 
Next, we prove that:
\begin{equation}
\label{Se-conv}
S_{\e}(B) \stackrel{P}{\to} S(B) \quad \mbox{as $\e \to 0$},
\end{equation}
where $S(B)$ is the sum on the right-hand side of \eqref{LePage1}. Note that
\[
S(B)-S_{\e}(B)=\sum_{i\geq 1}\e_i \Gamma_i^{-1/\alpha}\psi^{-1}(T_i,X_i) 
1_{\{\Gamma_i^{-1/\alpha} \leq \e
\psi(T_i,X_i)\}} 1_B(T_i,X_i)=:\overline{S}_{\e}(B).
\]
By the Cauchy-Schwarz inequality, 
\[
\|\overline{S}_{\e}(B)\|_0=\bE\big[\min\big(1,|\overline{S}_{\e}(B)|\big)\big] 
\leq \Big\{ \bE\big[\min\big(1,|\overline{S}_{\e}(B)|^2\big)\big]\Big\}^{1/2}.
\]
Using the inequality $\bE[\min(1,|X|)]\leq \min(1,\bE|X|)$ and the orthogonality of $(\e_i)_{i\geq 1}$, 
\begin{align*}
& \bE\Big[\min\big(1,|\overline{S}_{\e}(B)|^2\big)\,| \, (\Gamma_i)_i,(T_i)_i,(X_i)_i\Big] \\
& \leq \min\Big(1, \sum_{i\geq 1}\Gamma_i^{-2/\alpha} \psi^{-2}(T_i,X_i) 
1_{\{\Gamma_i^{-1/\alpha}\leq \e\psi(T_i,X_i) \}} 1_{B}(T_i,X_i) \Big).
\end{align*}
Taking expectation in the above inequality, we obtain:
\[
\bE\Big[\min\big(1,|\overline{S}_{\e}(B)|^2\big)\Big] \leq 
\bE\Big[ \min\Big(1, \sum_{i\geq 1}\Gamma_i^{-2/\alpha} \psi^{-2}(T_i,X_i) 
1_{\{\Gamma_i^{-1/\alpha}\leq \e \psi(T_i,X_i) \}} 1_{B}(T_i,X_i) \Big)\Big].
\]
The term on the right-hand side above converges to 0 as $\e\to 0$, by an application of the dominated convergence theorem, which is justified by the fact that
\[
X:=\sum_{i\geq 1}\Gamma_i^{-2/\alpha} \psi^{-2}(T_i,X_i)1_{\{\Gamma_i^{-1/\alpha}\leq
 \psi(T_i,X_i) \}} 1_{B}(T_i,X_i)=\int_{B \times \{|z|\leq 1\}}z^2 N_{\psi}(dt,dx,dz)<\infty \quad \mbox{a.s.}
\]
since $\bE \left[ X \right] =\int_{B \times \{|z|\leq 1\}}z^2 dtdx \nu_{\alpha}(dz)<\infty$.
\end{proof}

The following result gives an alternative representation for $\Lambda$. Its proof follows essentially using a change of measure, which allows us to pass not only from $N_{\psi}$ to $J_{\psi}$, but also from $\widehat{N_{\psi}}$ to $\widehat{J_{\psi}}$ (although these are not measures), using the symmetry of $\nu_{\alpha}$. 

\begin{proposition}
Let $\Lambda$ be the L\'evy basis given by \eqref{def-Lambda}. For any $A \in \wP_b$, we have:
\begin{align}
\nonumber
\Lambda(A)&=\int_{0}^T \int_{\bR^d}\int_{\{|z|\leq \psi(t,x)\}}1_{A}(t,x)\frac{z}{\psi(t,x)}\widehat{J_{\psi}}(dt,dx,dz)+\\
\label{Lambda-repr1}
& \quad
\int_{0}^T \int_{\bR^d}\int_{\{|z|>\psi(t,x)\}}1_{A}(t,x)\frac{z}{\psi(t,x)}J_{\psi}(dt,dx,dz),
\end{align}
where $\widehat{J_{\psi}}$ is the compensated version of $J_{\psi}$. Moreover, if $\alpha \in (0,1)$, then for any $A \in \wP_b$,
\begin{align}
\label{Lambda-repr2}
\Lambda(A)&=\int_{0}^T \int_{\bR^d}\int_{\bR_0}1_{A}(t,x)\frac{z}{\psi(t,x)}J_{\psi}(dt,dx,dz)\\
\label{Lambda-LePage}
&=\sum_{i\geq 1}\e_i \Gamma_i^{-1/\alpha}\frac{1}{\psi(T_i,X_i)}1_{A}(T_i,X_i).
\end{align}
\end{proposition}

\begin{proof}
We first prove \eqref{Lambda-repr1}. 
Let $\e \in (0,1)$ be arbitrary. We write
\begin{align*}
\Lambda(A)&=\int_0^T \int_{\bR^d} \int_{\{|z|\leq \e\}} 1_{A}(t,x)z\widehat{N_{\psi}}(dt,dx,dz)+\int_0^T \int_{\bR^d} \int_{\{\e<|z|\leq 1\}} 1_{A}(t,x)z\widehat{N_{\psi}}(dt,dx,dz)+\\
& \quad \int_0^T \int_{\bR^d} \int_{\{|z|> 1\}} 1_{A}(t,x)z N_{\psi}(dt,dx,dz) =:T_{1}^{(\e)}+T_{2}^{(\e)}+T_3,
\end{align*}
and then we let $\e \to 0$. Since $N_{\psi}=J_{\psi}\circ T_{\psi}^{-1}$,
using a change of measure, we have
\[
T_3=\int_0^T \int_{\bR^d} \int_{\{|z|>\psi(t,x)\}} 1_{A}(t,x)\frac{z}{\psi(t,x)} J_{\psi}(dt,dx,dz).
\]

By It\^o's isometry, $T_{1}^{(\e)} \stackrel{L^2}{\to} 0$ as $\e \to 0$. For the second term, 
\begin{align*}
T_{2}^{(\e)}&=\int_0^T \int_{\bR^d} \int_{\{\e<|z|\leq 1\}} 1_{A}(t,x)z N_{\psi}(dt,dx,dz)  \quad \mbox{(by the symmetry of $\nu_{\alpha}$)}\\
&=\int_0^T \int_{\bR^d} \int_{\{\e<\frac{|z|}{\psi(t,x)}\leq 1\}} 1_{A}(t,x)\frac{z}{\psi(t,x)}J_{\psi}(dt,dx,dz) \quad \mbox{(since $N_{\psi}=J_{\psi} \circ T_{\psi}^{-1}$)}\\
&=\int_0^T \int_{\bR^d} \int_{\{\e<\frac{|z|}{\psi(t,x)}\leq 1\}} 1_{A}(t,x)\frac{z}{\psi(t,x)}\widehat{J_{\psi}}(dt,dx,dz) \quad \mbox{(by the symmetry of $\nu_{\alpha}$)} \\
& \stackrel{L^2}{\longrightarrow} \int_0^T \int_{\bR^d} \int_{\{\frac{|z|}{\psi(t,x)}\leq 1\}} 1_{A}(t,x)\frac{z}{\psi(t,x)}\widehat{J_{\psi}}(dt,dx,dz) \quad \mbox{as $\e\to 0$}.
\end{align*}
The last convergence holds by It\^o's isometry, using the fact that
\begin{align*}
&\bE \left[ \left|\int_0^T \int_{\bR^d}\int_{\{\frac{|z|}{\psi(t,x)}\leq \e\}} 1_{A}(t,x) \frac{z^2}{\psi^2(t,x)}\widehat{J_{\psi}}(dt,dx,dz)\right|^2 \right]\\
&\quad =\bE \left[ \int_{0}^T \int_{\bR^d}1_{A}(t,x) \psi^{\alpha-2}(t,x)\left(\int_{\{\frac{|z|}{\psi(t,x)}\leq \e \} }z^2 \nu_{\alpha}(dz)\right)dtdx \right] \\
&=\frac{\e^{2-\alpha}}{2-\alpha}\bE \left[ \int_{0}^T \int_{\bR^d}1_{A}(t,x)dtdx \right].
\end{align*}
This proves \eqref{Lambda-repr1}.
Relation \eqref{Lambda-repr2} follows directly from \eqref{Lambda-repr1}, since
\begin{align*}
&\int_0^T \int_{\bR^d} \int_{\{\frac{|z|}{\psi(t,x)}\leq 1\}}1_{A}(t,x) \frac{z}{\psi(t,x)} \widehat{J_{\psi}}(dt,dx,dz) \\
& =\int_0^T \int_{\bR^d} \int_{\{\frac{|z|}{\psi(t,x)}\leq 1\}}1_{A}(t,x) \frac{z}{\psi(t,x)} J_{\psi}(dt,dx,dz). \\
\end{align*}
To see this, note that by Fubini's theorem and the symmetry of $\nu_{\alpha}$,
\[
\int_0^T \int_{\bR^d} \int_{\{\frac{|z|}{\psi(t,x)} \leq 1\}}1_{A}(t,x) z \psi^{\alpha-1}(t,x)dtdx\nu_{\alpha}(dz)=0.
\]
To justify the application of Fubini's theorem, we note that since $\alpha<1$,
\[
\int_0^T \int_{\bR^d} 1_{A}(t,x) \psi^{\alpha-1}(t,x) \left(\int_{\{|z|\leq \psi(t,x)\}}|z| \nu_{\alpha}(dz) \right) dxdt=\frac{1}{1-\alpha}\int_0^T \int_{\bR^d}1_{A}(t,x)dtdx<\infty.
\]
Finally, \eqref{Lambda-LePage} follows directly from \eqref{Lambda-repr2}, using definition \eqref{def-J-psi} of $J_{\psi}$.
\end{proof}

\begin{remark}
\label{rem-LePage}
{\rm
When $\alpha \in (0,1)$, relation \eqref{Lambda-LePage} extends the LePage series representation \eqref{LePage1} to sets $A \in \widetilde{\cP}_b$. This relation will play an important role 
in proving the recurrence relation \eqref{picard}; see the proof of Theorem \ref{proof-IZ-Xn} below.
We do not know if this representation holds for sets $A \in \widetilde{\cP}_b$, when $\alpha \in [1,2)$.
}
\end{remark}

For integration purposes, we will use the following representation of $\Lambda$.

\begin{proposition}
\label{prop-Lambda-repr3}
Let $\Lambda$ be the L\'evy basis given by \eqref{def-Lambda}. For any $A \in \wP_b$, we have:
\begin{align}
\nonumber
\Lambda(A)&=
\int_{0}^T \int_{\bR^d}\int_{\{|z|\leq 1\}}1_{A}(t,x)\frac{z}{\psi(t,x)}\widehat{J_{\psi}}(dt,dx,dz)+\\
\label{Lambda-repr3}
&\quad  \int_{0}^T \int_{\bR^d}\int_{\{|z|>1\}}1_{A}(t,x)\frac{z}{\psi(t,x)}J_{\psi}(dt,dx,dz).
\end{align}
\end{proposition}

\begin{proof}
Consider the following (non-homogeneous) L\'evy basis:
\begin{equation}
\label{def-Lpsi}
L_{\psi}(A)=\int_{0}^T \int_{\bR^d}\int_{\{|z|\leq 1\}}1_{A}(t,x)z\widehat{J_{\psi}}(dt,dx,dz)+ \int_{0}^T \int_{\bR^d}\int_{\{|z|>1\}}1_{A}(t,x)z J_{\psi}(dt,dx,dz).
\end{equation}
This is an orthogonal $L^0$-random measure which has the canonical decomposition (3.2) of \cite{chong-kluppelberg15}, with respect to the truncation function $\tau(z)=z1_{\{|z|\leq 1\}}$, with $b=0$, $C=0$, $\mu=J_{\psi}$, $\nu(dt,dx,dz)=\psi^{\alpha}(t,x)dydx \nu_{\alpha}(dz)$, i.e. $A(dt,dx)=dtdx$ and $K(t,x,dz)=\psi^{\alpha}(t,x) \nu_{\alpha}(dz)$. For this $L^0$-random measure, we have:
\begin{align*}
U(t,x,y)&=b(t,x)y+\int_{\bR_0} \big(\tau(yz)-y\tau(z)\big)K(t,x,dz)=0\\
V_0(t,x,y)&=\int_{\bR_0}\big(|yz|^2 \wedge 1 \big)K(t,x,dz)=  \frac{2}{2-\alpha}
|y|^{\alpha} \psi^{\alpha}(t,x).
\end{align*}
By Theorem 4.1 of \cite{chong-kluppelberg15}, if $H$ is a predictable process, then
\begin{equation}
\label{equiv-L}
\mbox{$H \in L^{0}(L_{\psi})$ if and only if $\int_{0}^T \int_{\bR^d}\psi^{\alpha}(t,x) |H(t,x)|^{\alpha}dxdt<\infty$ a.s.}
\end{equation}
In particular, $1_{A}\psi^{-1} \in L^{0}(L_{\psi})$,
for any  $A \in \widetilde{P}_b$.

By Theorem 3.5 of \cite{chong-kluppelberg15}, we can define a null-spatial $L^0$-random measure $H \cdot L_{\psi}$ by $(H \cdot L_{\psi})(A)=\int 1_{A}(t) H(t,x) L_{\psi}(dt,dx)$ for suitable sets $A \in \cP$, and this has canonical decomposition:
\begin{align*}
(H \cdot L_{\psi})(A)&=\int_{[0,T]\times \bR^d \times \bR_0} 1_{A}(t) \tau(H(t,x)z)\widehat{J}_{\psi}(dt,dx,dz)+\\
& \quad
\int_{[0,T] \times \bR^d \times \bR_0} 1_{A}(t) \big( H(t,x)z -\tau(H(t,x)z)\big) J_{\psi}(dt,dx,dz).
\end{align*}
On the other hand, using the canonical decomposition \eqref{def-Lpsi} of $L_{\psi}$, we have:
\begin{align*}
(H \cdot L_{\psi})(A)&=\int_{[0,T]\times \bR^d \times \bR_0} 1_{A}(t) H(t,x) \tau(z)\widehat{J}_{\psi}(dt,dx,dz)+\\ 
& \quad \int_{[0,T]\times \bR^d \times \bR_0} 1_{A}(t) H(t,x)\big(z-\tau(z) \big) J_{\psi}(dt,dx,dz).
\end{align*}
Combining these two expressions and writing them for
 $A=\overline{\Omega}:=\Omega \times [0,T]$, we obtain:
\begin{align*}
(H \cdot L_{\psi})(\overline{\Omega})&=\int_0^T \int_{\bR^d} \int_{\{ |H(t,x) z| \leq 1\}}H(t,x) z \widehat{J}_{\psi}(dt,dx,dz)+ \\
& \int_0^T \int_{\bR^d} \int_{\{|H(t,x) z| > 1\}}H(t,x) z  J_{\psi}(dt,dx,dz)
\\
&= \int_0^T \int_{\bR^d}\int_{\{|z|\leq 1\}}H(t,x) z \widehat{J}_{\psi}(dt,dx,dz)+ \\
& \int_0^T \int_{\bR^d}\int_{\{|z|> 1\}}H(t,x) z J_{\psi}(dt,dx,dz). \\
\end{align*}
We apply this for $H=1_{A}\psi^{-1} $ with $A \in \widetilde{P}_b$. The conclusion follows by \eqref{Lambda-repr1}.
\end{proof}

Finally, we give the following representation for the integral with respect to $\Lambda$.

\begin{proposition}
\label{prop-Lambda-repr4}
For any $H \in L^{0}(\Lambda)$, we have:
\begin{align}
\nonumber
\int_0^T \int_{\bR^d} H(t,x) \Lambda(dt,dx) &=\int_0^T \int_{\bR^d}\int_{\{|z|\leq 1\}}H(t,x) \frac{z}{\psi(t,x)} \widehat{J}_{\psi}(dt,dx,dz)+\\
\label{Lambda-repr4}
& \int_0^T \int_{\bR^d}\int_{\{|z|> 1\}}H(t,x) \frac{z}{\psi(t,x)} J_{\psi}(dt,dx,dz).
\end{align}
\end{proposition}

\begin{proof}
By Proposition \ref{prop-Lambda-repr3}, relation \eqref{Lambda-repr4} holds for $H=1_{A}$ with $A \in \widetilde{P}_b$, i.e.
\begin{equation}
\label{Lam-L-psi}
\Lambda(A)=\int_0^T \int_{\bR^d} 1_{A}(t,x) \frac{1}{\psi(t,x)}L_{\psi}(dt,dx),
\end{equation}
where $L_{\psi}$ is the L\'evy basis given by \eqref{def-Lpsi}. 
By Corollary \ref{integr-th}, $H \in L^{0}(\Lambda)$ is equivalent to $\int |H(t,x)|^{\alpha}dtdx<\infty$ a.s., which in turn is equivalent to $H\psi^{-1} \in L^0(L_{\psi})$, by \eqref{equiv-L}.

By Theorem 13.5 of \cite{billingsley95}, there exists a sequence $(S_n)_{n\geq 1}$ of simple integrands such that $S_n \to H$ and $|S_n| \leq |H|$ for all $n$. By Dominated Convergence Theorem (Theorem \ref{DCT}),
\[
I^{\Lambda}(S_n) \stackrel{P}{\to} I^{\Lambda}(H) \quad \mbox{and} \quad I^{L_{\psi}}(S_n \psi^{-1}) \stackrel{P}{\to} I^{L_{\psi}}(H\psi^{-1}).
\]
By \eqref{Lam-L-psi}, $I^{\Lambda}(S_n) =I^{L_{\psi}}(S_n \psi^{-1})$ for all $n$. By uniqueness of the limit,  $I^{\Lambda}(H)=I^{L_{\psi}}(H\psi^{-1})$.
\end{proof}

The following result will be used in the proof of Lemma \ref{lem-J-conv} below 
(in the case $\alpha \geq 1$).

\begin{proposition}
\label{prop-Lambda-repr5}
Assume that $\alpha \in [1,2)$. For any $H \in L^{0}(\Lambda)$ and $a \in (0,1)$, we have:
\begin{align*}
\int_0^T \int_{\bR^d} H(t,x) \Lambda(dt,dx) &=\int_0^T \int_{\bR^d}\int_{\{|z|\leq a\}}H(t,x) \frac{z}{\psi(t,x)} \widehat{J}_{\psi}(dt,dx,dz)+\\
& \quad \int_0^T \int_{\bR^d}\int_{\{|z|> a\}}H(t,x) \frac{z}{\psi(t,x)} J_{\psi}(dt,dx,dz).
\end{align*}
\end{proposition}

\begin{proof}
In \eqref{Lambda-repr4}, we write the first term on the right hand-side as the sum of two integrals corresponding to sets $\{|z|\leq a\}$ and $\{a<|z|\leq 1\}$. Notice that
\[
\int_{0}^T \int_{\bR^d}\int_{\{a<|z|\leq 1\}}H(t,x)\frac{z}{\psi(t,x)}\widehat{J}_{\psi}(dt,dx,dz)=
\int_0^T \int_{\bR^d} \int_{\{a<|z|\leq 1\}} H(t,x) \frac{z}{\psi(t,x)} J_{\psi}(dt,dx,dz).
\]
This is because $\int_0^T \int_{\bR^d}\int_{\{a<|z|\leq 1\}} H(t,x)z\psi^{\alpha-1}(t,x) dxdt \nu_{\alpha}(dz)=0$, by Fubini's theorem and the symmetry of $\nu_{\alpha}$. To justify the application of Fubini's theorem, we need to prove
\[
\int_0^T \int_{\bR^d} |H(t,x)|\psi^{\alpha-1}(t,x) \left(\int_{\{a<|z|\leq 1\}} |z| \nu_{\alpha}(dz)\right) dx dt<\infty,
\]
which is equivalent to $\int_0^T \int_{\bR^d} |H(t,x)|\psi^{\alpha-1}(t,x) dxdt<\infty$. This last fact is clear when $\alpha=1$, and can be proved using H\"older's inequality with $p=\alpha$, when $\alpha \in (1,2)$.
\end{proof}

\section{Convergence of the series}
\label{section-convergence}

In this section, we give the proof of Theorem \ref{main-th}.(a). 
First, we show that under a condition weaker than the one given by Assumption \ref{ass-A2}, the series on the right-hand side of \eqref{series} converges in probability. Then, we show that under Assumption \ref{ass-A2}, this series converges absolutely almost surely. For the rest of the paper, we will denote this series by $u(t,x)$:
\begin{equation}
\label{def-u}
u(t,x):=1+\sum_{n\geq 1}I_n\big( f_n(\cdot,t,x)\big)=1+\sum_{n\geq 1}I_n\big( \widetilde{f}_n(\cdot,t,x)\big).
\end{equation}


\medskip

Recall that $f_n(\cdot,t,x)$ is the kernel given by \eqref{def-fn} and $\widetilde{f}_n(\cdot,t,x)$ is its symmetrization. By Theorem \ref{integr-crit}, the multiple stable integral
$I_n(\widetilde{f}_n(\cdot,t,x))$ is well-defined if and only if
\[
h_n^{(2)}(t,x):=T^{\frac{2n}{\alpha}} (n!)^2 
\sum_{j_1<\ldots<j_n}\prod_{k=1}^{n}\Gamma_{j_k}^{-2/\alpha}
\phi^{-2}(X_{j_k}) \widetilde{f}_n^2(T_{j_1},X_{j_1},\ldots,T_{j_n},X_{j_n},t,x)<\infty \quad \mbox{a.s.}
\]
and in this case, it has the LePage series representation:
\begin{equation}
\label{LePage-In}
I_n(\widetilde{f}_n(\cdot,t,x))=T^{n/\alpha}  n! \sum_{j_1<\ldots<j_n}
\prod_{k=1}^{n} \e_{j_k} \Gamma_{j_k}^{-1/\alpha} \phi^{-1}(X_{j_k})  \widetilde{f}_n(T_{j_1},X_{j_1},\ldots,T_{j_n},X_{j_n},t,x) \quad \mbox{a.s.}
\end{equation}

We have the following result.

\begin{lemma}
\label{gK-ineq}
Let $(\e_j)_{j\geq 1}$ be i.i.d. with $\bP(\e_j=1)=\bP(\e_j=-1)=1/2$. Let $\{a_{j_1,\ldots,j_n};1\leq j_1 <\ldots<j_n,j_i \in \bN\}$ be an array of non-negative numbers. The series \[
\sum_{n\geq 1}\sum_{j_1<\ldots<j_n}\e_{j_1}\ldots \e_{j_n}a_{j_1,\ldots,j_n} \quad \mbox{converges in $L^2(\Omega)$}
\]
if and only if
$\sum_{n\geq 1}\sum_{j_1<\ldots<j_n}a_{j_1,\ldots,j_n}^2<\infty$,
and in this case,
\[
\bE\left[ \left|\sum_{n\geq 1}\sum_{j_1<\ldots<j_n}\e_{j_1}\ldots \e_{j_n}a_{j_1,\ldots,j_n}\right|^2 \right] = \sum_{n\geq 1}\sum_{j_1<\ldots<j_n}a_{j_1,\ldots,j_n}^2.
\]
\end{lemma}

\begin{proof} 
The terms of the series are orthogonal in $L^2(\Omega)$ since
\[
\bE \left[ \e_{j_1}\ldots \e_{j_n} \e_{k_1}\ldots e_{k_m} \right]=
\left\{
\begin{array}{ll}
1 & \mbox{if $\{j_1,\ldots,j_n\}=\{k_1,\ldots,k_m\},$} \\
0 & \mbox{otherwise.}
\end{array} \right.
\]
\end{proof}

\begin{remark}
{\rm If $\sum_{j_1<\ldots<j_n}a_{j_1,\ldots,j_n}^2<\infty$, then by the generalized Khintchine inequality (Proposition 1 of \cite{ST90}), the series
$\sum_{j_1<\ldots<j_n}\e_{j_1}\ldots \e_{j_n}a_{j_1,\ldots,j_n}$ converges a.s. and in $L^p(\Omega)$ for any $p \geq 0$.}
\end{remark}

Using Lemma \ref{gK-ineq}, we see that
\begin{equation}
\label{hn2}
h_n^{(2)}(t,x)=\bE\Big[|I_n(\widetilde{f}_n(\cdot,t,x))|^2 \left| \right. (\Gamma_i),(T_i),(X_i)\Big].
\end{equation}

The following basic conditioning fact will be used several times in the sequel: if $X$ and $Y$ are independent random elements with values in measurable spaces $(E,\cE)$, respectively $(F,\cF)$, and $f:E \times F \to [0,\infty]$ is a measurable function, then
\begin{equation}
\label{fact1}
\bE[f(X,Y)|X=x]=\bE[f(x,Y)].
\end{equation}

\begin{lemma}
\label{conv-proba}
Let $t \in [0,T]$ and $x \in \bR^d$ be arbitrary. If
\begin{equation}
\label{cond-C2}
\cA_2(t,x):=\sum_{n\geq 1}h_n^{(2)}(t,x) <\infty
 \quad \mbox{a.s.},
 \end{equation}
 then the series $\sum_{n\geq 1}I_{n}(\widetilde{f}_n(\cdot,t,x))$ converges in probability.
\end{lemma}

\begin{proof}
Recall that convergence in probability is equivalent with the convergence in $L^0$ equipped with the norm $\|X\|_0=\bE[\min(1,|X|)]$.

Let $S_n=\sum_{k=1}^{n}I_{k}(\widetilde{f}_k(\cdot,t,x))$.
By the Cauchy-Schwarz inequality, for any $n>m$,
$$\bE[\min(1,|S_n-S_m|)]\leq \left(\bE[\min(1,|S_n-S_m|^2)]\right)^{1/2}.$$
Conditioning on $(\Gamma_i),(T_i),(X_i)$, we have:
\begin{align*}
\bE[\min(1,|S_n-S_m|^2)]&=\bE[\bE[\min(1,|S_n-S_m|^2)|(\Gamma_i),(T_i),(X_i)]].
\end{align*}
Using the inequality $\bE[\min(1,|X|)] \leq \min(1,E|X|)$, followed by Lemma \ref{gK-ineq}, we have:
\begin{align*}
&\bE[\min(1,|S_n-S_m|^2)|(\Gamma_i)=(\gamma_i),(T_i)=(t_i),(X_i)=(x_i)]\\
&\quad =\bE\left[\min\left(1,\left|\sum_{k=m+1}^n T^{k/\alpha}  k! \sum_{j_1<\ldots<j_{k}}
\prod_{\ell=1}^{k} \e_{j_{\ell}} \gamma_{j_{\ell}}^{-1/\alpha} \phi^{-1}(x_{j_{\ell}})  \widetilde{f}_n(t_{j_1},x_{j_1},\ldots,t_{j_{k}},x_{j_{k}},t,x) \right|^2 \right) \right]\\
& \quad \leq \min \left\{ 1,\bE\left|\sum_{k=m+1}^n T^{k/\alpha}  k! \sum_{j_1<\ldots<j_k}
\prod_{\ell=1}^{k} \e_{j_{\ell}} \gamma_{j_{\ell}}^{-1/\alpha} \phi^{-1}(x_{j_{\ell}})  \widetilde{f}_n(t_{j_1},x_{j_1},\ldots,t_{j_k},x_{j_k},t,x) \right|^2\right \}\\
& \quad = \min \left\{ 1,\sum_{k=m+1}^n T^{2k/\alpha}  (k!)^2 \sum_{j_1<\ldots<j_k}
\prod_{\ell=1}^{k}  \gamma_{j_{\ell}}^{-1/\alpha} \phi^{-2}(x_{j_{\ell}})  \widetilde{f}_n^2(t_{j_1},x_{j_1},\ldots,t_{j_k},x_{j_k},t,x) \right \}.
\end{align*}
This means that
\[
\bE\Big[\min(1,|S_n-S_m|^2)|(\Gamma_i),(T_i),(X_i)\Big] \leq \min \left\{ 1,\sum_{k=m+1}^n h_{k}^{(2)}(t,x) \right\}.
\]
Hence, 
\begin{align*}
\bE[\min(1,|S_n-S_m|^2)] \leq \bE \left[\min \Big(1,\sum_{k=m+1}^n h_{k}^{(2)}(t,x) \Big) \right].
\end{align*}
By \eqref{cond-C2} and the dominated convergence theorem, the last term above converges to $0$ as $n,m\to \infty$. This proves that $(S_n)_{n\geq 1}$ is a Cauchy sequence in $L^0$. 
\end{proof}

The following result gives a criterion for verifying \eqref{cond-C2}.  Recall that $K_n^{(p)}(t,x)$ is defined by \eqref{def-K}.

\begin{proposition}
\label{prop-conv-series}
Let $t \in [0,T]$ and $x \in \bR^d$ be fixed. If there exists $p \in (\alpha,2]$ such that
\begin{equation}
\label{A1}
\sum_{n\geq 1}T^{(\frac{p}{\alpha}-1)n} K_n^{(p)}(t,x)\Big(\sum_{j\geq 1} \Gamma_j^{-p/\alpha}\Big)^n<\infty \quad \mbox{a.s.},
\end{equation}
then \eqref{cond-C2} holds, and consequently, $\sum_{n\geq 1}I_n\big(\widetilde{f}_n(\cdot,t,x)\big)$ converges in probability.
\end{proposition}

\begin{proof}
By the sub-additivity of the function $\varphi(x)=x^{p/2},x>0$, we have:
\begin{equation}
\label{h-p2}
\Big(h_n^{(2)}(t,x)\Big)^{p/2} \leq h_n^{(p)}(t,x),
\end{equation}
where
\[
h_n^{(p)}(t,x):= T^{pn/\alpha} (n!)^p\sum_{j_1<\ldots<j_n} \prod_{k=1}^{n}\Gamma_{j_k}^{-p/\alpha}\phi^{-p}(X_{j_k})  \widetilde{f}_n^{p}(T_{j_1},X_{j_1},\ldots,T_{j_n},X_{j_n},t,x).
\]
Hence,
\[
\cA_2^{p/2}(t,x) \leq \sum_{n\geq 1}\Big(h_n^{(2)}(t,x)\Big)^{p/2}  \leq \sum_{n\geq 1} h_n^{(p)}(t,x)=: \cA_p(t,x).
\]


We will prove that $\cA_p(t,x)<\infty$ a.s. For this, we apply Remark \ref{Fubini-lemma} to $X=(\Gamma_i)_{i\geq 1}$ and $Y=\{(T_i,X_i)\}_{i\geq 1}$. We will show that:
\begin{equation}
\label{EAp-finite}
\bE[\cA_p(t,x)|(\Gamma_i)]<\infty \quad \mbox{a.s.}
\end{equation}

By the independence between $(\Gamma_i)_i$ and $\{(T_i,X_i)\}_{i}$, we have:
\[
\bE\Big[h_n^{(p)}(t,x) |(\Gamma_i)\Big] = T^{pn/\alpha} (n!)^p \sum_{j_1<\ldots<j_n}\prod_{k=1}^{n}\Gamma_{j_k}^{-p/\alpha}\bE\Big[\prod_{k=1}^{n}\phi^{-p}(X_{j_k}) \widetilde{f}_n^p(T_{j_1},X_{j_1},\ldots,T_{j_n},X_{j_n},t,x)\Big].
\]

For any $j_1<\ldots<j_n$ fixed, we denote
\begin{equation}
\label{def-Aj1}
A_{j_1,\ldots,j_n}^{(p)}(t,x):=(n!)^p \bE\left[\prod_{k=1}^{n}\phi^{-p}(X_{j_k}) \widetilde{f}_n^p(T_{j_1},X_{j_1},\ldots,T_{j_n},X_{j_n},t,x)\right].
\end{equation}
We now use the explicit expression of $A_{j_1,\ldots,j_n}^{(p)}(t,x)$, given by Lemma \ref{A-K-lemma} below, which in particular, shows that $A_{j_1,\ldots,j_n}^{(p)}(t,x)$ does not depend on $j_1,\ldots,j_n$.
It follows that
\begin{align}
\nonumber
\bE\Big[h_n^{(p)}(t,x) |(\Gamma_i)\Big] 
&=T^{(p/\alpha-1)n} K_n^{(p)}(t,x) \, n! \sum_{j_1<\ldots<j_n}\prod_{k=1}^{n}\Gamma_{j_k}^{-p/\alpha} \\
\label{hnp}
& \leq T^{(p/\alpha-1)n} K_{n}^{(p)}(t,x) \left( \sum_{j\geq 1}\Gamma_{j}^{-p/\alpha}\right)^{n},
\end{align}
where for the last inequality, we used the fact that for any $a_j>0$,
\begin{equation}
\label{fact2}
n!\sum_{j_1<\ldots<j_n}a_{j_1} \ldots a_{j_n} \leq \Big(\sum_{j\geq 1}a_j\Big)^n.
\end{equation}
By the strong law of large numbers, $\Gamma_j/j \to 1$ a.s., and hence $\sum_{j\geq 1}\Gamma_j^{-p/\alpha}<\infty$ a.s. Finally, \eqref{EAp-finite} follows by \eqref{A1}, since
\begin{align*}
\bE[\cA_p(t,x)|(\Gamma_i)] & =\sum_{n\geq 1}\bE\Big[h_n^{(p)}(t,x) |(\Gamma_i)\Big] \leq \sum_{n\geq 1} T^{(p/\alpha-1)n} K_{n}^{(p)}(t,x) \left( \sum_{j\geq 1}\Gamma_{j}^{-p/\alpha}\right)^{n}.
\end{align*}
\end{proof}

The following lemma gives the explicit expression for $A_{j_1,\ldots,j_n}^{(p)}(t,x)$.

\begin{lemma}
\label{A-K-lemma}
For any $t\in [0,T]$, $x \in \bR^d$, $p>0$ and $j_1<\ldots<j_n$, we have:
\begin{equation}
\label{A-K}
A_{j_1,\ldots,j_n}^{(p)}(t,x)=\frac{n!}{T^n} K_n^{(p)}(t,x).
\end{equation}
\end{lemma} 
 
\begin{proof}
By the definition of $\widetilde{f}_n(\cdot,t,x)$,
\begin{align}
\nonumber
A_{j_1,\ldots,j_n}^{(p)}(t,x)&=\bE\left[\prod_{k=1}^{n}\phi^{-p}(X_{j_k}) \Big(\sum_{\pi \in \Sigma_n} f_n(T_{j_{\pi(1)}},X_{j_{\pi(1)}},\ldots,T_{j_{\pi(n)}},X_{j_{\pi(n)}},t,x)\Big)^p\right]\\
\label{A-term}
&=\bE \left[ \left( \sum_{\pi \in \Sigma_n} \frac{f_n(T_{j_{\pi(1)}},X_{j_{\pi(1)}},\ldots,T_{j_{\pi(n)}},X_{j_{\pi(n)}},t,x)}{\phi(X_{j_{\pi(1)}}) \ldots \phi(X_{j_{\pi(n)}})}\right)^p \right].
\end{align}
We now present two methods leading to the desired relation \eqref{A-K}.

{\em Method 1.} (due to G. Samorodnitsky)
Recall that $f_n(T_{j_{\pi(1)}},X_{j_{\pi(1)}},\ldots,T_{j_{\pi(n)}},X_{j_{\pi(n)}},t,x)$ contains the indicator of
\begin{equation}
\label{order}
0<T_{j_{\pi(1)}}<\ldots<T_{j_{\pi(n)}}<t.
\end{equation}
For fixed values $T_{j_1},\ldots,T_{j_n} \in [0,t]$, there is a unique permutation $\pi$ for which \eqref{order} holds.
Therefore, the sum over all permutations above in fact contains {\em only one term}, all the other terms vanishing. The vector
$(T_{j_{\pi(1)}},X_{j_{\pi(1)}},\ldots,T_{j_{\pi(n)}},X_{j_{\pi(n)}})$ corresponding to that term
has the same distribution as $(T_{(1)},X_1,\ldots,T_{(n)},X_n)$, where $T_{(1)}<\ldots<T_{(n)}$ are the order statistics of $T_1,\ldots,T_n$. Hence,
\begin{align*}
A_{j_1,\ldots,j_n}^{(p)}(t,x)&=\bE \left[ \prod_{k=1}^{n}\phi^{-p}(X_k) f_n^p(T_{(1)},X_{1},\ldots,T_{(n)},X_n,t,x) \right]\\
&=\frac{n!}{T^n} \int_{T_n(T)} \int_{(\bR^d)^n} f_n^p(t_1,x_1,\ldots,t_n,x_n,t,x) \prod_{k=1}^{n} \phi^{\alpha-p}(x_k) dx_1 \ldots dx_n  dt_1 \ldots dt_n,
\end{align*}
using the fact that $(T_{(1)},\ldots,T_{(n)})$ has a uniform distribution over the simplex $T_n(T)$. The desired relation follows since $f_n(t_1,x_1,\ldots,t_n,x_n,t,x)=0$ if $t_n>t$.

{\em Method 2.}
From \eqref{A-term}, since $(T_{j_1},X_{j_1},\ldots,T_{j_n},X_{j_n})$ has density $T^{-n} \prod_{k=1}^{n}\phi^{\alpha}(x_k)$, we obtain:
\begin{align*}
& A_{j_1,\ldots,j_n}^{(p)}(t,x)=T^{-n}\int_{[0,T]^n} \int_{(\bR^d)^n}  \left( \sum_{\pi \in \Sigma_n} \frac{f_n(t_{\pi(1)},x_{\pi(1)},\ldots,t_{\pi(n)},
x_{\pi(n)},t,x)}{\phi(x_{\pi(1)}) \ldots \phi(x_{\pi(n)})}\right)^p \prod_{k=1}^{n}\phi^{\alpha}(x_k)  d\pmb{x} d\pmb{t} \\
&=T^{-n}\sum_{\rho \in \Sigma_n}
\int_{0<t_{\rho(1)}<\ldots<t_{\rho(n)}<t}\int_{(\bR^d)^n}  \left( \sum_{\pi \in \Sigma_n} \frac{f_n(t_{\pi(1)},x_{\pi(1)},\ldots,t_{\pi(n)},
x_{\pi(n)},t,x)}{\phi(x_{\pi(1)}) \ldots \phi(x_{\pi(n)})}\right)^p \prod_{k=1}^{n}\phi^{\alpha}(x_k)  d\pmb{x} d\pmb{t}\\
&=T^{-n}\sum_{\rho \in \Sigma_n}
\int_{0<t_{\rho(1)}<\ldots<t_{\rho(n)}<t}\int_{(\bR^d)^n}  \left(  \frac{f_n(t_{\rho(1)},x_{\rho(1)},\ldots,t_{\rho(n)},
x_{\rho(n)},t,x)}{\phi(x_{\rho(1)}) \ldots \phi(x_{\rho(n)})}\right)^p \prod_{k=1}^{n}\phi^{\alpha}(x_k)  d\pmb{x} d\pmb{t}\\
&=T^{-n}n!
\int_{0<t_{1}<\ldots<t_{n}<t}\int_{(\bR^d)^n}  \left(  \frac{f_n(t_{1},x_{1},\ldots,t_{n},
x_{n},t,x)}{\phi(x_{1}) \ldots \phi(x_{n})}\right)^p \prod_{k=1}^{n}\phi^{\alpha}(x_k)  d\pmb{x} d\pmb{t}.
\end{align*}
\end{proof}

The following result shows that a condition stronger than \eqref{cond-C2} implies
 the a.s. absolute convergence of the series \eqref{def-u}, and thus proves 
 Theorem \ref{main-th}.(a). The fact that the series converges {\em absolutely} will play a 
 crucial role in Section \ref{section-solvability} when we will show that the process $u(t,x)$ given by \eqref{def-u} is indeed a solution of equation \eqref{eq}.

\begin{proposition}
\label{prop-abs-conv}
Let $t \in [0,T]$ and $x \in \bR^d$ be fixed.
If there exists $p\in (\alpha,2]$ such that \eqref{A2} holds,
then $\cV(t,x):=\sum_{n\geq 1}|I_n\big(\widetilde{f}_n(\cdot,t,x) \big)|<\infty$ a.s.
\end{proposition}

\begin{proof} 
We apply Remark \ref{Fubini-remark} to $X=\{(\Gamma_i),(T_i),(X_i)\}$ and $Y=(\e_i)$. We will prove that:
\[
\bE\Big[\cV(t,x)|(\Gamma_i),(T_i),(X_i)\Big]<\infty \quad \mbox{a.s.}
\]
By Jensen's inequality for conditional expectation and \eqref{hn2}, we have 
\[
\bE[|I_n\big(\widetilde{f}_n(\cdot,t,x) \big)| \, |(\Gamma_i),(T_i),(X_i)]\leq \Big(h_n^{(2)}(t,x)\Big)^{1/2},
\]
 and hence $
\bE\Big[\cV(t,x)|(\Gamma_i),(T_i),(X_i)\Big] \leq \sum_{n\geq 1} \Big(h_n^{(2)}(t,x)\Big)^{1/2}$.
We take power $p/2$. Using sub-additivity of the function $\varphi(x)=x^{p/2},x>0$ and \eqref{h-p2}, we get:
\[
\left(\bE\Big[\cV(t,x)|(\Gamma_i),(T_i),(X_i)\Big] \right)^{p/2}\leq \sum_{n\geq 1}\Big(h_n^{(2)}(t,x)\Big)^{p/4} \leq  \sum_{n\geq 1}\Big(h_n^{(p)}(t,x)\Big)^{1/2}=:\cB_p(t,x).
\]
To prove that $\cB_p(t,x)<\infty$ a.s., we use Remark \ref{Fubini-remark} with $X=(\Gamma_i)$ and $Y=\{(T_i),(X_i)\}$. We will show that:
\begin{equation}
\label{Vptx}
\bE[\cB_p(t,x)|(\Gamma_i)]<\infty \quad \mbox{a.s.}
\end{equation} We will use the following inequality:
\begin{equation}
\label{Jensen1}
\bE(X^p | \cG)  \leq \left(\bE[X|\cG]\right)^p \quad \mbox{if $p\in (0,1)$ and $X\geq 0$},
\end{equation}
which is Jensen's inequality applied to the concave function $\varphi(x)=x^p,x>0$. Hence,
\begin{align*}
\bE[\cB_p(t,x)|(\Gamma_i)]&=\sum_{n\geq 1}\bE\Big[\Big(h_n^{(p)}(t,x)\Big)^{1/2}|(\Gamma_i)\Big] \leq \sum_{n\geq 1}\left(\bE\Big[h_n^{(p)}(t,x) |(\Gamma_i)\Big]\right)^{1/2}\\
&\leq \sum_{n\geq 1}
\left(T^{(p/\alpha-1)n} K_{n}^{(p)}(t,x) \right)^{1/2} \Big( \sum_{j\geq 1}\Gamma_{j}^{-p/\alpha}\Big)^{n/2},
\end{align*}
where the last inequality is due to \eqref{hnp}. Relation \eqref{Vptx} follows by \eqref{A2}.
\end{proof}

\section{The recurrence relation}
\label{section-recurrence}

In this section, 
 we prove that the partial sum sequence $(u_n)_{n \geq 0}$ given by
\begin{equation}
\label{def-partial}
u_0(t,x)=1 \quad \mbox{and} \quad u_n(t,x)=1+\sum_{k=1}^{n}I_k(f_k(\cdot,t,x)), \ n\geq 1
\end{equation}
is indeed the Picard's iteration sequence, i.e. it satisfies the recurrence relation \eqref{picard}. 

\medskip

By linearity, instead of proving \eqref{picard}, it is enough to show that for any $n\geq 0$,
\[
I_{n+1}(f_{n+1}(\cdot,t,x))=\int_0^t \int_{\bR^d}G_{t-s}(x-y)I_n(f_n(\cdot,s,y))Z(ds,dy).
\]
where $f_0(s,y)=1$ and $I_0(x)=x$. For this,
we define {\em the multiple integral process}:
$$X_n^{(t,x)}(s,y)=G_{t-s}(x-y)I_n(f_n(\cdot,s,y)), \quad s \in [0,T],y \in \bR^d.$$

For any $n\geq 0$, $t \in [0,T]$ and $x \in \bR^d$ fixed, we have to prove that: \\
(i) the multiple integral process is integrable wr.r.t. $Z$, i.e.
\begin{equation}
\label{Xn-integr}
X_n^{(t,x)} \in L^0(Z);
\end{equation}
(ii) the integral w.r.t. $Z$ of the multiple integral process coincides with $I_{n+1}(f_{n+1}(\cdot,t,x))$:
\begin{equation}
\label{IZ-X}
I^{Z}(X_n^{(t,x)})=I_{n+1}(f_{n+1}(\cdot,t,x)).
\end{equation}
These facts will be proved separately in the following two sections.

\subsection{Integrability of the muliple integral process}
\label{subsection-Picard}

In this section, we give the proof of \eqref{Xn-integr}. By Corollary \ref{integr-th}, we need to check that 
\begin{equation}
\label{def-cIn}
\cI_n(t,x):=\int_0^t \int_{\bR^d}G_{t-s}^{\alpha}(x-y)|I_n(f_n(\cdot,s,y))|^{\alpha}dyds<\infty \quad \mbox{a.s.},
\end{equation}
provided that $X_n^{(t,x)}$ is predictable. The next lemma addresses the issue of predictability.


\begin{lemma}
The process $\{I_n(\widetilde{f}_n(\cdot,t,x));t\in [0,T],x\in \bR^d\}$ has a predictable modification. Using this modification, $X_n^{(t,x)}$ is predictable.
\end{lemma}

\begin{proof}
We only need to prove the first statement, since the map $(\omega,s,y) \mapsto G_{t-s}(x-y)$ is clearly predictable. By LePage representation \eqref{LePage-In}, $I_n(f_n(\cdot,t,x))=S_n(t,x)$ a.s. where
\begin{align*}
S_n(t,x)&= T^{n/\alpha} n!  \sum_{j_1<\ldots<j_n}
\prod_{k=1}^{n} \e_{j_k} \Gamma_{j_k}^{-1/\alpha}  \phi^{-1}(X_{j_k}) \widetilde{f}_n(T_{j_1},X_{j_1},\ldots,T_{j_n},X_{j_n},t,x).
\end{align*}

Hence, it is enough to prove that $S_n$ has a predictable modification. For this, we proceed as in the proof of Lemma 6.2 of \cite{chong17-JTP}. For any $x \in \bR^d$, let 
$S_n^P(\cdot,x)$ be the {\em extended predictable projection} of $S_n(\cdot,x)$, given by Theorem I.2.28 of \cite{JS}, i.e. 
$S_n^P(\cdot,x)$
is $\cP$-measurable and $S_n^P(\cdot,x)=\bE[S_n(t,x)|\cF_{t-}]$
a.s. for all $t\in [0,T]$.
By Proposition 3 of \cite{stricker-yor78}, there exists a predictable process $\widetilde{S}_n$ such that 
$\widetilde{S}_n(t,x)= S_n^P (t,x)$ a.s. for all $(t,x)$. Hence, for all $(t,x)\in [0,T] \times \bR^d$, 
\[
\widetilde{S}_n(t,x)=\bE[S_n(t,x)|\cF_{t-}]=S_n(t,x) \quad \mbox{a.s.},
\]
where the last equality is due to the fact that $S_n(t,x)$ is $\cF_{t-}$-measurable, which is true because $S_n(t,x)$ is a function of the points of $N_{\psi}$ situated in $[0,t) \times \bR^d \times \bR_0$. (Recall that $(\cF_t)_{t \in [0,T]}$ is the filtration associated with $N_{\psi}$, given by \eqref{def-filtration}.)
\end{proof}

We continue now with the verification of \eqref{def-cIn}. 
This will be the consequence of the following more general result, which will be needed for the proof of Theorem 
\ref{solvability-th} below.

\begin{theorem}
\label{A3-th}
Suppose that Hypothesis \ref{G_assumption} holds. Let $t \in [0,T]$ and $x \in \bR^d$ be fixed. If there exists $p \in (\alpha,2]$ such that \eqref{A3} holds, then
\begin{equation}
\label{sum-cI}
\sum_{n\geq 1}\cI_n(t,x)^{\frac{1}{\alpha \vee 1}}<\infty \quad \mbox{a.s.}
\end{equation}
Moreover, if there exists $p \in (\alpha,2]$ such that
\begin{equation}
\label{cI}
\int_0^t \int_{\bR^d} G_{t-s}^{\alpha}(x-y) K_n^{(p)}(s,y) dyds<\infty,
\end{equation}
then $\cI_n(t,x)<\infty$ a.s., and consequently $X_n^{(t,x)} \in L^0(Z)$.
\end{theorem}

\begin{proof}
We only need to prove the first statement. The second statement is proved in the same way, dropping the sum over $n \geq 1$.

{\em We consider first the case $\alpha \leq 1$.} We will prove that:
\[
\mathcal{C}(t,x):=\sum_{n\geq 1}\cI_n(t,x)<\infty \quad \mbox{a.s.}
\]
We apply Remark \ref{Fubini-remark} with $X=(\Gamma_i)$ and $Y=\{(T_i),(X_i),(\e_i)\}$. We will prove that:
\begin{equation}
\label{C-fin1}
\bE[\mathcal{C}(t,x)|(\Gamma_i)] <\infty \quad \mbox{a.s.}
\end{equation}
We use the fact that
$\bE[\mathcal{C}(t,x)|(\Gamma_i)]=\bE[\bE[\mathcal{C}(t,x)|(\Gamma_i),(T_i),(X_i)]|(\Gamma_i)]$,
and we estimate separately the inner conditional expectation. By Jensen's inequality, 
\begin{align*}
\bE[\mathcal{C}(t,x)|(\Gamma_i),(T_i),(X_i)] &=\sum_{n\geq 1}\int_0^t \int_{\bR^d} G_{t-s}^{\alpha}(x-y) 
\bE[|I_n(f_n(\cdot,s,y))|^{\alpha}|(\Gamma_i),(T_i),(X_i)]dyds\\
&  \leq \sum_{n\geq 1}\int_0^t \int_{\bR^d} G_{t-s}^{\alpha}(x-y) 
\left(h_n^{(2)}(s,y)\right)^{\alpha/2}dyds.
\end{align*}

We now use the following form of Jensen's inequality: if $(E,\mathcal{E},\mu)$ is a finite measure space with $\mu(E)=a$, then for any measurable function $f:E \to \bR_{+}$,
\begin{equation}
\label{Jensen2}
\int_{E}f^{p}d\mu \leq a^{1-p}\left(\int_{E}fd\mu\right)^{p} \quad \mbox{for any $p\in (0,1]$}.
\end{equation}
We apply this inequality to the finite measure $\mu(ds,dy)=G_{t-s}^{\alpha}(s-y)1_{(0,t)}(s)dsdy$ (whose total mass we denote $C_t$), and the exponent $p'=\alpha/p$. We obtain:
\[
\int_0^t \int_{\bR^d} G_{t-s}^{\alpha}(x-y) 
\left(h_n^{(2)}(s,y)\right)^{\frac{\alpha}{2}}dyds \leq C_t^{1-\frac{\alpha}{p}}
\left(\int_0^t \int_{\bR^d}G_{t-s}^{\alpha}(x-y) \left(h_{n}^{(2)}(s,y)\right)^{\frac{p}{2}}dyds\right)^{\frac{\alpha}{p}}.
\]

It follows that 
\begin{align*}
\bE[\mathcal{C}(t,x)|(\Gamma_i)] & \leq C_{t}^{1-\frac{\alpha}{p}}\sum_{n\geq 1}\bE \left[
\left(\int_0^t \int_{\bR^d}G_{t-s}^{\alpha}(x-y) \left(h_{n}^{(2)}(s,y)\right)^{\frac{p}{2}}dyds\right)^{\frac{\alpha}{p}} \Bigg|(\Gamma_i)\right]\\
&\leq C_{t}^{1-\frac{\alpha}{p}}\sum_{n\geq 1} \left(\int_0^t \int_{\bR^d}G_{t-s}^{\alpha}(x-y) 
\bE[\big(h_n^{(2)}(s,y)\big)^{p/2}|(\Gamma_i)]dyds \right)^{\frac{\alpha}{p}},
\end{align*}
using \eqref{Jensen1} (with exponent $p'=\alpha/p$) for the last inequality. We now pass from 
$(h_n^{(2)}(s,y))^{p/2}$ to $h_n^{(p)}(s,y)$ (using inequality \eqref{h-p2}), and estimate
$\bE[h_n^{(p)}(s,y)|(\Gamma_i)]$ using \eqref{hnp}. We get:
\begin{align*}
\bE[\mathcal{C}(t,x)|(\Gamma_i)] & \leq C_{t}^{1-\frac{\alpha}{p}} 
\sum_{n\geq 1} \left(T^{(p/\alpha-1)n}\big(\sum_{j\geq 1}\Gamma_j^{-p/\alpha}\big)^n \int_0^t \int_{\bR^d}G_{t-s}^{\alpha}(x-y) K_n^{(p)}(s,y) dyds \right)^{\frac{\alpha}{p}}.
\end{align*}
The last series converges with probability $1$, due to \eqref{A3}. This proves \eqref{C-fin1}.

\medskip

{\em Next, we consider the case $\alpha >1$.} We will prove that:
\[
\mathcal{C}'(t,x):=\sum_{n\geq 1}\big(\cI_n(t,x)\big)^{\frac{1}{\alpha}}<\infty \quad \mbox{a.s.}
\]
By the same application of Remark \ref{Fubini-remark} as above, it suffices to prove that:
\begin{equation}
\label{C-fin2}
\bE[\mathcal{C}'(t,x)|(\Gamma_i)] <\infty \quad \mbox{a.s.}
\end{equation}
We use again double conditioning relation, and we estimate separately the inner conditional expectation. By Jensen's inequality \eqref{Jensen1},
\begin{align*}
\bE[\mathcal{C}'(t,x)|(\Gamma_i),(T_i),(X_i)] & =\sum_{n\geq 1}\bE \Big[ \big(\cI_n(t,x)\big)^{\frac{1}{\alpha}}|(\Gamma_i),(T_i),(X_i) \Big] \\
& \leq \sum_{n\geq 1}\left(\bE \Big[\cI_n(t,x)|(\Gamma_i),(T_i),(X_i) \Big]\right)^{\frac{1}{\alpha}}\\
&=\sum_{n\geq 1} \left(\int_0^t \int_{\bR^d}G_{t-s}^{\alpha}(x-y) \bE[|I_n(f_n(\cdot,s,y))|^{\alpha}|
(\Gamma_i),(T_i),(X_i)] dyds \right)^{\frac{1}{\alpha}}\\
& \leq \sum_{n\geq 1} \left(\int_0^t \int_{\bR^d}G_{t-s}^{\alpha}(x-y) \big(h_n^{(2)}(s,y)\big)^{\alpha/2} dyds \right)^{\frac{1}{\alpha}}.
\end{align*}
Conditioning on $(\Gamma_i)$, and using again the conditional Jensen's inequality \eqref{Jensen1} to push the power $1/\alpha$ outside the conditional expectation, we get:
\begin{align*}
\bE[\mathcal{C}'(t,x)|(\Gamma_i)]& \leq \sum_{n\geq 1} \bE\left[ \left(\int_0^t \int_{\bR^d}G_{t-s}^{\alpha}(x-y) \big(h_n^{(2)}(s,y)\big)^{\alpha/2} dyds  \right)^{\frac{1}{\alpha}} \Bigg| (\Gamma_i)\right]\\
&\leq \sum_{n\geq 1} \left(\int_0^t \int_{\bR^d}G_{t-s}^{\alpha}(x-y) \bE[\big(h_n^{(2)}(s,y)\big)^{\alpha/2}|(\Gamma_i)] dyds  \right)^{\frac{1}{\alpha}}.
\end{align*}
To arrive to the desired exponent $p/2$ for $h_n^{(2)}(s,y)$, we apply Jensen's inequality $\bE(|X||\cG) \leq \big(\bE[|X|^p|\cG]\big)^{1/p}$ for $p\geq 1$. 
Combining this with bound \eqref{h-p2}, we obtain:
\[
\bE[\big(h_n^{(2)}(s,y)\big)^{\alpha/2}|(\Gamma_i)]\leq \left(\bE[\big(h_n^{(2)}(s,y)\big)^{p/2}|(\Gamma_i)]\right)^{\frac{\alpha}{p}} \leq 
\left(\bE[h_n^{(p)}(s,y)|(\Gamma_i)]\right)^{\frac{\alpha}{p}}.
\]
Finally, we use \eqref{hnp} to estimate the last conditional expectation. Therefore,
\begin{align*}
\bE[\mathcal{C}'(t,x)|(\Gamma_i)] & \leq \sum_{n\geq 1} \left(\int_0^t \int_{\bR^d}G_{t-s}^{\alpha}(x-y)  \left(\bE[h_n^{(p)}(s,y)|(\Gamma_i)]\right)^{\frac{\alpha}{p}} dyds\right)^{\frac{1}{\alpha}} \\
& \leq \sum_{n\geq 1} \left(T^{(p/\alpha-1)n} \int_0^t \int_{\bR^d}G_{t-s}^{\alpha}(x-y) 
  K_n^{(p)}(s,y) dyds \right)^{\frac{1}{p}}\Big(\sum_{j\geq 1}\Gamma_j^{-p/\alpha}\Big)^{n/p}.
\end{align*}
Relation \eqref{C-fin2} follows by \eqref{A3}.
\end{proof}

\subsection{The integral of the multiple integral process: $\alpha \in (0,1)$}

In this section, we give the proof of \eqref{IZ-X} in the case $\alpha<1$. For this, an essential role is played by LePage representation \eqref{Lambda-LePage} of $\Lambda(A)$ for $A \in \widetilde{\cP}_b$, obtained in Section \ref{subsection-construction}.

\medskip

We start with a preliminary result, which gives the series representation for $I^Z(X_n^{(t,x)})$.

\begin{theorem}
\label{proof-IZ-Xn}
Assume that $\alpha \in (0,1)$. Let $t\in [0,T]$ and $x \in \bR^d$ be fixed. Suppose that there exists
$p\in (\alpha,1]$ such that $K_{n+1}^{(p)}(t,x)<\infty$, and $\cI_n(t,x)<\infty$ a.s. Then
\begin{equation}
\label{series-IZ}
I^{Z}(X_n^{(t,x)})=T^{1/\alpha}\sum_{i\geq 1}\e_i \Gamma_i^{-1/\alpha} \frac{1}{\phi(X_i)} G_{t-T_i}(x-X_i)I_n \big(f_n(\cdot,T_i,X_i)\big) \quad \mbox{a.s.}
\end{equation}
\end{theorem}

\begin{proof} 
Condition $\cI_n(t,x)<\infty$ a.s. ensures that $X_n^{(t,x)} \in L^0(Z)$.

\medskip

{\em Step 1.} In this step, we prove that the series on the right hand-side of \eqref{series-IZ} converges absolutely a.s., i.e.
\[
U:=\sum_{i\geq 1}\Gamma_i^{-1/\alpha}\phi^{-1}(X_i)  G_{t-T_i}(x-X_i)
\left|I_n \big(f_n(\cdot,T_i,X_i)\big)\right|<\infty \quad \mbox{a.s.}
\]
Since $p\leq 1$, by subadditivity,
\[
U^p \leq \sum_{i\geq 1}\Gamma_i^{-p/\alpha}\phi^{-p}(X_i)  G_{t-T_i}^p(x-X_i)
\left|I_n \big(f_n(\cdot,T_i,X_i)\big)\right|^p=:U'.
\]
So it is enough to prove that $U'<\infty$ a.s. 

Using the LePage representation \eqref{LePage-In}, the subadditivity of the function $\varphi(x)=|x|^p$, and the fact that $\widetilde{f}_n(t_1,x_1,\ldots,t_n,x_n,t,x)=0$ if $t_i=t$ for some $i=1,\ldots,n$, we have:
\begin{align*}
&|I_n(\widetilde{f}_n(\cdot,T_i,X_i))|^p  \leq T^{pn/\alpha} (n!)\sum_{j_1<\ldots<j_n} \prod_{k=1}^{n}\Gamma_{j_k}^{-p/\alpha}\phi^{-p}(X_{j_k})
\widetilde{f}_n^p(T_{j_1},X_{j_1},\ldots,T_{j_n},X_{j_n},T_i,X_i)\\
& \quad \quad \quad \leq T^{pn/\alpha} 
 \sum_{\substack{j_1<\ldots<j_n \\ i\not \in \{j_1,\ldots,j_n\} }}
\sum_{\pi \in \Sigma_n}\prod_{k=1}^{n}\Gamma_{j_k}^{-p/\alpha}\phi^{-p}(X_{j_k})
f_n^p(T_{j_{\pi(1)}},X_{j_{\pi(1)}},\ldots,T_{j_{\pi(n)}},X_{j_{\pi(n)}},T_i,X_i).
\end{align*}

We multiply this inequality by $\Gamma_i^{-p/\alpha}\phi^{-p}(X_i)G_{t-T_i}^p(x-X_i)$, then we take the sum for all $i\geq 1$. We use the fact that $ G_{t-T_i}^p(x-X_i) f_n^p(\cdot,T_i,X_i)=f_{n+1}^p
(\cdot,T_i,X_i,t,x)$. We obtain:
\begin{align*}
U' & \leq T^{pn/\alpha}\sum_{i\geq 1}\Gamma_i^{-p/\alpha}\phi^{-p}(X_i) \sum_{\substack{j_1<\ldots<j_n \\ i\not \in \{j_1,\ldots,j_n\} }} \sum_{\pi \in \Sigma_n} \prod_{k=1}^{n}\Gamma_{j_k}^{-p/\alpha}\phi^{-p}(X_{j_k})\\
& \qquad \qquad \qquad 
f_{n+1}^p
(T_{j_{\pi(1)}},X_{j_{\pi(1)}},\ldots,T_{j_{\pi(n)}},X_{j_{\pi(n)}},T_i,X_i,t,x).
\end{align*}

We now use the following fact: for any $a_{j_1,\ldots,j_n} \geq 0$,
\begin{equation}
\label{fact5}
\sum_{j_1<\ldots<j_n}\sum_{\pi \in \Sigma_n}a_{j_{\pi(1)},\ldots,j_{\pi(n)}}=
\sum_{ j_1,\ldots,j_n\geq 1 \ {\rm distinct} } a_{j_1,\ldots,j_n}.
\end{equation}

Hence,
\begin{align}
\nonumber
U' & \leq T^{pn/\alpha}\sum_{i\geq 1}\Gamma_i^{-p/\alpha}\phi^{-p}(X_i) 
\sum_{ \substack{j_1,\ldots,j_n\geq 1 \ {\rm distinct} \\ i \not\in \{j_1,\ldots,j_n\}  } } \, \, \prod_{k=1}^{n}\Gamma_{j_k}^{-p/\alpha}\phi^{-p}(X_{j_k})\\
\nonumber & \qquad \qquad \qquad 
f_{n+1}^p
(T_{j_{1}},X_{j_{1}},\ldots,T_{j_{n}},X_{j_{n}},T_i,X_i,t,x) \\
\nonumber & = T^{pn/\alpha}
\sum_{ j_1,\ldots,j_{n+1}\geq 1 \, {\rm distinct}    } \, \, \prod_{k=1}^{n+1}\Gamma_{j_k}^{-p/\alpha}\phi^{-p}(X_{j_k})
f_{n+1}^p
(T_{j_{1}},X_{j_{1}},\ldots,T_{j_{n+1}},X_{j_{n+1}},t,x)\\
\nonumber &= T^{pn/\alpha}
\sum_{ j_1<\ldots<j_{n+1}  }  \prod_{k=1}^{n+1}\Gamma_{j_k}^{-p/\alpha}\phi^{-p}(X_{j_k})
\sum_{\pi \in \Sigma_{n+1}}  f_{n+1}^p
(T_{j_{\pi(1)}},X_{j_{\pi(1)}},\ldots,T_{j_{\pi(n+1)}},X_{j_{\pi(n+1)}},t,x)\\
\label{def-B}
&=:  B_{n+1}^{(p)}(t,x).
\end{align}

We prove that $B_{n+1}^{(p)}(t,x)<\infty$ a.s. For this, we apply Remark \ref{Fubini-remark} with $X=(\Gamma_i)$ and $Y=\{(T_i,X_i)\}$. We will prove that
\[
\bE[B_{n+1}^{(p)}(t,x)| (\Gamma_i)]<\infty \quad \mbox{a.s.}
\]
Note that $(T_{j_{\pi(1)}},X_{j_{\pi(1)}},\ldots,T_{j_{\pi(n+1)}},X_{j_{\pi(n+1)}})$ has density
$T^{-(n+1)/\alpha}\prod_{k=1}^{n+1}\phi^{\alpha}(x_k)$. Hence,
\begin{align*}
&  \bE\left[ \prod_{k=1}^{n+1} \phi^{-p}(X_{j_k}) f_{n+1}^p
(T_{j_{\pi(1)}},X_{j_{\pi(1)}},\ldots,T_{j_{\pi(n+1)}},X_{j_{\pi(n+1)}},t,x)\right]
=\frac{ 1}{T^{(n+1)/\alpha}} K_{n+1}^{(p)}(t,x),
\end{align*}
and 
\begin{align*}
\bE[B_{n+1}^{(p)}(t,x)| (\Gamma_i)] &=  T^{pn/\alpha} \sum_{ j_1<\ldots<j_{n+1}  } \prod_{k=1}^{n+1}\Gamma_{j_k}^{-p/\alpha} \cdot  (n+1)! \frac{1}{T^{(n+1)/\alpha}}  K_{n+1}^{(p)}(t,x) \\
& \leq T^{pn/\alpha} \frac{ 1}{T^{(n+1)/\alpha}} K_{n+1}^{(p)}(t,x) \left(\sum_{j\geq 1}\Gamma_j^{-p/\alpha}\right)^{n+1}<\infty \ \mbox{a.s.}
\end{align*}

{\em Step 2.} Since $X_{n}^{(t,x)}$ is predictable, by Theorem 13.5 of \cite{billingsley95}, there exists a sequence $(S_k)_{k\geq 1}$ of simple integrands such that $S_k \to X_{n}^{(t,x)}$ as $k \to \infty$, and $|S_k| \leq |X_n^{(t,x)}|$ for all $k$. Note that $X_{n}^{(t,x)} \in L^p(Z)$, since $\cI_n(t,x)<\infty$ a.s. By Theorem \ref{DCT}, $I^Z(S_k) \stackrel{P}{\to} I^Z(X_n^{(t,x)})$ as $k \to \infty$. This convergence is a.s., along a subsequence.

Since $S_k$ is a linear combination of sets in $\widetilde{\cP}_b$, by the LePage representation \eqref{Lambda-LePage}, 
\begin{align*}
I^{Z}(S_k) & =I^{\Lambda}(S_k)
=\sum_{i\geq 1}\e_i \Gamma_i^{-1/\alpha} \frac{1}{\psi(T_i,X_i)} S_k(T_i,X_i).
\end{align*}
Relation \eqref{series-IZ} follows letting $k \to \infty$. On the right hand-side, we use the dominated convergence theorem, whose application is justified by {\em Step 1}.
\end{proof}

\begin{theorem}
\label{th-alpha1}
Assume that $\alpha \in (0,1)$. Let $t>0$ and $x \in \bR^d$ be arbitrary. Suppose that there exists $p \in (\alpha,1]$ such that $K_{n+1}^{(p)}(t,x)<\infty$ a.s. and $\cI_n(t,x)<\infty$ a.s. Then \eqref{IZ-X} holds.
\end{theorem}

\begin{proof}
We use \eqref{series-IZ}, in which we replace $I_n\big(f_n(\cdot,T_i,X_i)\big)$ by its LePage representation \eqref{LePage-In}. Using the fact that $\widetilde{f}_n(T_{j_1},X_{j_1},\ldots, T_{j_n},X_{j_n},T_i,X_i)=0$ if $i \in \{j_1,\ldots,j_n\}$, we have:
\begin{align*}
I^{Z}(X_n^{(t,x)})&= T^{\frac{n+1}{\alpha}}\sum_{i\geq 1}\e_i \Gamma_i^{-1/\alpha} 
\phi^{-1}(X_i)G_{t-T_i}(x-X_i) \sum_{ \substack{j_1<\ldots < j_n \\ i \not\in \{j_1,\ldots,j_n\}  } }\prod_{k=1}^{n}\e_{j_k}\Gamma_{j_k}^{-1/\alpha}\phi^{-1}(X_{j_k})\\
& \quad \quad  \quad \sum_{\pi \in S_n} f_n(T_{j_{\pi(1)}},X_{j_{\pi(1)}},\ldots, T_{j_{\pi(n)}},X_{j_{\pi(n)}},T_i,X_i).
\end{align*}
We now use the fact that $G_{t-T_i}(x-X_i)f_n(\cdot,T_i,X_i)= f_{n+1}(\cdot,T_i,X_i,t,x)$, and we apply  \eqref{fact5} two times. We obtain:
\begin{align*}
& I^{Z}(X_n^{(t,x)})=T^{\frac{n+1}{\alpha}}\sum_{i\geq 1}\e_i \Gamma_i^{-1/\alpha} \phi^{-1}(X_i) \sum_{ \substack{j_1,\ldots, j_n \geq 1 \, {\rm distinct} \\ i \not\in \{j_1,\ldots,j_n\}  }} \, \prod_{k=1}^{n}\e_{j_k}\Gamma_{j_k}^{-1/\alpha}\phi^{-1}(X_{j_k})\\
& \qquad \qquad \qquad f_{n+1}(T_{j_1},X_{j_1},\ldots,T_{j_n},X_{j_n},T_i,X_i,t,x)\\
& \quad =T^{\frac{n+1}{\alpha}}
\sum_{ \substack{j_1,\ldots, j_{n+1} \geq 1 \, {\rm distinct}  }} \, \prod_{k=1}^{n+1}\e_{j_k}\Gamma_{j_k}^{-1/\alpha}\phi^{-1}(X_{j_k}) f_{n+1}(T_{j_1},X_{j_1},\ldots,T_{j_{n+1}},X_{j_{n+1}},t,x)\\
& \quad =T^{\frac{n+1}{\alpha}} \sum_{j_1<\ldots<j_{n+1}}  \sum_{\pi \in \Sigma_{n+1}}
\prod_{k=1}^{n+1} \e_{j_k} \Gamma_{j_k}^{-1/\alpha}\phi^{-1}(X_{j_k})f_{n+1}(T_{j_{\pi(1)}}, X_{j_{\pi(1)}},\ldots,T_{j_{\pi(n+1)}}, X_{j_{\pi(n+1)}},t,x)\\
&\quad =T^{\frac{n+1}{\alpha}} (n+1)! \sum_{j_1<\ldots<j_{n+1}} 
\prod_{k=1}^{n+1} \e_{j_k} \Gamma_{j_k}^{-1/\alpha} \phi^{-1}(X_{j_k})
\widetilde{f}_{n+1} (T_{j_{1}}, X_{j_{1}},\ldots,T_{j_{n+1}}, X_{j_{n+1}},t,x)\\
& \quad = I_{n+1}\big( \widetilde{f}_{n+1}(\cdot,t,x)\big).
\end{align*}

\end{proof}

\subsection{The integral of the multiple integral process: $\alpha \in [1,2)$}

In this section, we give the proof of \eqref{IZ-X} in the case $\alpha \geq1$. The proof is significantly more involved than in the case $\alpha<1$, being the most technical part of the paper.
The reason is that in the case $\alpha \geq 1$, we are not able to prove a LePage representation for $\Lambda(A)$ for all $A \in \widetilde{\cP}_b$, as it was mentioned in Remark \ref{rem-LePage}.

\medskip

We fix $k \geq 1$. In the LePage representation \eqref{LePage-In} of $I_n\big(\widetilde{f}_n(\cdot,s,y)\big)$, in front of $\prod_{i=1}^n$, we introduce the factor $1=1_{\{\Gamma_{j_n}^{-1/\alpha}>k^{-1}\}}+1_{\{\Gamma_{j_n}^{-1/\alpha}\leq k^{-1} \}}$. Since $\Gamma_{j_1}^{-1/\alpha}>\ldots>\Gamma_{j_n}^{-1/\alpha}$, we see that
$\{\Gamma_{j_n}^{-1/\alpha}>k^{-1}\}=\bigcap_{j=1}^{n}\{\Gamma_{j_i}^{-1/\alpha}>k^{-1}\}$, 
We obtain the decomposition:
\begin{equation}
\label{decomp-In}
I_n\big( \widetilde{f}_n(\cdot, s,y)\big)=\cJ_n^{(k)}(s,y)+\cR_n^{(k)}(s,y),
\end{equation}
where
\begin{align*}
\cJ_n^{(k)}(s,y)&=T^{n/\alpha}n! \sum_{j_1<\ldots<j_n} 
 \prod_{i=1}^{n}\e_{j_i}\Gamma_{j_i}^{-1/\alpha} 1_{\{\Gamma_{j_i}^{-1/\alpha}>k^{-1}\}} \phi^{-1}(X_{j_i}) \widetilde{f}_n(T_{j_1},X_{j_1},\ldots,
T_{j_n},X_{j_n},s,y),\\
\cR_n^{(k)}(s,y)&= T^{n/\alpha}n! \sum_{j_1<\ldots<j_n} 
1_{\{\Gamma_{j_n}^{-1/\alpha}\leq k^{-1}\}} \prod_{i=1}^{n}\e_{j_i}\Gamma_{j_i}^{-1/\alpha} \phi^{-1}(X_{j_i}) \widetilde{f}_n(T_{j_1},X_{j_1},\ldots,
T_{j_n},X_{j_n},s,y).
\end{align*}

The following result shows that $\cR_n^{(k)}(t,x)$ is asymptotically negligible in probability, when $k \to \infty$.

\begin{lemma} 
\label{lem-R0}
Assume that $\alpha \in [1,2)$.
Let $t \in [0,T]$ and $x \in \bR^d$ be arbitrary. If there exists $p \in (\alpha,2]$ such that $K_n^{(p)}(t,x)<\infty$, then $\cR_{n}^{(k)}(t,x) \stackrel{P}{\to} 0$ as $k \to \infty$.
\end{lemma}

\begin{proof}
We argue as in the proof of Lemma \ref{conv-proba}. By the Cauchy-Schwarz inequality,
\[
\|\cR_n^{(k)}(t,x)\|_{L^0} = \bE\big[\min\big(1,|\cR_n^k(t,x)|\big)\big] \leq \left(\bE\big[\min\big(1,|\cR_n^k(s,y)|^2\big)\big]\right)^{1/2}.
\]

By Lemma \ref{gK-ineq} and the inequality
$\bE[\min(1,|X|)]\leq \min(1,\bE|X|)$, we have:
\begin{align*}
& \bE\big[\min\big(1, |\cR_n^k(s,y)|^2 \big)\, \big|\, (\Gamma_i),(T_i),(X_i)\big] \leq \\
& \quad \min\Big(1, T^{2n/\alpha}(n!)^2 \sum_{j_1<\ldots<j_n} 1_{\{\Gamma_{j_n}^{-1/\alpha} \leq k^{-1}\} } \prod_{i=1}^{n}\Gamma_{j_i}^{-2/\alpha} \phi^{-2}(X_{j_i}) \widetilde{f}_n^2(T_{j_1},X_{j_1},\ldots,T_{j_n},X_{j_n},t,x) \Big).
\end{align*}
The last expression converges to $0$ as $k \to \infty$, by the dominated convergence theorem. To justify the application of this theorem, we bound the indicator above by $1$, which leads to $\min(1,h_n^{(2)}(t,x))$. To see that $h_n^{(2)}(t,x)<\infty$ a.s., we apply Remark \ref{Fubini-remark} to $X=(\Gamma_i)$ and $Y=\{(T_i),(X_i)\}$, noting that, by \eqref{h-p2} and \eqref{hnp},
\[
\bE[\big(h_n^{(2)}(t,x)\big)^{p/2}|(\Gamma_i)] \leq \bE[h_n^{(p)}(t,x)|(\Gamma_i)] \leq T^{(p/\alpha-1)n}K_n^{(p)}(t,x) \left(\sum_{j\geq 1}\Gamma_j^{-p/\alpha} \right)^n<\infty \quad \mbox{a.s.}
\]
Finally, another application of dominated convergence theorem shows that
\[
\bE[\min(1,|\cR_n^{(k)}(t,x)|^2)]=\bE\big[\bE\big[\min\big(1,|\cR_n^{(k)}(t,x)|^2 \big)\,| \, (\Gamma_i),(T_i),(X_i)\big] \big] \to 0.
\]
\end{proof}

Note that, if there exists $p \in (\alpha,2]$ such that \eqref{cI} holds, then the same argument as in the proof of Theorem \ref{A3-th} shows that
\[
\int_0^{t} \int_{\bR^2}G_{t-s}^{\alpha}(x-y) |\cJ_n^{(k)}(s,y)|^{\alpha}dyds<\infty \quad \mbox{a.s.}
\]
and
\[
\int_0^{t} \int_{\bR^2}G_{t-s}^{\alpha}(x-y) |\cR_n^{(k)}(s,y)|^{\alpha}dyds<\infty \quad \mbox{a.s.}
\]
Choosing predictable modifications for the processes $Y_{n,k}^{(t,x)}(s,y)=G_{t-s}(x-y) \cJ_n^{(k)}(s,y)$ and $Z_{n,k}^{(t,x)}(s,y)=G_{t-s}(x-y) \cR_n^{(k)} (s,y)$, we infer that these processes are in $L^0(Z)$, by Corollary \ref{integr-th}.
By decomposition \eqref{decomp-In},
\[
I^Z(X_n^{(t,x)})=\int_0^t \int_{\bR^d}G_{t-s}(x-y)\cJ_n^{(k)}(s,y)Z(ds,dy)+
\int_0^t \int_{\bR^d}G_{t-s}(x-y)\cR_n^{(k)}(s,y)Z(ds,dy).
\]
Hence, to prove that $I^Z(X_n^{(t,x)}) =I_{n+1}\big(f_{n+1}(\cdot,t,x)\big)$, it suffices to show that the first term converges in probability to $I_{n+1}\big(f_{n+1}(\cdot,t,x)\big)$ as $k \to \infty$, and the second one is negligible. 
This will be achieved by the following two lemmas.

\begin{lemma}
\label{lem-cR}
Assume that $\alpha \in [1,2)$ and Hypothesis \ref{G_assumption} holds.
Let $t \in [0,T]$ and $x \in \bR^d$ be arbitrary. If there exists $p \in (\alpha,2]$ such that $K_{n+1}^{(p)}(t,x)<\infty$ and \eqref{cI} holds, then
\[
A_k:=\int_0^t \int_{\bR^d}G_{t-s}(x-y)\cR_n^{(k)}(s,y)Z(ds,dy) \stackrel{P}{\longrightarrow} 0 \quad \mbox{as $k \to \infty$}.
\]
\end{lemma}

\begin{proof}
By Proposition \ref{prop-Lambda-repr4}, we have the decomposition:
\begin{align*}
A_k& = \int_0^t \int_{\bR^d} \int_{\{|z|\leq 1\}}G_{t-s}(x-y)\cR_n^{(k)}(s,y)\frac{z}{\psi(s,y)}\widehat{J}_{\psi}(ds,dy,dz)
+\\
& \quad
\int_0^t \int_{\bR^d} \int_{\{|z|> 1\}}G_{t-s}(x-y)\cR_n^{(k)}(s,y)\frac{z}{\psi(s,y)}J_{\psi}(ds,dy,dz)=:T_1^{(k)}+T_{2}^{(k)}.
\end{align*}
We treat separately the two terms.
Recall that $\psi(s,y)=T^{-1/\alpha} \phi(y)$. 

\medskip

{\bf Step 1.} \underline{First, we treat $T_2^{(k)}$.} Since $J_{\psi}$ has points $\{(T_j,X_j,\e_j \Gamma_j^{-1/\alpha})\}_{j\geq 1}$, we have
\begin{align*}
T_2^{(k)}
= \sum_{j\geq 1} \e_j \Gamma_j^{-1/\alpha} 1_{\{\Gamma_j^{-1/\alpha}>1\}} W_{n,j}^{(k)},
\end{align*}
with 
\begin{equation}
\label{def-W-nj}
W_{n,j}^{(k)}:=T^{1/\alpha}G_{t-T_j}(x-X_j)\cR_{n}^{(k)}(T_j,X_j) \phi^{-1}(X_j).
\end{equation}
 Using the fact that $\Gamma_j^{-1/\alpha} \leq \Gamma_j^{-2/\alpha}$ on the event $\{\Gamma_j^{-1/\alpha}>1\}$, we obtain:
\[
\big|T_2^{(k)}\big| \leq  \sum_{j\geq 1}\Gamma_j^{-2/\alpha} 1_{\{\Gamma_j^{-1/\alpha}>1\}} \big|W_{n,j}^{(k)}\big| \leq  \sum_{j\geq 1}\Gamma_j^{-2/\alpha} \big|W_{n,j}^{(k)}\big|.
\]

We use the following fact for any random variable $X$ and sub-$\sigma$-field $\mathcal{G}$,
\begin{equation}
\label{ineq-X0}
\|X\|_{L^0}\leq \big\| \bE[|X| \, |\mathcal{G}]\big\|_{L^0}.
\end{equation}
This leads to the following inequality:
\begin{equation}
\label{Tk-norm0}
\big\|T_2^{(k)}\big\|_{L^0} \leq  \big\| \sum_{j\geq 1}\Gamma_j^{-2/\alpha} \big|W_{n,j}^{(k)}\big|\, \big\|_{L^0} \leq \Big\|\sum_{j\geq 1}  \Gamma_j^{-2/\alpha}  \bE  \Big[ \big|W_{n,j}^{(k)}\big|\,\Big|(\Gamma_i)_i \Big]\Big\|_{L^0}.
\end{equation}

We fix $j\geq 1$. By H\"older's inequality for conditional expectation,
\[
\bE  \Big[ \big|W_{n,j}^{(k)}\big|\,\Big|(\Gamma_i)_i \Big] \leq 
\left(\bE  \Big[ \big|W_{n,j}^{(k)}\big|^{\alpha}\,\Big|(\Gamma_i)_i \Big] \right)^{1/\alpha}.
\]

Note that $\cR_{n}^{(k)}(T_j,X_j)$ is a series depending on multi-indices $j_1<\ldots<j_n$ and each term in this series contains the factor $\widetilde{f}_n(T_{j_1},X_{j_1},\ldots,T_{j_n},X_{j_n},T_j,X_j)$, which vanishes if $j \in \{j_1,\ldots,j_n\}$. So, we can assume that $j \not \in \{j_1,\ldots,j_n\}$. Moreover, $\cR_{n}^{(k)}(T_j,X_j)$ is a function of the sequences $(\e_i)_{i\geq 1},(\Gamma_i)_{i\geq 1},(T_i)_{i\geq 1},(X_i)_{i\geq 1}$.

We apply the basic conditioning fact \eqref{fact1} to $X=(\Gamma_i)_i$ and $Y=\{(\e_i)_i, (T_i)_i,(X_i)_i\}$. Using the definition of $\cR_n^{(k)}(T_j,X_j)$, we obtain:
\begin{align*}
& \bE\Big[\big|W_{j,n}^{(k)}\big|^{\alpha} \, \Big|(\Gamma_i)_i=(\gamma_i)_i\Big]=T\bE\Big[G_{t-T_j}^{\alpha}(x-X_j)\frac{1}{\phi^{\alpha}(X_j)}\\
& 
\Big|T^{n/\alpha}n! \sum_{ \substack{j_1<\ldots < j_n \\ j \not\in \{j_1,\ldots,j_n\}  } } 1_{\{\gamma_{j_n}^{-1/\alpha} \leq k^{-1} \}} \prod_{i=1}^{n} \e_{j_i} \gamma_{j_i}^{-1/\alpha}\phi^{-1}(X_{j_i}) \widetilde{f}_{n}(T_{j_1},X_{j_1},\ldots,T_{j_n},X_{j_n},T_{j},X_{j})\Big|^{\alpha}\Big].
\end{align*}

We use the fact that for independent random elements $X$ and $Y$ with values in measurable spaces $(E,\cE)$, respectively $(F,\cF)$, and a measurable function $f: E \times F \to [0,\infty]$,
\[
\bE[f(X,Y)]=\int_{E} \bE[f(x,Y)] \bP_{X}(dx)=\bE \left[\int_{E} f(x,Y)\bP_{X}(dx)\right],
\]
where $\bP_{X}$ denotes the law of $X$. We will use this fact with $X=(T_j,X_j)$ (which has law $T^{-1}\phi^{\alpha}(y)dsdy$) and $Y=\{(\e_i)_{i\geq 1},(T_i,X_i)_{i\not=j}\}$, to compute the previous expectation:
\begin{align*}
& \bE\Big[\big|W_{j,n}^{(k)}\big|^{\alpha} \, \Big|(\Gamma_i)_i=(\gamma_i)_i\Big]=T\, \bE \Bigg[\int_0^t \int_{\bR^d} G_{t-s}^{\alpha}(x-y)\frac{1}{\phi^{\alpha}(y)} \\
&  
\Big|T^{\frac{n}{\alpha}}n! \sum_{ \substack{j_1<\ldots < j_n \\ j \not\in \{j_1,\ldots,j_n\}  } } 1_{\{\gamma_{j_n}^{-1/\alpha} \leq k^{-1} \}} \prod_{i=1}^{n} \e_{j_i} \gamma_{j_i}^{-1/\alpha}\phi^{-1}(X_{j_i}) \widetilde{f}_{n}(T_{j_1},X_{j_1},\ldots,T_{j_n},X_{j_n},s,y)\Big|^{\alpha} T^{-1}\phi^{\alpha}(y)dyds \Bigg].
\end{align*}

Using again the basic conditioning fact \eqref{fact1} for the expectation above, we infer that
\begin{align*}
& \bE\Big[\big|W_{j,n}^{(k)}\big|^{\alpha} \, \Big|(\Gamma_i)_i= (\gamma_i)_i\Big]= \, \bE \left[\int_0^t \int_{\bR^d} G_{t-s}^{\alpha}(x-y) \right.\\
& \left.
\Big|T^{\frac{n}{\alpha}}n! \sum_{ \substack{j_1<\ldots < j_n \\ j \not\in \{j_1,\ldots,j_n\}  } } 1_{\{\Gamma_{j_n}^{-1/\alpha} \leq k^{-1} \}} \prod_{i=1}^{n} \e_{j_i} \Gamma_{j_i}^{-1/\alpha}\phi^{-1}(X_{j_i}) \widetilde{f}_{n}(T_{j_1},X_{j_1},\ldots,T_{j_n},X_{j_n},s,y)\Big|^{\alpha} 
\,\Big|(\Gamma_i)_i=(\gamma_i)_i \right].
\end{align*}
In summary, we have proved the following non-trivial fact:
\begin{equation}
\label{non-trivial}
A_{n,j}^{(k)}:=\bE\Big[\big|W_{j,n}^{(k)}\big|^{\alpha} \, \Big|(\Gamma_i)_i\Big]=\bE \left[
\int_0^t \int_{\bR^d}G_{t-s}^{\alpha}(x-y) |\cR_{n,j}^{(k)}(s,y)|^{\alpha} dyds\, \Big|(\Gamma_i)_i \right],
\end{equation}
where 
\[
\cR_{n,j}^{(k)}(s,y)=T^{n/\alpha}n! \sum_{ \substack{j_1<\ldots < j_n \\ j \not\in \{j_1,\ldots,j_n\}  } } 1_{\{\Gamma_{j_n}^{-1/\alpha} \leq k^{-1} \}} \prod_{i=1}^{n} \e_{j_i} \Gamma_{j_i}^{-1/\alpha}\phi^{-1}(X_{j_i}) \widetilde{f}_{n}(T_{j_1},X_{j_1},\ldots,T_{j_n},X_{j_n},s,y).
\]

\underline{We continue with the estimation of $A_{n,j}^{(k)}$.} Using a double conditioning argument for the term on the right hand side of \eqref{non-trivial}, we see that $A_{n,j}^{(k)}=\bE\big[Q_{n,j}^{(k)}\,\big|(\Gamma_i)_i\big]$, where
\[
Q_{n,j}^{(k)}:=\int_0^t \int_{\bR^d} G_{t-s}^{\alpha}(x-y) \bE \left[ |\cR_{n,j}^{(k)}(s,y)|^{\alpha}\, \Big| (\Gamma_i)_i, (T_i)_i, (X_i)_i\right] dyds.
\]
Using H\"older's inequality for conditional expectation, we have:
\[
 \bE \left[ |\cR_{n,j}^{(k)}(s,y)|^{\alpha}\, \Big| (\Gamma_i)_i, (T_i)_i, (X_i)_i\right]  \leq \left(\bE \left[ |\cR_{n,j}^{(k)}(s,y)|^{2}\, \Big| (\Gamma_i)_i, (T_i)_i, (X_i)_i\right] \right)^{\frac{\alpha}{2}}=: \Big(h_{n,j}^{(k)}(s,y)\Big)^{\frac{\alpha}{2}}.
\]
We obtain that:
\begin{align*}
Q_{n,j}^{(k)} & \leq \int_0^t \int_{\bR^d} G_{t-s}^{\alpha}(x-y) \Big(h_{n,j}^{(k)}(s,y)\Big)^{\frac{\alpha}{2}}dyds \\
& \leq C_t^{1-\frac{\alpha}{p}} \left( \int_0^t \int_{\bR^d} G_{t-s}^{\alpha}(x-y) \big(h_{n,j}^{(k)}(s,y) \big)^{p/2} dyds\right)^{\frac{\alpha}{p}},
\end{align*}
where for the last line we applied Jensen's inequality \eqref{Jensen2}, exactly as before. Hence,
\begin{align}
\nonumber
A_{n,j}^{(k)} & \leq C_{t}^{1-\frac{\alpha}{p}} \bE\left[\left( \int_0^t \int_{\bR^d} G_{t-s}^{\alpha}(x-y) \big(h_{n,j}^{(k)}(s,y) \big)^{p/2} dyds\right)^{\frac{\alpha}{p}}\, \Big|(\Gamma_i)_i\right]\\
\label{estim-Anj}
&\leq  C_{t}^{1-\frac{\alpha}{p}} \left( \int_0^t \int_{\bR^d} G_{t-s}^{\alpha}(x-y) \bE\big[\big(h_{n,j}^{(k)}(s,y)\big)^{p/2} \,\big|(\Gamma_i)_i\big]dyds \right)^{\alpha/p}
\end{align}
where for the last line we used Jensen's inequality \eqref{Jensen1}, and we switched the $dyds$ integral with the conditional expectation.

To continue the previous estimation, we need to evaluate $h_{n,j}^{(k)}(s,y)$, which is
in fact very similar to $h_n^{(2)}(s,y)$ (see \eqref{hn2}):
\begin{align*}
h_{n,j}^{(k)} (s.y)&=T^{2n/\alpha}(n!)^2 \sum_{ \substack{j_1<\ldots < j_n \\ j \not\in \{j_1,\ldots,j_n\}  } } 1_{\{\Gamma_{j_n}^{-1/\alpha} \leq k^{-1} \}} \prod_{i=1}^{n} \Gamma_{j_i}^{-2/\alpha}\phi^{-2}(X_{j_i}) \widetilde{f}_{n}^2(T_{j_1},X_{j_1},\ldots,T_{j_n},X_{j_n},s,y).
\end{align*}
We take power $p/2$, which we then move inside the sum, by sub-additivity. Dropping the restriction  $j \not\in \{j_1,\ldots,j_n\}$, and recalling definition \eqref{def-Aj1} of $A_{j_1,\ldots,j_n}^{(p)}(s,y)$, we obtain:
\begin{align*}
\bE\big[\big(h_{n,j}^{(k)} (s.y)\big)^{p/2}\, \big|(\Gamma_i)\big] & \leq T^{pn/\alpha} \sum_{j_1<\ldots < j_n} 1_{\{\Gamma_{j_n}^{-1/\alpha} \leq k^{-1} \}} \prod_{i=1}^{n} \Gamma_{j_i}^{-p/\alpha}A_{j_1,\ldots,j_n}^{(p)}(s,y) \\
& = n! T^{(p/\alpha-1)n} \sum_{j_1<\ldots < j_n} 1_{\{\Gamma_{j_n}^{-1/\alpha} \leq k^{-1} \}} \prod_{i=1}^{n} \Gamma_{j_i}^{-p/\alpha} K_n^{(p)}(s,y),
\end{align*}
where for the last line, we used the expression of $A_{j_1,\ldots,j_n}^{(p)}(s,y)$ given by Lemma \ref{A-K-lemma}. Plugging this into \eqref{estim-Anj}, we obtain:
\begin{align*}
A_{n,j}^{(k)} & \leq  C_t^{1-\frac{\alpha}{p}} \left(T^{(p/\alpha-1)n}n!\sum_{j_1<\ldots < j_n} 1_{\{\Gamma_{j_n}^{-1/\alpha} \leq k^{-1} \}} \prod_{i=1}^{n} \Gamma_{j_i}^{-p/\alpha} \right)^{\frac{\alpha}{p}} 
 \left( \int_0^t \int_{\bR^d}G_{t-s}^{\alpha}(x-y)K_n^{(p)}(s,y) dyds\right)^{\frac{\alpha}{p}}.
\end{align*}
By the dominated convergence theorem, the first factor converges to $0$ a.s. as $k\to \infty$. The second factor is a constant, that we denote $C_t'$. Therefore $A_{n,j}^{(k)} \to 0$ a.s. as $k \to \infty$. Moreover,  $A_{n,j}^{(k)} \leq C_t^{1-\frac{\alpha}{p}}
\big[T^{(p/\alpha-1)n} \big(\sum_{j\geq 1}\Gamma_j^{-p/\alpha}\big)^n \big]^{\alpha/p}C_t' $, due to inequality \eqref{fact2}.
By another application of the dominated convergence theorem,
\[
\sum_{j\geq 1}\Gamma_j^{-2/\alpha}A_{n,j}^{(k)} \to 0 \quad \mbox{a.s.} \quad \mbox{as $k \to \infty$}
\]
In particular, the previous sum converges in probability to $0$, as $k \to \infty$. Recalling  inequality \eqref{Tk-norm0}, we infer that $T_2^{(k)} \stackrel{P}{\to} 0$ as $k \to \infty$, which concludes {\em Step 1}.

\bigskip

{\bf Step 2.} \underline{Next, we treat $T_1^{(k)}$.} Note that for any $(t_0,x_0)$ fixed, the process
\[
M_t^{(t_0,x_0)}:=\int_0^t \int_{\bR^d}\int_{\{|z|\leq 1\}}G_{t_0-s}(x_0-y) \cR_{n}^{(k)}(s,y)\frac{z}{\psi(s,y)} \widehat{J}_{\psi}(ds,dy,dz), \quad t \in [0,T],
\]
is a local martingale. By Lenglart's inequality, for any $\e>0$ and $\eta>0$,
\[
\bP(|M_{t}^{(t_0,x_0)}|>\e)\leq \frac{\eta}{\e^2}+\bP\left( \int_0^t \int_{\bR^d}\int_{\{|z|\leq 1\}}G_{t_0-s}^2(x_0-y) |\cR_{n}^{(k)}(s,y)|^2 \frac{z^2}{\psi^2(s,y)} J_{\psi}(ds,dy,dz)>\e\right).
\]

We apply this inequality to $(t_0,x_0)=(t,x)$. Using the points of $J_{\psi}$, we obtain:
\[
\bP(|T_{1}^{(k)}|>\e) \leq \frac{\eta}{\e^2}+\bP\left(\sum_{j\geq 1}G_{t-T_j}^2(x-X_j)|\cR_{n,j}^{(k)}(T_j,X_j)|^2 1_{\{\Gamma_j^{-1/\alpha}\leq 1\}} \frac{\Gamma_j^{-2/\alpha}}{T^{-2/\alpha}\phi^2(X_j)}>\e\right).
\]

We will show below that:
\begin{equation}
\label{sum-conv-0}
T_3^{(k)}:=\sum_{j\geq 1}G_{t-T_j}^2(x-X_j)|\cR_{n,j}^{(k)}(T_j,X_j)|^2 1_{\{\Gamma_j^{-1/\alpha}\leq 1\}} \frac{\Gamma_j^{-2/\alpha}}{ T^{-2/\alpha}\phi^{2}(X_j) } \stackrel{P}{\longrightarrow} 0 \quad 
\mbox{as} \quad k\to \infty.
\end{equation}

It will follow that $\limsup_{k \to \infty}\bP(|T_{k}^{(1)}|>\e) \leq \eta/\e^2$ for any $\e>0$ and $\eta>0$. Letting $\eta \to 0$, we infer that $\lim_{k \to \infty}\bP(|T_{k}^{(1)}|>\e)=0$ for any $\e >0$, i.e. $T_{k}^{(1)} \stackrel{P}{\to} 0$ as $k \to \infty$.

\underline{It remains to prove \eqref{sum-conv-0}.} We proceed as for $T_{1}^{(k)}$. Recalling definition \eqref{def-W-nj} of $W_{n,j}^{(k)}$, we see that
\[
T_{3}^{(k)}=\sum_{j\geq 1}\Gamma_j^{-2/\alpha}1_{\{\Gamma_j^{-1/\alpha}\leq 1\}} \big(W_{n,j}^{(k)}\big)^2.
\]
We use again inequality \eqref{ineq-X0},
with $\cG$ the $\sigma$-field generated by $(\Gamma_i)_i, (T_i)_i,(X_i)_i$. 
We obtain:
\begin{equation}
\label{T3k}
\|T_{3}^{(k)}\|_{L^0} \leq \Big\| \sum_{j\geq 1} \Gamma_{j}^{-2/\alpha} 1_{\{\Gamma_j^{-1/\alpha}\leq 1\}} B_{n,j}^{(k)}\Big\|_{L^0},
\end{equation}
where
\begin{align*}
B_{n,j}^{(k)}&:=\bE\big[\big(W_{n,j}^{(k)}\big)^2\, \big|(\Gamma_i)_i, (T_i)_i,(X_i)_i\big] \\
&=T^{2/\alpha}G_{t-T_j}^2(x-X_j)\phi^{-2}(X_j)\,
\bE\Big[\big|\cR_{n}^{(k)}(T_j,X_j)\big|^2\, \big|(\Gamma_i)_i, (T_i)_i,(X_i)_i\Big].
\end{align*}

Similarly to \eqref{hn2},
\begin{align*}
& \bE\Big[\big|\cR_{n}^{(k)}(T_j,X_j)\big|^2\, \big|(\Gamma_i)_i, (T_i)_i,(X_i)_i\Big]\\
& \quad =T^{\frac{2n}{\alpha}}(n!)^2
\sum_{ \substack{j_1<\ldots < j_n \\ j \not\in \{j_1,\ldots,j_n\}  } } 1_{\{\Gamma_{j_n}^{-1/\alpha} \leq k^{-1} \}} \prod_{i=1}^{n}\Gamma_{j_i}^{-2/\alpha}\phi^{-2}(X_{j_i})
\widetilde{f}_n^2(T_{j_1},X_{j_1},\ldots,T_{j_n},X_{j_n},T_j,X_j)\\
& \quad \leq h_n^{(2)}(T_j,X_j),
\end{align*}
where for the last line, we just bounded $1_{\{\Gamma_{j_n}^{-1/\alpha} \leq k^{-1} \}}$ by $1$. Therefore,
\[
B_{n,j}^{(k)} \leq  T^{2/\alpha}G_{t-T_j}^2(x-X_j)\phi^{-2}(X_j)  h_n^{(2)}(T_j,X_j).
\]

We will prove that
\begin{equation}
\label{Bnj-0}
B_{n,j}^{(k)} \to 0 \quad \mbox{a.s.} \quad \mbox{as $k \to \infty$, for any $j\geq1$,}
\end{equation}
and
\begin{equation}
\label{sum-B}
S:=\sum_{j\geq 1}\Gamma_{j}^{-2/\alpha}G_{t-T_j}^2(x-X_j)\phi^{-2}(X_j)  h_n^{(2)}(T_j,X_j)<\infty \quad \mbox{a.s.}
\end{equation}
Then, by the dominated convergence theorem, $\sum_{j\geq 1}\Gamma_j^{-2/\alpha}1_{\{\Gamma_j^{-1/\alpha}\leq 1\}}B_{n,j}^{(k)} \to 0$ a.s. as $k \to \infty$. Consequently, by \eqref{T3k}, $\|T_{3}^{(k)} \|_0 \to 0$ as $k \to \infty$. This proves \eqref{sum-conv-0}.

\medskip

\underline{We now prove \eqref{Bnj-0}.} We fix $j\geq 1$. It is enough to prove that
\[
\bE\Big[\big|\cR_{n}^{(k)}(T_j,X_j)\big|^2\, \big|(\Gamma_i)_i, (T_i)_i,(X_i)_i\Big] \to 0 \quad \mbox{as} \quad k \to \infty.
\]
This follows from the dominated convergence theorem, provided that we show that
\[
\mathcal{M}_2:=(n!)^2\sum_{ \substack{j_1<\ldots < j_n \\ j \not\in \{j_1,\ldots,j_n\}  } }  \prod_{i=1}^{n}\Gamma_{j_i}^{-2/\alpha}\phi^{-2}(X_{j_i})
\widetilde{f}_n^2(T_{j_1},X_{j_1},\ldots,T_{j_n},X_{j_n},T_j,X_j)<\infty \quad \mbox{a.s.}
\]
For this, we proceed as in the proof of Proposition \ref{prop-conv-series}. By sub-additivity, 
\begin{align*}
\mathcal{M}_2^{p/2} \leq (n!)^p\sum_{ \substack{j_1<\ldots < j_n \\ j \not\in \{j_1,\ldots,j_n\}  } }  \prod_{i=1}^{n}\Gamma_{j_i}^{-p/\alpha}\phi^{-p}(X_{j_i})
\widetilde{f}_n^p(T_{j_1},X_{j_1},\ldots,T_{j_n},X_{j_n},T_j,X_j)=:\mathcal{M}_p.
\end{align*}
So, it is enough to prove that $\mathcal{M}_p<\infty$ a.s. Note that $\cM$ is a measurable function of independent random elements $X=\{(\Gamma_i)_i,T_j,X_i\}$ and $Y=\{(T_i)_{i\not=j},(X_i)_{i\not=j}\}$. Using Lemma \ref{Fubini-lemma}, it suffices to prove that
\[
\bE\big[\mathcal{M}_p\,\big|(\Gamma_i)_i,T_j,X_j \big]<\infty \quad \mbox{a.s.}
\]
Using the basic conditioning fact \eqref{fact1}, we see that
\begin{align*}
&\bE\big[\mathcal{M}_p\,\big|(\Gamma_i)_i=(\gamma_i)_i,T_j=t_j,X_j=x_j \big]\\
& \quad =(n!)^p\sum_{ \substack{j_1<\ldots < j_n \\ j \not\in \{j_1,\ldots,j_n\}  } }
\prod_{i=1}^{n}\gamma_{j_i}^{-p/\alpha} \bE\Big[\prod_{i=1}^{n}\phi^{-1}(X_{j_i})
\widetilde{f}_n^p(T_{j_1},X_{j_1},\ldots,T_{j_n},X_{j_n},t_j,x_j)\Big]\\
&\quad =\sum_{ \substack{j_1<\ldots < j_n \\ j \not\in \{j_1,\ldots,j_n\}  } } \prod_{i=1}^{n}\gamma_{j_i}^{-p/\alpha} \cdot{\frac{n!}{T^n}} K_n^{(p)}(t_j,x_j),
\end{align*}
using Lemma \ref{A-K-lemma} for the last line. Hence,
\begin{align*}
\bE\big[\mathcal{M}_p\,\big|(\Gamma_i)_i,T_j,X_j \big]&=\frac{n!}{T^n}K_n^{(p)}(T_j,X_i) 
\sum_{ \substack{j_1<\ldots < j_n \\ j \not\in \{j_1,\ldots,j_n\}  } } \prod_{i=1}^{n}\Gamma_{j_i}^{-p/\alpha}\\
&\leq \frac{1}{T^n}K_n^{(p)}(T_j,X_i)  \left( \sum_{j\geq 1}\Gamma_j^{-p/\alpha}\right)^n<\infty \quad \mbox{a.s.}
\end{align*}

\medskip

\underline{Finally, we prove \eqref{sum-B}.} We denote by $\overline{f}_n^{(2)}(\cdot,t,x)$ the symmetrization of $f_n^{(2)}(\cdot,t,x)$. Note that
$\widetilde{f}_n^2(\cdot,t,x) \leq \frac{1}{n!}\overline{f}_n^{(2)}(\cdot,t,x) $. Therefore,
\[
h_n^{(2)}(T_j,X_j) \leq T^{\frac{2n}{\alpha}}n!\sum_{ \substack{j_1<\ldots < j_n \\ j \not\in \{j_1,\ldots,j_n\}  } }  \prod_{i=1}^{n}\Gamma_{j_i}^{-2/\alpha} \phi^{-2}(X_{j_i}) \overline{f}_n^{(2)}(T_{j_1},X_{j_1},\ldots,T_{j_n},X_{j_n},T_j,X_j)
\]
and
\begin{align*}
S & \leq T^{\frac{2n}{\alpha}} \sum_{j\geq 1}\Gamma_{j}^{-2/\alpha}\phi^{-2}(X_j)G_{t-T_j}^2(x-X_j) \\
& \quad \sum_{ \substack{j_1<\ldots < j_n \\ j \not\in \{j_1,\ldots,j_n\}  } } \sum_{\pi \in \Sigma_n} \prod_{i=1}^{n}\Gamma_{j_i}^{-2/\alpha} \phi^{-2}(X_{j_i}) f_n^{2}(T_{j_{\pi(1)}},X_{j_{\pi(1)}},\ldots,T_{j_{\pi(n)}},X_{j_{\pi(n)}},T_j,X_j)\\
& = T^{\frac{2n}{\alpha}}  \sum_{j\geq 1}\Gamma_{j}^{-2/\alpha}\phi^{-2}(X_j)G_{t-T_j}^2(x-X_j) \\
& \quad \sum_{ \substack{j_1,\ldots,j_n\geq 1 \ {\rm distinct}\\ j \not\in \{j_1,\ldots,j_n\}  } } \, \prod_{i=1}^{n}\Gamma_{j_i}^{-2/\alpha} \phi^{-2}(X_{j_i}) f_n^{2}(T_{j_{1}},X_{j_{1}},\ldots,T_{j_{n}},X_{j_{n}},T_j,X_j),
\end{align*}
where for the last line we used \eqref{fact5}. We now use the fact that
$G_{t-T_j}^2(x-X_j)f_n^{2}(\cdot,T_j,X_j)=f_{n+1}^2 (\cdot,T_j,X_j,t,x)$. Denoting $j=j_{n+1}$, it follows that
\begin{align*}
S &\leq  T^{\frac{2n}{\alpha}}  \sum_{j_1<\ldots < j_{n+1}} \, \prod_{i=1}^{n+1}\Gamma_{j_i}^{-2/\alpha}\phi^{-2}(X_{j_i})
f_{n+1}^2 (T_{j_1},X_{j_1},\ldots,T_{j_{n+1}},X_{j_{n+1}},t,x)\\ 
& =T^{\frac{2n}{\alpha}}  \sum_{j_1<\ldots < j_{n+1} } \sum_{\pi \in \Sigma_{n+1}} 
\prod_{i=1}^{n+1}\Gamma_{j_i}^{-2/\alpha}\phi^{-2}(X_{j_i})
f_{n+1}^2 (T_{j_{\pi(1)}},X_{j_{\pi(1)}},\ldots, T_{j_{\pi(n+1)}},X_{j_{\pi(n+1)}},t,x),
\end{align*}
where for the last line we used again \eqref{fact5}. Taking power $p/2$, and recalling definition \eqref{def-B}, we obtain that
$S^{p/2} \leq B_{n+1}^{(p)}(t,x)<\infty$ a.s.
\end{proof}

\begin{lemma}
\label{lem-J-conv}
Assume that $\alpha \in [1,2)$.
Let $t \in [0,T]$ and $x \in \bR^d$ be arbitrary. If there exists $p \in (\alpha,2]$ such that $K_{n+1}^{(p)}(t,x)<\infty$ and \eqref{cI} holds, then
\[
B_k:=\int_0^t \int_{\bR^d}G_{t-s}(x-y)\cJ_n^{(k)}(s,y)Z(ds,dy) \stackrel{P}{\longrightarrow
} I_{n+1}\big(f_{n+1}(\cdot,t,x)\big) \quad \mbox{as $k \to \infty$}.
\]
\end{lemma}

\begin{proof} Applying Proposition \ref{prop-Lambda-repr5} with $a=k^{-1}$, we have:
\begin{align*}
B_k&= \int_0^t \int_{\bR^d} 
\int_{\{|z|\leq k^{-1}\}} G_{t-s}(x-y)\cJ_n^{(k)}(s,y) \frac{z}{\psi(s,y)}\widehat{J}_{\psi}(ds,dy,dz)+\\
& \quad
\int_0^t \int_{\bR^d} \int_{\{|z|> k^{-1}\}} G_{t-s}(x-y)\cJ_n^{(k)}(s,y) \frac{z}{\psi(s,y)}J_{\psi}(ds,dy,dz)=:S_1^{(k)}+S_{2}^{(k)}.
\end{align*}
We will prove that:
\begin{align}
\label{conv-Sk1}
&S_1^{(k)} \stackrel{P}{\longrightarrow} 0 \quad \mbox{as} \quad k \to \infty,\\
\label{conv-Sk2}
&S_2^{(k)} \stackrel{P}{\longrightarrow} I_{n+1}\big( f_{n+1}(\cdot,t,x) \big) \quad \mbox{as} \quad k \to \infty.
\end{align}

\underline{We prove \eqref{conv-Sk1}.} As in the proof of Lemma \ref{lem-cR} (Step 2), by Lenglart's inequality, it suffices to prove that:
\begin{equation}
\label{conv-J}
S_3^{(k)}:=\sum_{j\geq 1} \Gamma_{j}^{-2/\alpha} 1_{\{\Gamma_j^{-1/\alpha} <k^{-1}\}} G_{t-T_j}^2(x-X_j)|\cJ_n^{(k)}(s,y)|^2 \phi^{-2}(X_j) \stackrel{P}{\longrightarrow} 0 \quad \mbox{as} \quad k \to \infty.
\end{equation}
For this, denote $V_{n,j}^{(k)}:=G_{t-T_j}(x-X_j)|\cJ_{n,j}^{(k)}(T_j,X_j)|\phi^{-2}(X_j)$.
Using again inequality \eqref{ineq-X0}, we obtain:
\[
\|S_3^{(k)}\|_{L^0} \leq  \Big\|\sum_{j\geq 1}\Gamma_j^{-2/\alpha} 1_{\{\Gamma_j^{-1/\alpha} \leq k^{-1}\}} D_{n,j}^{(k)}\Big\|_{L^0},
\]
where
\begin{align*}
D_{n,j}^{(k)} &:= \bE\big[\big(V_{n,j}^{(k)}\big)^2 \, \big| (\Gamma_i)_i,(T_i)_i, (X_i)_i\big] \\
&=G_{t-T_i}^2(x-X_i) \phi^{-2}(X_i) \, \bE \big[ \big|\cJ_n^{(k)}(T_j,X_j)\big|^2 \, \big| (\Gamma_i)_i,(T_i)_i, (X_i)_i \big].
\end{align*}

By direct calculation,
\begin{align*}
& \bE \big[ \big|\cJ_n^{(k)}(T_j,X_j)\big|^2 \, \big| (\Gamma_i)_i,(T_i)_i, (X_i)_i \big]
= \\
& \quad (n!)^2 \sum_{ \substack{j_1<\ldots < j_n \\ j \not\in \{j_1,\ldots,j_n\}  } }
\prod_{i=1}^{n}\Gamma_{j_i}^{-2/\alpha} 1_{\{\Gamma_{j_i}^{-1/\alpha} > k^{-1} \}}\phi^{-2}(X_{j_i}) \widetilde{f}_n^2 (T_{j_1},X_{j_1},\ldots,T_{j_n},X_{j_n},T_j,X_j),
\end{align*}
which we can bound by $T^{-\frac{2n}{\alpha}}h_n^{(2)}(T_j,X_j)$, using the fact that
$1_{\{\Gamma_{j_i}^{-1/\alpha} > k^{-1} \}}\leq 1$. Hence,
\[
\|S_3^{(k)}\|_{L^0}  \leq \left\|T^{-\frac{2n}{\alpha}}\sum_{j\geq 1}\Gamma_j^{-2/\alpha} 
1_{\{ \Gamma_j^{-1/\alpha} \leq k^{-1}\}} G_{t-T_j}^2(x-X_j) \phi^{-2}(X_j) h_n^{(2)}(T_j,X_j)\right\|_{L^0}.
\]
By the dominated convergence theorem, the last series converges to $0$ a.s. as $k \to \infty$. The  application of this theorem is justified due to \eqref{sum-B}. This proves \eqref{conv-Sk1}.

\medskip

\underline{We prove \eqref{conv-Sk2}.} Using the points of $J_{\psi}$, we write:
\[
S_2^{(k)}=\sum_{j\geq 1} G_{t-T_j}(x-X_j)\cJ_{n}^{(k)}(T_j,X_j) \frac{\e_j \Gamma_j^{-1/\alpha}}{\psi(T_j,X_j)} 1_{\{\Gamma_j^{-1/\alpha}>k^{-1}\}}.
\]
Recalling the definition of $\cJ_n^{(k)}(s,y)$, we see that
\begin{align*}
S_2^{(k)}& =n! \sum_{j\geq 1}G_{t-T_j}(x-X_j) \e_j \phi^{-1}(X_j) \Gamma_{j}^{-1/\alpha}
1_{\{\Gamma_j^{-1/\alpha}>k^{-1}\}} \\
& 
\quad \sum_{ \substack{j_1<\ldots < j_n \\ j \not\in \{j_1,\ldots,j_n\}  } }
\prod_{i=1}^{n}\e_{j_i} 
\Gamma_{j}^{-1/\alpha}
1_{\{\Gamma_{j_i}^{-1/\alpha}>k^{-1}\}}
\phi^{-1}(X_{j_i}) \widetilde{f}_n(T_{j_1},X_{j_1},\ldots,T_{j_n},X_{j_n},T_j,X_j).
\end{align*}

Using the same argument as in the proof of Theorem \ref{th-alpha1}, with $\Gamma_{j}^{-1/\alpha}$ replaced by $\Gamma_{j}^{-1/\alpha}1_{\{\Gamma_j^{-1/\alpha}>k^{-1}\}}$. We obtain:
\begin{align*}
S_2^{(k)}&= (n+1)! T^{\frac{n+1}{\alpha}} \sum_{j_1<\ldots<j_{n+1}} \prod_{i=1}^{n+1}\e_{j_i}\Gamma_{j_i}^{-1/\alpha} 1_{\{\Gamma_{j_i}^{-1/\alpha}>k^{-1} \}}
\phi^{-1}(X_{j_i}) \widetilde{f}_{n+1}(T_{j_1},X_{j_1},\ldots, T_{j_n},X_{j_n},t,x)\\
&=\cJ_{n+1}^{(k)}(t,x) \stackrel{P}{\longrightarrow} I_{n+1}\big(f_{n+1}(\cdot,t,x) \big),
\end{align*}
since $I_{n+1}\big(f_{n+1}(\cdot,t,x) -\cJ_{n+1}^{(k)}(t,x)=\cR_{n+1}^{(k)}(t,x) \stackrel{P}{\to}
0$ as $k \to \infty$, by Lemma \ref{lem-R0}.
\end{proof}

\section{Solvability}
\label{section-solvability}

In this section, we give the proof of Theorem \ref{main-th}.(b). More precisely, we show that the process $\{u(t,x);t\in [0,T],x\in \bR^d\}$ defined by \eqref{def-u} is indeed a solution of equation \eqref{eq}. This will be achieved in Theorem 
\ref{solvability-th}, by letting $n\to \infty$ in the recurrence relation \eqref{picard}.

The last statement of Theorem \ref{main-th}.(b) does not require any effort:
relation \eqref{Gu-int} is a restatement of the fact that $v^{(t,x)} \in L^0(Z)$, where
\[
v^{(t,x)}(s,y)=G_{t-s}(x-y)u(s,y), \quad s \in [0,T],y \in \bR^d,
\]
according to Corollary \ref{integr-th}, while the fact that $u(t,x)$ has representation \eqref{u-series} follows from definition
\eqref{def-u} of $u(t,x)$, combined with the series representation \eqref{LePage-In} of $I_n\big(f_n(\cdot,t,x)\big)$.

\medskip

\begin{theorem}
\label{solvability-th}
Suppose that Hypothesis \ref{G_assumption} holds. If Assumptions \ref{ass-A2} and \ref{ass-A3} hold (with possibly different values $p \in (\alpha,2]$), then for any $(t,x) \in [0,T] \times \bR^d$, 
$v^{(t,x)} \in L^0(Z)$ and
\begin{equation}
\label{eq101}
u(t,x)=1+I^Z(v^{(t,x)}) \quad a.s.
\end{equation}
\end{theorem}

\begin{proof} 
{\em Step 1.} Denote $v_n^{(t,x)}(s,y)=G_{t-s}(x-y)u_n(s,y)$. In this step, we show that: 
\begin{equation}
\label{conv-v}
v^{(t,x)}\in L^0(Z) \quad \mbox{and} \quad \|v_n^{(t,x)}-v^{(t,x)}\|_{Z} \to 0.
\end{equation}
By Proposition \ref{prop-abs-conv}, for any $(s,y)\in [0,t] \times \bR^d$,
$\sum_{n\geq 1}|I_n(f_n(\cdot,s,y))|<\infty$ a.s., and so,
\begin{align*}
v_n^{(t,x)}(s,y)-v^{(t,x)}(s,y)&=G_{t-s}(x-y)\big(u_n(s,y)-u(s,y)\big)\\
&=G_{t-s}(x-y)\sum_{k\geq n+1}I_k(f_k(\cdot,s,y)) \to 0 \quad \mbox{a.s.}
\end{align*}
Notice that we have the following natural dominator:
\[
|v_n^{(t,x)}(s,y)|\leq G_{t-s}(x-y) \sum_{k=1}^{n}|I_k(f_k(\cdot,s,y))|\leq  G_{t-s}(x-y)
\sum_{n\geq 1}|I_n(f_n(\cdot,s,y))|=:\overline{v}^{(t,x)}(s,y).
\]
Recalling our convention \eqref{convention}, the desired conclusion \eqref{conv-v} will follow by Theorem \ref{DCT}, provided that  we show that $\overline{v}^{(t,x)} \in L^0(Z)$, which is equivalent to: (see Corollary \ref{integr-th})
\begin{equation}
\label{bar-vtx}
\int_0^t \int_{\bR^d}|\overline{v}^{(t,x)}(s,y)|^{\alpha}dyds<\infty \quad \mbox{a.s.}
\end{equation}

To prove \eqref{bar-vtx}, we will use Minkowski's inequality in $L^{\alpha}([0,t]\times \bR^d)$. Recall that for a measure space $(E,\mathcal{E},\mu)$ and a measurable function $f:E\to \bR$,
\[
\|f\|_{L^{\alpha}(E)}=\int_{E} |f|^{\alpha} d\mu \ \ \ \mbox{if $\alpha \leq 1$} \quad \mbox{and} \quad  \|f\|_{L^{\alpha}(E)}=\left(\int_{E}|f|^{\alpha}d\mu\right)^{1/\alpha} \ \ \ \mbox{if $\alpha>1$}.
\]
Hence,
\[
\|\overline{v}^{(t,x)}\|_{L^{\alpha}([0,t] \times \bR^d)} = \left\|\sum_{n\geq 1}G_{t-*}(x-*)I_n\big(f_n(\cdot,*)\big)\right\|_{L^{\alpha}([0,t] \times \bR^d)} \leq 
\sum_{n\geq 1}|\cI_n(t,x)|^{\frac{1}{\alpha \vee 1}}.
\]
By Theorem \ref{A3-th}, the last series converges almost surely.

\medskip

{\em Step 2.} In this step, we show that \eqref{eq101} holds almost surely. 
For this, we let $n\to \infty$ in \eqref{picard}. On the left hand-side,
 $u_{n+1}(t,x) \to u(t,x)$ almost surely, by Proposition \ref{prop-abs-conv}. 
The term on the right hand-side of \eqref{picard} is equal to $1+I^{Z}(v_n^{(t,x)})$, 
which converges in probability to $1+I^{Z}(v^{(t,x)})$, due to 
 \eqref{conv-v} and property \eqref{contraction2} of integral $I^Z$.
\end{proof}

\section{Applications: the heat and wave equations}
\label{section-applications}

In this section, we give the proof of Theorem \ref{main-appl}. We consider separately the heat and
wave equations.

\subsection{Heat equation}

In this section, we show that Assumptions \ref{ass-A2} and \ref{ass-A3} are satisfied
in the case of the heat equation.

 We will use the following property of the heat kernel: 
\begin{equation}
\label{Gp-heat}
G_t^p(x)=\overline{K}_{p,d} \, t^{\frac{d(1-p)}{2}}G_{t/p}(x), \quad
\mbox{with} \quad \overline{K}_{p,d}=(2\pi)^{\frac{d(1-p)}{2}}p^{-\frac{d}{2}}. 
\end{equation}

The following result is essentially contained in the proof of Theorem 3.1 of \cite{chong17-SPA}. We include its proof since we need the explicit form of the constant $C_{\eta,p,d}$. 

\begin{lemma}
\label{lem-Gh}
If $G$ is the fundamental solution of the heat equation,
then for any $\eta>0$ and $0<p<1+\frac{2}{d}$, 
\begin{align*}
&I_{\eta,p}^{\rm heat}(t,x):=\int_{T_n(t)}\int_{(\bR^d)^n} \prod_{k=1}^{n} G_{t_{k+1}-t_k}^p(x_{k+1}-x_k) (1+|x_k|^{\eta}) d\pmb{x} d\pmb{t} \leq \\
\nonumber
& \quad \quad \quad C_{\eta,p,d}^n \left\{1+|x|^{n\eta}+t^{n\eta/2}\Gamma\left(\frac{1+n\eta}{2} \right) \right\} \frac{t^{n(\frac{d(1-p)}{2}+1)}}{\Gamma(n(\frac{d(1-p)}{2}+1)+1)},
\end{align*}
where $t_{n+1}=t$, $x_{n+1}=x$ and 
\begin{equation}
\label{def-C-eta}
C_{\eta,p,d}=3\overline{K}_{p,d} (2^{\eta-1}\vee 1) d^{\eta}\Big[\Big(\frac{2}{p}\Big)^{\eta/2}\vee 1\Big]\Gamma\Big(\frac{d(1-p)}{2}+1\Big).
\end{equation}

\end{lemma}

\begin{proof}
Let $h(x)=1+|x|^{\eta}$ for $x \in \bR^d$. By the generalized H\"older's inequality,
\begin{align}
\nonumber
& \int_{(\bR^d)^n} \prod_{k=1}^nG_{t_{k+1}-t_k}^p(x_{k+1}-x_k) h(x_k) d\pmb{x}  \leq
\prod_{i=1}^{n} \left( \int_{(\bR^d)^n} \prod G_{t_{k+1}-t_k}^p(x_{k+1}-x_k) h(x_i)^n d\pmb{x} \right)^{1/n}\\
\label{int-G1}
& \qquad \qquad \qquad   =\prod_{i=1}^{n} \left( \int_{(\bR^d)^n} \prod_{k=1}^{n} G_{t_{k+1}-t_k}^p(x_k) h(x-\sum_{j=i}^{n}x_j)^n d\pmb{x} \right)^{1/n}.
\end{align}
We estimate separately each of the integrals appearing in the product above. Using relation \eqref{Gp-heat}, we write
\[
\prod_{k=1}^{n} G_{t_{k+1}-t_k}^p(x_k)=\overline{K}_{p,d}^n \prod_{k=1}^{n}(t_{k+1}-t_k)^{\frac{d(1-p)}{2}} \prod_{k=1}^{n}f_{X_k}(x_k),
\]
where $X_1,\ldots,X_k$ are independent random variables,  $X_k \sim N_d(0, \frac{t_{k+1}-t_k}{p}I_d)$, and $f_{X_k}$ is the density of $X_k$. Hence,
\begin{align*}
\int_{(\bR^d)^n} \prod_{k=1}^{n} G_{t_{k+1}-t_k}^p(x_k) h(x-\sum_{j=i}^{n}x_j)^n d\pmb{x} =\overline{K}_{p,d}^n \prod_{k=1}^{n}(t_{k+1}-t_k)^{\frac{d(1-p)}{2}} \bE \Big[h\big(x-\sum_{j=i}^n X_j \big)^n \Big].
\end{align*}

Note that $h(x-\sum_{j=i}^n X_j) \leq (2^{\eta-1} \vee 1) (1+|x|^{\eta}+|\sum_{j=i}^{n}X_j|^{\eta})$, and hence
\[
\bE\Big[h\big(x-\sum_{j=i}^n X_j\big)^n\Big] \leq (2^{\eta-1} \vee 1)^n 3^{n-1}\Big(1+|x|^{n\eta }+\bE\Big|\sum_{j=i}^{n}X_j\Big|^{n\eta }\Big).
\]
Moreover, $|x|=\sqrt{\sum_{\ell=1}^{d} |x^{(\ell)}|^2} \leq \sum_{\ell=1}^{d}|x^{(\ell)}|$, for any
$x=(x^{(1)},\ldots,x^{(d)})\in \bR^d$, and hence,
\begin{align*}
\bE\Big|\sum_{j=i}^{n}X_j\Big|^{n\eta } & \leq \bE \Big[\Big(\sum_{\ell=1}^{d} \Big|\sum_{j=i}^n X_{j}^{(\ell)}\Big| \Big)^{n\eta} \Big] \leq d^{n\eta  -1} \sum_{\ell=1}^{d} \bE\Big| \sum_{j=i}^{n}X_j^{(\ell)}\Big|^{n \eta}\\
& \leq d^{n\eta} \max_{\ell=1,\ldots,d} \bE\Big| \sum_{j=i}^{n}X_j^{(\ell)}\Big|^{n \eta}.
\end{align*}
Putting together these estimates, we infer that, for any $i=1,\ldots, n$ fixed,
\begin{align}
\label{int-Gh}
& \int_{(\bR^d)^n} \prod_{k=1}^{n} G_{t_{k+1}-t_k}^p(x_k) h(x-\sum_{j=i}^{n}x_j)^n d\pmb{x} \leq \\
\nonumber
& \quad \Big(3\overline{K}_{p,d} (2^{\eta-1} \vee 1)  d^{\eta}\Big)^n 
\prod_{k=1}^{n}(t_{k+1}-t_k)^{\frac{d(1-p)}{2}} 
\Big(1+|x|^{n\eta}+ \max_{\ell=1,\ldots,d} \bE\Big| \sum_{j=i}^{n}X_j^{(\ell)}\Big|^{n \eta}\Big).
\end{align}

We will show that this can be bounded by a quantity not depending on $i$. Note that
\[
\sum_{j=i}^n X_{j}^{(\ell)} \sim N(0,\gamma_i), \quad \mbox{where} \quad \gamma_i:=\sum_{j=i}^n\frac{t_{k+1}-t_k}{p}=\frac{t-t_i}{p}.
\]
Recall that if $Z \sim N(0,\sigma^2)$, then $\bE|Z|^p=(2\sigma^2)^{p/2} \pi^{-\frac{1}{2}}\Gamma(\frac{1+p}{2})$ for any $p>0$. Hence,
\[
\bE\Big|\sum_{j=i}^n X_{j}^{(\ell)} \Big|^{n\eta}=(2\gamma_i)^{n\eta/2} \pi^{-\frac{1}{2}}\Gamma\left(\frac{1+n\eta}{2}\right) \leq \left(\frac{2t}{p}\right)^{n\eta/2}
\Gamma\left(\frac{1+n\eta}{2}\right).
\]

\noindent
Using this estimate in \eqref{int-Gh}, we obtain the following bound, which does not depend on $i$:
\begin{align*}
& \int_{(\bR^d)^n} \prod_{k=1}^{n} G_{t_{k+1}-t_k}^p(x_k) h(x-\sum_{j=i}^{n}x_j)^n d\pmb{x} \leq \\
& \quad (C_{\eta,p,d}')^n 
\prod_{k=1}^{n}(t_{k+1}-t_k)^{\frac{d(1-p)}{2}} 
\left(1+|x|^{n\eta}+ t^{n\eta/2}
\Gamma\Big(\frac{1+n\eta}{2}\Big)\right),
\end{align*}
where $C_{\eta,p,d}':=3\overline{K}_{p,d} (2^{\eta-1} \vee 1)  d^{\eta} \big(\big(\frac{2}{p}\big)^{\eta/2} \vee 1 \big)$. 

Returning to \eqref{int-G1}, and integrating on $T_n(t)$ we get:
\begin{align*}
I_{\eta,p}^{\rm heat}(t,x) & \leq (C_{\eta,p,d}')^n \left(1+|x|^{n\eta}+ t^{n\eta/2}
\Gamma\Big(\frac{1+n\eta}{2}\Big)\right) \int_{T_n(t)} \prod_{k=1}^{n}(t_{k+1}-t_k)^{\frac{d(1-p)}{2}} d\pmb{t}.
\end{align*}
The last integral is equal to $\frac{\Gamma(a+1)^{n} t^{n(a+1)}}{\Gamma(n(a+1)+1)}$, with $a:=\frac{d(1-p)}{2}>-1$. The conclusion follows.

\end{proof}

The following result gives an estimate for $K_n^{(p)}(t,x)$ in the case of the heat equation.

\begin{lemma}
\label{est-K-heat}
Suppose that $\phi$ satisfies Hypothesis \ref{hypo2}. In the case of the heat equation, for any $\alpha<p<1+\frac{2}{d}$, we have:
\[
K_n^{(p)}(t,x) \leq c_0^{n(p-\alpha)}C_{\eta,p,d}^n\left\{ 1+|x|^{n\eta}+t^{n\eta/2} \Gamma\left( \frac{1+n\eta}{2}\right)\right\} \frac{t^{n(\frac{d(1-p)}{2}+1)}}{\Gamma(n(\frac{d(1-p)}{2}+1)+1)},
\]
where $\eta=\delta(p-\alpha)$, and $C_{\eta,p,d}$ is given by \eqref{def-C-eta}.
\end{lemma}

\begin{proof}
Recall definition \eqref{def-K} of $K_n^{(p)}(t,x)$. 
Using \eqref{bound-phi}, we see that 
\begin{equation}
\label{bound-phi1}
\phi^{\alpha-p}(x) \leq c_0^{p-\alpha}(1+|x|^{\delta(p-\alpha)}) \quad \mbox{for any $x\in \bR^d$}.
\end{equation}
The conclusion follows by Lemma \ref{lem-Gh}.  
\end{proof}

To find upper and lower bounds for the Gamma functions appearing in the above estimate, we use Stirling's formula. For any $a>0$,
$\Gamma(an+1) \sim a^{an+1/2} (2\pi n)^{(1-a)/2} (n!)^a$, and hence
\begin{equation}
\label{stir1}
C_{a}^{-n} (n!)^a \leq \Gamma(an+1) \leq C_{a}^n (n!)^a  \quad \mbox{for all $n\geq 1$},
\end{equation}
where $C_a>1$ is a constant depending on $a$. Moreover, 
for any $a>0$ and $b \in \bR$, $\Gamma(an+1+b) \sim \Gamma(an+1)n^b$, and hence
\begin{equation}
\label{stir2}
C_{a,b}^{-n} (n!)^a \leq \Gamma(an+1+b) \leq C_{a,b}^n (n!)^a \quad \mbox{for all $n\geq 1$}.
\end{equation}
where $C_{a,b}>1$ is a constant depending on $a$ and $b$.

\medskip

The following result shows that Assumptions \ref{ass-A2} and \ref{ass-A3} are satisfied in the case of the 
heat equation.

\begin{proposition}
\label{prop-A2-heat}
Suppose that $\phi$ satisfies Hypothesis \ref{hypo2}. If $\mathcal{L}=\frac{\partial }{\partial t}-\frac{1}{2}\Delta$ is the heat operator and $\alpha<1+\frac{2}{d}$, 
then \eqref{A2} and \eqref{A3} hold for any $(t,x)\in [0,T] \times \bR^d$, and for any 
$\alpha<p<1+\frac{2}{d}$ such that
\begin{equation}
\label{cond-p-delta}
0<\delta(p-\alpha)<d(1-p)+2.
\end{equation}
\end{proposition}

\begin{proof}
We use the estimate for $K_n^{(p)}(t,x)$ given by  Lemma \ref{est-K-heat}. 
Recall that $\eta=\delta(p-\alpha)$ and $C_{\eta,p,d}$ is given by \eqref{def-C-eta}. 

{\em Step 1.} We first verify \eqref{A2}. By \eqref{stir1} and \eqref{stir2}, there exist some constants $C_{d,p}>1$ and $C_{\eta}>1$ such that for all $n\geq 1$, 
\[
\Gamma\Big(n\Big(\frac{d(1-p)}{2}+1\Big)+1\Big) \geq C_{d,p}^{-n} (n!)^{\frac{d(1-p)}{2}+1}
\quad \mbox{and} \quad
\Gamma\Big(\frac{1+n\eta}{2}\Big)
\leq C_{\eta}^{n} (n!)^{\eta/2}.
\]
 Hence,
\begin{align*}
K_n^{(p)}(t,x) & \leq C^n \big( 1+|x|^{n\eta}+ t^{n\eta/2} \big) 
\frac{t^{n(\frac{d(1-p)}{2}+1)}}{(n!)^{\frac{d(1-p)-\eta}{2}+1}},
\end{align*}
with $C= c_0^{p-\alpha} C_{\eta,p,d} C_{d,p} C_{\eta}$. It follows that
\begin{align*}
& \sum_{n\geq 1}\Big(T^{(\frac{p}{\alpha}-1)n}K_{n}^{(p)}(t,x)\Big)^{1/2} \Big( \sum_{j\geq 1} \Gamma_j^{-p/\alpha}\Big)^{n/2} \leq \\
& \quad \quad \quad \sum_{n\geq 1} C^{n/2} \big( 1+|x|^{n\eta}+ t^{n\eta/2} \big)^{1/2} 
\frac{T^{\frac{n}{2}(\frac{d(1-p)}{2}+\frac{p}{\alpha})}}{(n!)^{\frac{1}{2}(\frac{d(1-p)-\eta}{2}+1)}}
\Big( \sum_{j\geq 1} \Gamma_j^{-p/\alpha}\Big)^{n/2}.
\end{align*}
The last series converges provided that $0<\eta<d(1-p)+2$. 
This proves \eqref{A2}.

\medskip

{\em Step 2.} Next, we verify \eqref{A3}. Using Lemma \ref{est-K-heat}, we have:
\begin{align*}
& T^{(\frac{p}{\alpha}-1)n}\int_0^t \int_{\bR^d} G_{t-s}^{\alpha}(x-y) K_n^{(p)}(s,y)dyds \leq c_0^{n(p-\alpha)}C_{\eta,p,d}^n \frac{T^{n(\frac{p}{\alpha}+\frac{d(1-p)}{2}) }}{\Gamma(n(\frac{d(1-p)}{2}+1)+1)}\\
& \quad 
\left\{ \int_0^t \int_{\bR^d}
G_{t-s}^{\alpha}(x-y)
(1+|y|^{n\eta})dyds +t^{n\eta/2} \Gamma\left( \frac{1+n\eta}{2}\right) \int_0^t \int_{\bR^d}
G_{t-s}^{\alpha}(x-y) dyds\right\}.
\end{align*}
We use Lemmas \ref{lemC1} and \ref{lemC2} to estimate the two integrals above. We get:
\begin{align*} 
& T^{(\frac{p}{\alpha}-1)n}\int_0^t \int_{\bR^d} G_{t-s}^{\alpha}(x-y) K_n^{(p)}(s,y)dyds \leq c_0^{n(p-\alpha)}C_{\eta,p,d}^n \frac{T^{n(\frac{p}{\alpha}+\frac{d(1-p)}{2}) }}{\Gamma(n(\frac{d(1-p)}{2}+1)+1)}\\
& \quad \left\{K_{\alpha,d}t^{\frac{d(1-\alpha)}{2}+1}+C_{n\eta,\alpha,d}' t^{\frac{d(1-\alpha)}{2}+1} (|x|^{n\eta}+t^{n\eta/2})+K_{\alpha,d}\Gamma\left( \frac{1+n\eta}{2}\right)t^{\frac{n\eta+d(1-\alpha)}{2}+1} \right\}.
\end{align*}
Recalling definition \eqref{def-c} of $C_{\gamma,p,d}'$, and denoting
\begin{equation}
C_{\gamma,p,d}^*=\frac{C_{\gamma,p,d}'}{K_{p,d}}=  (2^{\gamma-1} \vee 1)  (1\wedge p)^{-\gamma/2} \left[1+\frac{2^{
\gamma/2}}{\Gamma(d/2)}\Gamma\left(\frac{\gamma+d}{2}\right) \right],
\end{equation}
we obtain:
\begin{align*}
& T^{(\frac{p}{\alpha}-1)n}\int_0^t \int_{\bR^d} G_{t-s}^{\alpha}(x-y) K_n^{(p)}(s,y)dyds \leq c_0^{n(p-\alpha)}C_{\eta,p,d}^n \frac{T^{n(\frac{p}{\alpha}+\frac{d(1-p)}{2}) }}{\Gamma(n(\frac{d(1-p)}{2}+1)+1)}  \\
& \quad \quad \quad K_{\alpha,d} \left\{ t^{\frac{d(1-\alpha)}{2}+1}\Big[1+C_{n\eta,\alpha,d}^*  (|x|^{n\eta}+t^{n\eta/2})\Big]+ \Gamma\left( \frac{1+n\eta}{2}\right) t^{\frac{n\eta+d(1-\alpha)}{2}+1} \right\}.
\end{align*}
Using inequalities \eqref{stir1} and \eqref{stir2}, we obtain the estimates:
$C_{n\eta,\alpha,d}^* \leq C^n(n!)^{\eta/2}$,
\[
\Gamma\left(n\Big(\frac{d(1-p)}{2}+1\Big)+1\right) \geq C^{-n} (n!)^{\frac{d(1-p)}{2}+1} \quad \mbox{and} \quad
\Gamma\left( \frac{1+n\eta}{2}\right) \leq C^n (n!)^{\eta/2},
\]
where $C>0$ is a constant that depends on $(\eta,d,p)$. Hence,
\begin{align*}
& T^{(\frac{p}{\alpha}-1)n}\int_0^t \int_{\bR^d} G_{t-s}^{\alpha}(x-y) K_n^{(p)}(s,y)dyds \leq c_0^{n(p-\alpha)}C_{\eta,p,d}^n \frac{T^{n(\frac{p}{\alpha}+\frac{d(1-p)}{2}) }}{ (n!)^{\frac{d(1-p)}{2}+1} } C^{2n} \\
& \quad \quad \quad K_{\alpha,d} \left\{t^{\frac{d(1-\alpha)}{2}+1}\Big[1+ (n!)^{\eta/2} (|x|^{n\eta}+t^{n\eta/2})\Big]+  t^{\frac{n\eta+d(1-\alpha)}{2}+1}(n!)^{\eta/2} \right\}\\
& \quad = c_0^{n(p-\alpha)}C_{\eta,p,d}^n T^{n(\frac{d(1-p)}{2}+\frac{p}{\alpha})}  C^{2n} K_{\alpha,d}\left\{\frac{t^{\frac{d(1-\alpha)}{2}+1}}{(n!)^{\frac{d(1-p)}{2}+1} }+\frac{|x|^{n\eta}+
t^{n\eta/2}+t^{\frac{n\eta+d(1-\alpha)}{2}+1}}{(n!)^{\frac{d(1-p)-\eta}{2}+1} } \right\}.
\end{align*}
Using this estimate, it is not difficult to see that condition \eqref{A3} holds, since
$\frac{d(1-p)-\eta}{2}+1>0$ (due to condition \eqref{cond-p-delta}).
\end{proof}

\subsection{Wave equation}

In this section, we show that Assumptions \ref{ass-A2} and \ref{ass-A3} are satisfied
in the case of the wave equation.

\begin{lemma}
\label{lem-Gw}
If $G$ is the fundamental solution of the wave equation in dimension $d\leq 2$, then for any $\eta>0$ and for any $p>0$ if $d=1$, respectively $p \in (0,2)$ if $d=2$, we have
\begin{align*}
I_{\eta,p}^{\rm wave}(t,x)  & :=\int_{T_n(t)}\int_{(\bR^d)^n} \prod G_{t_{k+1}-t_k}^p(x_{k+1}-x_k) (1+|x_k|^{\eta}) d\pmb{x} d\pmb{t} \\
&\leq 
 C_{\eta,p,d}^n (1+|x|^{n\eta}+t^{n\eta}) \frac{t^{an}}{\Gamma(an+1)},
\end{align*}
where $t_{n+1}=t$, $x_{n+1}=x$, 
\begin{equation}
\label{def-a}
a=
	\begin{cases}
		2                               & \text{if $d=1$,}\\[1em]
		3-p & \text{if $d=2$},
	\end{cases}
\end{equation}
and the constant $C_{\eta,p,d}$ is given by
\begin{equation}
\label{def-C-eta-w}
C_{\eta,p,d}=
	\begin{cases}
		\displaystyle 3(2^{\eta-1}\vee 1)2^{1-p}                               & \text{if $d=1$,}\\[1em]
		\displaystyle 3 (2^{\eta-1}\vee 1)\frac{(2\pi)^{1-p}}{2-p}\Gamma(3-p) & \text{if $d=2$}.
	\end{cases}
\end{equation}
\end{lemma}

\begin{proof}
We use similar arguments to those contained in the proof of Theorem 2.4 of 
\cite{JJ1}. In both cases $d=1$ and $d=2$, the product $\prod_{k=1}^{n}G_{t_{k+1}-t_k}^p(x_{k+1}-x_k)$ contains the indicator of $\{|x_2-x_1|<t_2-t_1,\ldots, |x-x_n|<t-t_n\}$. On this set, for any $k=1,\ldots,n$, 
\[
|x-x_k| \leq \sum_{j=k}^n |x_{j+1}-x_j| \leq \sum_{j=k}^n (t_{j+1}-t_j)=t-t_k<t,
\]
and $|x_k|\leq |x|+|x_k-x| \leq |x|+t$. Hence
$\prod_{k=1}^n (1+|x_k|^{\eta}) \leq C_{\eta}^n(1+|x|^{n \eta}+t^{n \eta})$,
where $C_{\eta}=3(2^{\eta-1}\vee 1)$.
It follows that
\begin{align*}
I_{t,x}^{\rm wave} & \leq C_{\eta}^n (1+|x|^{n \eta}+t^{n \eta}) \int_{T_n(t)}\int_{(\bR^d)^2}
\prod_{k=1}^{n}G_{t_{k+1}-t_k}^p(x_{k+1}-x_k) d\pmb{x} d\pmb{t} \\
&=C_{\eta}^n (1+|x|^{n \eta}+t^{n\eta}) \int_{T_n(t)} \prod_{k=1}^{n} 
\left(\int_{\bR^d} G_{t_{k+1}-t_k}^p (x_k) dx_k \right) d\pmb{t}.
\end{align*}

\noindent
If $d=1$, $\int_{\bR}G_t^p(x)dx=2^{1-p}t$ for any $p>0$, and
\[
I_{t,x}^{\rm wave} \leq (C_{\eta} 2^{1-p})^n (1+|x|^{n \eta}+t^{n \eta}) \int_{T_n(t)} \prod_{k=1}^{n}(t_{k+1}-t_k)d\pmb{t}=(C_{\eta} 2^{1-p})^n (1+|x|^{n \eta}+t^{n \eta}) \frac{t^{2n}}{(2n)!}.
\]
If $d=2$, $\int_{\bR^2}G_t^p(x)dx=\frac{(2\pi)^{1-p}}{2-p}t^{2-p}$ for any $p \in (0,2)$, and
\begin{align*}
I_{t,x}^{\rm wave} & \leq \left(C_{\eta} \frac{(2\pi)^{1-p}}{2-p} \right)^n (1+|x|^{n \eta}+t^{n \eta}) \int_{T_n(t)} \prod_{k=1}^{n}(t_{k+1}-t_k)^{2-p}d\pmb{t}\\
&=\left( C_{\eta}\frac{(2\pi)^{1-p}}{2-p} \right)^n (1+|x|^{n \eta}+t^{n \eta}) \cdot \frac{\Gamma(3-p)^n t^{n(3-p)}}{\Gamma((3-p)n+1)}.
\end{align*}
\end{proof}

\begin{lemma}
\label{est-K-wave}
Suppose that $\phi$ satisfies Hypothesis \ref{hypo2}.
In the case of the wave equation in dimension $d\leq 2$, for any $p>0$ if $d=1$, respectively $p \in (0,2)$ if $d=2$,
\[
K_{n}^{(p)}(t,x) \leq c_0^{n(p-\alpha)}C_{\eta,p,d}^n (1+|x|^{n\eta}+t^{n\eta}) \frac{t^{an}}{\Gamma(an+1)},
\]
where 
$\eta=\delta(p-\alpha)$, $a$ is given by \eqref{def-a}, and $C_{\eta,p,d}$ is given by \eqref{def-C-eta-w}.
\end{lemma}

\begin{proof}
This follows using the definition \eqref{def-K} of $K_n^{(p)}(t,x)$, the bound \eqref{bound-phi1} for $\phi^{\alpha-p}(x)$ and Lemma \ref{lem-Gw}.
\end{proof}

Note that for any $p>0$ if $d=1$, respectively 
for any $p \in (0,2)$ if $d=2$, we have:
\begin{equation}
\label{int-Gp-wave}
\int_0^t \int_{\bR^d} G_{t-s}^{p}(x-y)dyds=C_p t^a,
\end{equation}
where $a$ is given by \eqref{def-a}, and $C_p=2^{-p}$ if $d=1$, respectively $C_p=\frac{(2\pi)^{1-p}}{(2-p)(3-p)}$ if $d=2$.

\medskip

The following result shows that Assumptions \ref{ass-A2} and \ref{ass-A3} are satisfies in the case of the
wave equation.

\begin{proposition}
\label{prop-A2-wave}
Suppose that $\phi$ satisfies Hypothesis \ref{hypo2}.
If $\cL=\frac{\partial}{\partial t^2}-\Delta$ is the wave operator in dimension $d\leq 2$, then
\eqref{A2} and \eqref{A3} hold for any $(t,x)\in [0,T] \times \bR^d$, and for any $p>0$ if $d=1$, 
respectively for any $p \in (0,2)$ if $d=2$. 
\end{proposition}

\begin{proof}
We use the estimate for $K_n^{(p)}(t,x)$ given by Lemma \ref{est-K-wave}. Recall that the constant 
$C_{\eta,p,d}$ is given by \eqref{def-C-eta-w} and $a$ is given by \eqref{def-a}.

{\em Step 1.} We first prove that \eqref{A2} holds. 
By \eqref{stir1}, there exists a constant $C_a>1$ such that $\Gamma(an+1) \geq C_a^{-n}(n!)^a$ for all $n\geq 1$. Hence,
\[
K_n^{(p)}(t,x) \leq C^n (1+|x|^{n\eta}+t^{n\eta}) \frac{t^{an}}{(n!)^a},
\]
where $C=c_0^{p-\alpha} C_{\eta,p,d} C_a$.
It follows that
\begin{align*}
& \sum_{n\geq 1}\Big(T^{(p/\alpha-1)n}K_{n}^{(p)}(t,x)\Big)^{1/2} \Big( \sum_{j\geq 1} \Gamma_j^{-p/\alpha}\Big)^{n/2} \leq \\
& \quad \quad \quad \sum_{n\geq 1} C^{n/2} \big( 1+|x|^{n\eta}+ t^{n\eta} \big)^{1/2} 
\frac{T^{\frac{n}{2}(a+\frac{p}{\alpha}-1)}}{(n!)^{\frac{a}{2}  }}
\Big( \sum_{j\geq 1} \Gamma_j^{-p/\alpha}\Big)^{n/2}<\infty.
\end{align*}

{\em Step 2.} Next, we prove that \eqref{A3} holds. 
Note that $G_{t-s}^{\alpha}(x-y)$ contains the indicator of the set
$B_{t,x}:=\{s\in (0,t),y \in \bR^d;|x-y|<t-s\}$. For any $(s,y)\in B_{t,x}$, we have:
\[
1+|y|^{n\eta}+s^{n\eta} \leq 1+(|x|+t)^{n\eta} +s^{n\eta} \leq 1+ (2^{n\eta-1}\vee 1)(|x|^{n\eta}+t^{n\eta})+t^{n\eta} \leq C^n(1+|x|^{n\eta}+t^{n\eta}),
\]
where $C>0$ is a constant depending on $\eta$. Combing this with \eqref{int-Gp-wave}, we infer that:
\begin{align*}
& T^{(\frac{p}{\alpha}-1)n}\int_0^t \int_{\bR^d} G_{t-s}^{\alpha}(x-y)K_n^{(p)}(s,y)dyds  \\
& \quad \leq c_0^{n(p-\alpha)} C_{\eta,p,d}^n 
 \frac{T^{(\frac{p}{\alpha}-1+a)n}}{\Gamma(an+1)} C^n (1+|x|^{n\eta}+t^{n\eta})  \int_0^t \int_{\bR^d}G_{t-s}^{\alpha}(x-y)dyds \\
& \quad = c_0^{n(p-\alpha)} C_{\eta,p,d}^n 
 \frac{T^{(\frac{p}{\alpha}-1+a)n}}{\Gamma(an+1)} C^n (1+|x|^{n\eta}+t^{n\eta}) C_p t^{an}.
\end{align*}
From this estimate, it is not difficult to see that relation \eqref{A3} holds.
\end{proof}

\appendix

\section{The stochastic integral with respect to $\widehat{N}$}
\label{app-integrN}

In this section, we recall some key components of the integration theory with respect to the compensated process $\widehat{N}$.

Let $N$ be a PRM on the space $U=\bR_{+}\times \bR^d \times \bR_0$ of intensity $dtdx \nu(dz)$, where $\nu$ is a L\'evy measure on $\bR$, i.e. $\nu(\{0\})=0$ and $\int_{\bR} (|z|^2 \wedge 1) \nu(dz)<\infty$. For any Borel set $F$ in $U$ with $\mu(F)<\infty$, we let $\widehat{N}(F)=N(F)-\mu(F)$.

The stochastic integral with respect to the compensated process $\widehat{N}$ is defined similarly to the It\^o integral, as explained for instance in Chapter 4 of \cite{applebaum09}. 
More precisely, for any $\wP \times \cB(\bR_0)$-measurable process $H$ with $\bE\int_{U} H^2 d\mu<\infty$, the stochastic integral
$I^{\widehat{N}}(H)=\int_{U} X d\widehat{N}$ is a zero mean random variable with
$\bE|I^{\widehat{N}}(H)|^2=\bE\int_{U} X^2 d\mu$, and $\{M_t=I^{\widehat{N}}(1_{[0,t]}H);t\geq 0\}$ is a square-integrable martingale. The definition of the integral can be extended to processes $H$ satisfying $\int_U |H|^2d\mu<\infty$ a.s., and in this case $M$ is a local martingale, which satisfies:
\[
\bP(|M_t|>\e)\leq \frac{\eta}{\e^2}+\bP\left( \int_0^t \int_{\bR^d} \int_{\bR_0} H^2(s,x,z)\mu(dz,dx,dz)>\eta\right)
\]
for any $\e>0$ and $\eta>0$.
The process $M$ has a c\`adl\`ag modification (denoted also by $M$), whose jump at time $s$ is given by
\[
\Delta M_s=\sum_{i\geq 1} H(T_i,X_i,Z_i) 1_{\{T_i= s\}} \quad
\mbox{where $N=\sum_{i\geq 1}\delta_{(T_i,X_i,Z_i)}$}.
\]
By Lemma I.4.51 of \cite{JS}, the quadratic variation of $M$ is 
\[
[M]_t=\sum_{s\in [0,t]}(\Delta M_s)^2=\sum_{i\geq 1} H^2(T_i,X_i,Z_i) 1_{\{T_i\leq s\}}=\int_0^t \int_{\bR^d} \int_{\bR_0} H^2(s,x,z)N(ds,dx,dz).
\]
 By the Burkholder-Davis-Gundy inequality for c\`adl\`ag local martingales, 
\[
\bE\left(\sup_{s\leq \tau}M_s^2\right) \leq \bE [M]_{\tau},
\]
for any stopping time $\tau$, i.e. the process $M_t^*=\sup_{s\leq t}M_s^2$ is $L$-dominated by $[M]$ (in the sense of Definition I.3.29 of \cite{JS}. By Lenglart's inequality (Lemma I.3.30 of \cite{JS}),
\[
\bP(\sup_{s\leq t}|M_s|>\e)\leq \frac{\eta}{\e^2}+\bP\left( \int_0^t \int_{\bR^d} \int_{\bR_0} H^2(s,x,z)N(dz,dx,dz)>\eta\right),
\]
for any $\e>0$ and $\eta>0$.

\section{An application of Fubini's theorem}
\label{app-Fubini}

In this section, we include an application of Fubini's theorem which is used frequently in the article.

\begin{lemma}
\label{Fubini-lemma}
Let $X$ and $Y$ be independent random variables with values in measurable spaces $(E,\cE)$, respectively $(F,\cF)$, and $f:E \times F \to [0,\infty]$ be a measurable function. Let $\bP_X$ be the law of $X$. If $f(x,Y)<\infty$ a.s. for $\bP_{X}$-almost all $x \in E$, then $f(X,Y)<\infty$ a.s. In particular, if $\bE[f(x,Y)]<\infty$ for $\bP_{X}$-almost all $x \in E$, then $f(X,Y)<\infty$ a.s.
\end{lemma}

\begin{proof}
We know that $\bP(f(x,Y)<\infty)=1$ for all $x \in N^c$, where $\bP_{X}(N)=0$. By Fubini's theorem,
\begin{align*}
\bP(f(X,Y)=\infty)&=\int_{E}\int_{F}1_{\{f(x,y)=\infty\}}\bP_{X}(dx) \bP_{Y}(dy)=\int_{N^c} \left( \int_{F} 1_{\{f(x,y)=\infty\}} \bP_{Y}(dy) \right) \bP_{X}(dx) \\
&=\int_{N^c}\bP(f(x,Y)=\infty)\bP_{X}(dx)=0.
\end{align*}
\end{proof}

\begin{remark}
\label{Fubini-remark}
{\rm Note that $h(x)=\bE[f(x,Y)]<\infty$ for $\bP_{X}$-almost all $x \in E$ is equivalent to $h(X)<\infty$ a.s. On the other hand, $h(X)=\bE[f(X,Y)|X]$ a.s., since $X$ and $Y$ are independent. So the criterion given by Lemma \ref{Fubini-lemma} can be stated as follows: for independent random variables $X$ and $Y$,
\[
\mbox{if $\bE[f(X,Y)|X]<\infty$ a.s., then $f(X,Y)<\infty$ a.s.}
\]
Here we use a generalized definition of the conditional expectation $\bE[Z|\mathcal{G}]$ of a random variable $Z$ given a $\sigma$-field $\cG$, for which $Z$ does not have to be integrable. }
\end{remark}


\section{Some integrals of the heat kernel}
\label{app-heat}

In this section, we include some results which are used in the proof of
 Proposition \ref{prop-A2-heat} for the verification of condition \eqref{A3}. In particular, we need the explicit form of all constants.

\begin{lemma}
\label{lemC1}
In the case of the heat equation, for any $p>0$, $t>0$ and $x\in \bR^d$,
\[
\int_{\bR^d}G_t^p(x-y)dy=\overline{K}_{p,d} t^{\frac{d(1-p)}{2}} \quad \mbox{with}
 \quad \overline{K}_{p,d}=(2\pi)^{\frac{d(1-p)}{2}}p^{-d/2}.
\]
Consequently, if $p<1+\frac{2}{d}$, then
\[
\int_0^t \int_{\bR^d} G_{t-s}^p(x-y)dyds=K_{p,d}t^{\frac{d(1-p)}{2}+1} 
\quad \mbox{with} \quad K_{p,d}=\frac{\overline{K}_{p,d}}{\frac{d(1-p)}{2}+1}.
\]
\end{lemma}

\begin{proof}
This follows by direct calculation, using relation \eqref{Gp-heat}.
\end{proof}

\begin{lemma}
\label{lemC2}
In the case of the heat equation, for any $\gamma>0$, $p>0$, $t>0$ and $x \in \bR^d$,
$$\int_{\bR^d}G_t^px-y)|y|^{\gamma}dy \leq \overline{C}_{\gamma,p,d}\, t^{\frac{d(1-p)}{2}} (|x|^{\gamma}+t^{\gamma/2}),$$
where
\[
\overline{C}_{\gamma,p,d}=\overline{K}_{p,d}  (2^{\gamma-1} \vee 1)  (1\wedge p)^{-\gamma/2} \left[1+\frac{2^{
\gamma/2}}{\Gamma(d/2)}\Gamma\left(\frac{\gamma+d}{2}\right) \right].
\]
Consequently, if $p<1+\frac{2}{d}$, then
\[
\int_0^t \int_{\bR^d} G_{t-s}^p(x-y) |y|^{\gamma}dyds \leq C_{\gamma,p,d}' \, t^{\frac{d(1-p)}{2}+1} (|x|^{\gamma}+t^{\gamma/2}),
\]
where
\begin{equation}
\label{def-c}
C_{\gamma,p,d}'=\frac{\overline{C}_{\gamma,p,d}}{\frac{d(1-p)}{2}+1}=K_{p,d}  (2^{\gamma-1} \vee 1)  (1\wedge p)^{-\gamma/2} \left[1+\frac{2^{
\gamma/2}}{\Gamma(d/2)}\Gamma\left(\frac{\gamma+d}{2}\right) \right].
\end{equation}

\end{lemma}

\begin{proof} Let  $X$ be a random vector with a $N_d(0,(t/p)I_d)$ distribution. By \eqref{Gp-heat}, we have:
\begin{align*}
\int_{\bR^d}G_t^p(x-y)|y|^{\gamma}dy&= \int_{\bR^d}G_t^p(y)|x-y|^{\gamma}dy=\overline{K}_{p,d}\,t^{\frac{d(1-p)}{2}}
\int_{\bR^d}G_{t/p}(y)|x-y|^{\gamma}dy\\
& = \overline{K}_{p,d}\,t^{\frac{d(1-p)}{2}} \bE|x-X|^{\gamma} 
\leq \overline{K}_{p,d}\, t^{(1-p)d/2} (2^{\gamma-1} \vee 1)(|x|^{\gamma}+\bE|X|^{\gamma}).
\end{align*}
Let
$Z=X/\sqrt{t/p}$. Then $Z$ a $N_d(0,I_{d})$ distribution, and $\bE|X|^{\gamma}=(t/p)^{\gamma/2}z_{\gamma} $, where
\[
z_{\gamma}:=\bE|Z|^{\gamma}=\frac{2^{
\gamma/2}}{\Gamma(d/2)}\Gamma\Big(\frac{\gamma+d}{2}\Big).
\]
Hence,
\[
|x|^{\gamma}+\bE|X|^{\gamma}\leq (1+z_{\gamma})(|x|^{\gamma}+p^{-\gamma/2}t^{\gamma/2})\leq (1+z_{\gamma})(1\wedge p)^{-\gamma/2}(|x|^{\gamma}+t^{\gamma/2}).
\]
This proves the first statement. The second statement follows by direct calculation.
\end{proof}

\small{
{\bf Acknowledgement.} The first author would like to thank Robert Dalang and Gennady Samorodnitsky for the invitation to visit EPFL, respectively Cornell University, their warm hospitality, and the fruitful discussions about the 
problem considered in this article.
}

\normalsize{

}

\end{document}